\PassOptionsToPackage{table}{xcolor}
\documentclass[oneside]{amsart}

\usepackage{blindtext}
\usepackage[left=3cm,right=3cm]{geometry}
\usepackage{latexsym,amsmath,amssymb,amsfonts,amscd,graphics}
\usepackage[mathscr]{eucal}
\usepackage{mathrsfs}
\usepackage{tikz-cd}
\usepackage[ruled, vlined, linesnumbered, commentsnumbered]{algorithm2e}
\usepackage{hyperref}
\usepackage{cleveref}
\usepackage{setspace}
\SetKwInput{Input}{Input} 
\SetKwInput{Output}{Output}
\usepackage{thmtools}
\hypersetup{
	colorlinks=true,
	linkcolor=blue,
	filecolor=blue,      
	urlcolor=cyan,
	citecolor=magenta}
\usepackage{makecell}
\usepackage{longtable}
\usepackage{hhline}
\usepackage{multirow}
\usepackage{colortbl}
\usepackage{caption} 
\numberwithin{equation}{section}
\theoremstyle{plain}
\newtheorem{theorem}[equation]{Theorem}

\newtheorem{lemma}[equation]{Lemma}
\newtheorem{proposition}[equation]{Proposition}
\newtheorem{corollary}[equation]{Corollary}

\theoremstyle{remark}
\newtheorem{remark}[equation]{Remark}

\Crefname{algocf}{Algorithm}{Algorithms}
\Crefname{remark}{Remark}{Remarks}
\crefalias{AlgoLine}{line}
\Crefname{conjecture}{Conjecture}{Conjectures}
\theoremstyle{definition}
\newtheorem{definition}[equation]{Definition}

\newtheorem{example}[equation]{Example}

\newcommand{\bP}{\mathbb{P}}

\newcommand{\bR}{\mathbb{R}}

\newcommand{\bZ}{\mathbb{Z}}

\newcommand{\bC}{\mathbb{C}}

\newcommand{\bL}{\mathbb{L}}
\newcommand{\bB}{\mathbb{B}}

\newcommand{\calC}{\mathcal{C}}

\newcommand{\calK}{\mathcal{K}}

\newcommand{\calO}{\mathcal{O}}

\newcommand{\calP}{\mathcal{P}}

\newcommand{\calB}{\mathcal{B}}

\newcommand{\calD}{\mathcal{D}}
\newcommand{\calJ}{\mathcal{J}}

\newcommand{\Aut}{\mathrm{Aut}}
\newcommand{\Bir}{\mathrm{Bir}}
\newcommand{\Stab}{\mathrm{Stab}}

\newcommand{\GL}{\mathrm{GL}}

\newcommand{\Sat}{\mathrm{Sat}}

\newcommand{\im}{\mathrm{Im}}

\newcommand{\rank}{\mathrm{rank}}

\newcommand{\Pic}{\mathrm{Pic}}

\newcommand{\NS}{\mathrm{NS}}

\newcommand{\git}{/\kern-0.2em/}

\newcommand{\II}{\textnormal{II}}
\newcommand{\OG}{\Lambda}
\newcommand{\markman}{\mathcal{W}^{pex}}

\newcommand{\leech}{\mathbb{L}}

\newcommand{\sat}{\textnormal{Sat}}

\DeclareMathOperator{\divi}{div}
\DeclareMathOperator{\id}{id}

\author{Lisa Marquand and Stevell Muller}
\title[]{Finite groups of symplectic birational transformations of IHS manifolds of \boldmath $OG10$ type} 
\date{\today}
\thanks{The second author was supported by the Deutsche Forschungsgemeinschaft (DFG, German Research Foundation) [Gef\"ordert durch die Deutsche Forschungsgemeinschaft (DFG) –
Projektnummer 286237555 – TRR 195].}
\begin{document}
{\theoremstyle{plain}
\newtheorem*{theorem*}{Theorem}
}
	\bibliographystyle{halpha}

\begin{abstract}
    We classify finite groups that act faithfully by symplectic birational transformations on an irreducible holomorphic symplectic (IHS) manifold of $OG10$ type. In particular, if $X$ is an IHS manifold of $OG10$ type and $G$ a finite subgroup of symplectic birational transformations of $X$, then the action of $G$ on $H^2(X,\bZ)$ is conjugate to a subgroup of one of 375 groups of isometries. We prove a criterion for when such a group is determined by a group of automorphisms acting on a cubic fourfold, and apply it to our classification. Our proof is computer aided and our results are available in a Zenodo dataset \cite{database}.
\end{abstract}
 \maketitle
\section{Introduction}
In recent years, there has been much progress in the study of symplectic automorphisms of irreducible holomorphic symplectic (IHS) manifolds.
The model for this study is that of $K3$ surfaces; Mukai classifies symplectic automorphisms in his celebrated paper \cite{MR958597}.
This was later reproved by Xiao \cite{xiao}, and a more streamlined approach was reached by Kond\=o \cite{MR1620514} using automorphisms of the Niemeier lattices.
A full classification has been later provided by Hashimoto \cite{hashimoto} in terms of cohomological action. 
The techniques of Kond\=o were adapted to the classification of symplectic automorphisms of cubic fourfolds by Laza and Zheng in \cite{laza2019automorphisms}; here an automorphism is symplectic if it induces a symplectic automorphism on the Fano variety of lines. 
In a similar vein, Mongardi obtained a classification of prime order symplectic automorphisms of IHS manifolds of $K3^{[n]}$ type \cite{MR3102529, MR3473636}, with similar results obtained by \cite{huybrechtsderived}. 
H\"ohn and Mason \cite{HMsympk32} completed the full classification of finite regular symplectic automorphisms for IHS manifolds of $K3^{[2]}$ type. In the case of O'Grady's exceptional examples, the authors in \cite{grossi2020finite,giovenzana2023symplectic} show that for IHS manifolds of deformation type $OG6$ and $OG10$, any symplectic automorphism of finite order acts trivially on the second integral cohomology. In particular, IHS manifolds of $OG10$ type do not admit any nontrivial symplectic automorphism of finite order.

For the two previous deformation types, the situation becomes richer if instead one considers symplectic \emph{birational} transformations. The isometry classes of the invariant and coinvariant sublattices associated to symplectic birational transformations of finite order for IHS manifolds of $OG6$ type are classified in \cite{grossi2020finite}. A lattice classification of symplectic birational involutions for IHS manifolds of $OG10$ type was reached by the authors of the present paper in \cite{marqOG10}. Among them, three of the possible cases can be geometrically realised by involutions of a cubic fourfold via the construction of \cite{LSV,sac2021birational} (see also \cite{marquand2022cubic}). In this paper, we study finite groups of symplectic birational transformations of IHS manifolds of $OG10$ type, and obtain a classification of (saturated) groups $G$.

For an IHS manifold $X$, we denote by $\Bir_s(X)$ the group of symplectic birational transformations of $X$. Let $(X,\eta, G)$ be a marked IHS manifold of $OG10$ type, and $G\leq \Bir_s(X)$ a finite subgroup. 
 For such a triple, we obtain induced isometries of the second integral cohomology, i.e. $\eta_*(G)\leq O(\Lambda)$, where $\eta:H^2(X,\bZ)\cong \Lambda := U^3\oplus E_8^2\oplus A_2$ is an isometry. 
We can now state our main classification result.

\begin{theorem*}[\Cref{thm: main thm for OG10 intro}]\label{intro: thm1}
    Let $(X, \eta, G)$ be a triple consisting of a marked IHS manifold $(X, \eta)$ of $OG10$ type and a finite group $G\leq\Bir_s(X)$. Then, up to conjugacy, $\eta_\ast(G)\leq O(\Lambda)$ is contained in one of the 375 saturated groups of the dataset \cite{database}.
\end{theorem*}

Moreover, Table 4 in the ancillary files gives information about each such conjugacy class. 
The strategy to obtain the classification in \Cref{thm: main thm for OG10 intro} follows three main steps. 

Firstly, we use the Global Torelli theorem to reduce the classification of groups of symplectic birational transformations of IHS manifolds of $OG10$ type to classifying subgroups $H\leq O^+(\Lambda)$ satisfying certain lattice-theoretic conditions. We recall that $O^+(\Lambda)\trianglelefteq O(\Lambda)$ denotes the normal subgroup consisting of orientation-preserving isometries. We classify such groups $H\leq O^+(\Lambda)$ by distinguishing their induced action on the discriminant group $D_\Lambda:=\Lambda^\vee/\Lambda.$ 
Note that $O(D_\Lambda)= \{\pm \id_{D_\Lambda}\}$, and let us denote by $O^\#(\Lambda)$ the kernel of the natural morphism $O(\Lambda)\to O(D_\Lambda)$. We let moreover $O^{+, \#}(\Lambda) := O^+(\Lambda)\cap O^\#(\Lambda)$.

The second step of our proof of \Cref{thm: main thm for OG10 intro} consists of classifying finite subgroups $H\leq O^{+,\#}(\Lambda)$; we say such a group $H$ is \textbf{stable}\footnote{We borrow this definition from \cite{ghs13}; the naming ``stable" comes from the fact that acting trivially on the discriminant group is stable by extending with the identity along primitive embeddings of lattices (compare with \Cref{lem: stably sat} for instance)}. 
A classical approach to classify stable groups is to realise $H$ as an isometry group of the Leech lattice $\bL$, by embedding the coinvariant sublattice primitively into $\bL$ and extending the group of isometries. 
Unfortunately, this is no longer a sufficient strategy in the case of $OG10$ type manifolds $X$, due to the fact that $ H^2(X,\bZ)\cong \Lambda$ is not unimodular. 
Indeed, there exist stable groups $H\leq O^{+, \#}(\Lambda)$ whose coinvariant sublattice does not primitively embed into $\bL$ (see \Cref{counterexample}).
Instead, we realise such stable groups $H$ as subgroups of isometries of the Borcherds lattice $\bB:=U\oplus \bL$ to obtain our classification.
We obtain the following result:
\begin{theorem*}[\Cref{theo: sec 3 main}]
    Let $(X, \eta, G)$ be a triple consisting of a marked IHS manifold $(X, \eta)$ of $OG10$ type and a finite group $G\leq\Bir_s(X)$ such that $H:=\eta_\ast(G)\leq O^{+, \#}(\Lambda)$ is stable. Then, up to conjugacy, $H\leq O^{+,\#}(\Lambda)$ is contained in one of the 192 stable groups of the dataset \cite{database}.
\end{theorem*}

Lastly, the final step in the proof of \Cref{thm: main thm for OG10 intro} is to classify the possible groups $H\leq O^+(\Lambda)$ with nontrivial action on $D_\Lambda$. For such a group, there is a short exact sequence
\[1\rightarrow H^\#\rightarrow H\rightarrow O(D_\Lambda)\cong \mu_2\rightarrow 1,\] where $H^\#$ is stable. Our approach to classify such possible groups $H$ is to start with one of the 192 stable groups $H^\#$ mentioned in \Cref{theo: sec 3 main}, and classify all possible $\mu_2$-extensions, along with their action on $\Lambda$. We then apply the Global Torelli theorem to extract those extensions that appear as groups of symplectic birational transformations of an IHS manifold of $OG10$ type.

Our strategy to perform this extension classification is adapted from the original techniques of Brandhorst and Hashimoto \cite{brandhorst_hashimoto}, who classified maximal finite groups of automorphisms of $K3$ surfaces.
In particular, such groups contain both symplectic and nonsymplectic automorphisms --- we call such a group \textbf{mixed}.  
These techniques have been revisited by the independent works \cite{comparin2023irreducible,wawak} in the case of maximal mixed actions on IHS manifolds of $K3^{[2]}$ type.
Through a series of work \cite{brandhorst_hashimoto, brandhorst_hofmann_manthe, brandhorst_veniani}, this classification approach culminated in the recent major algorithmic progress of Brandhorst and Hofmann \cite{bh22} where the authors completed a long-standing full classification work for finite mixed automorphisms on K3 surfaces.

The algorithms developed in \cite{bh22} have been successfully implemented by the second author of the present paper on Oscar \cite[QuadFormAndIsom]{oscar-book}, together with Nikulin's theory on primitive extensions \cite{nikulin}. 
This package allows one to consider the classification of finite groups of symmetries of higher dimensional IHS manifolds. We develop an algorithmic procedure to complete the extension classification and we obtain the classification of finite subgroups of symplectic birational transformations on IHS manifolds of $OG10$ type in \Cref{thm: main thm for OG10 intro}.

 A natural next step is to provide geometrical realisations for the groups $H$ occuring in the classification developed in \Cref{thm: main thm for OG10 intro}. 
 For some of the groups with trivial action on $D_\Lambda,$ this can be achieved by applying O'Grady's original construction \cite{MR1703077} to a $K3$ surface with a specific group of automorphisms. 
 In \cite{marqOG10}, we provided geometric realisations for two of the symplectic birational involutions obtained as extensions of the trivial group $H^\#=1$,  via antisymplectic involutions of a cubic fourfold $V\subset \bP^5$. 
 We briefly recall the construction.
 An antisymplectic involution $\phi\in \Aut(V)$ induces an antisymplectic birational involution $\phi$ on a compactified intermediate Jacobian $X_V$, which is an IHS manifold of $OG10$ type as in \cite{LSV,sac2021birational}. 
 The manifold $X_V$ has a Lagrangian fibration $\pi\colon X_V\to \mathbb{P}^5$ which admits a section, and it is equipped with another antisymplectic birational involution $\tau\in \Bir(X_V)$ acting by $-1$ on the smooth fibers. 
 The composition $\tau\circ\phi$ is a symplectic birational involution of $X_V$.

  Let $V$ be a cubic fourfold, and $\Aut_s(V)\leq \Aut(V)$ the subgroup of symplectic automorphisms. A group of symplectic automorphisms $G_s\leq \Aut_s(V)$ will induce a group of symplectic birational transformations $G_s\leq \Bir_s(X_V)$, acting trivially on the discriminant group. We show that if $V$ exhibits in addition an antisymplectic automorphism, i.e. a nonsymplectic automorphism with symplectic square, then one can induce larger groups of symplectic birational transformations $G\leq \Bir_s(X_V)$ that are $\mu_2$-extensions of $G_s$. More precisely, we show:
 \begin{theorem*}[\Cref{thm: geo realise LSV}]\label{intro: thm3}
    Let $X$ be an IHS manifold of $OG10$ type and let $G\leq \Bir_s(X)$ be a finite group of symplectic birational transformations. Suppose that $\Lambda^G\cong U\oplus \Gamma$ holds, for some lattice $\Gamma$. Then there exists some smooth cubic fourfold $V$ and an embedding $j\colon G\hookrightarrow \Aut(V)$ such that:
    \begin{enumerate}
        \item either $G$ acts trivially on the discriminant group $D_\Lambda$, and $j(G)\leq \Aut_s(V)$;
        \item or $G=\langle G_s, \phi\rangle$ with $j(G_s)\leq \Aut_s(V)$, and $j(\phi)\in\Aut(V)\setminus\Aut_s(V)$ is antisymplectic.
    \end{enumerate}
    The pair $(G_s, \Lambda_{G_s})$ occurs in the classification of \cite{laza2019automorphisms}.

    Conversely, for any smooth cubic fourfold $V$, any LSV manifold $X_V$ associated to $V$ and any finite subgroup $G\leq \Aut(V)$ so that $[G:G_s]\leq 2$ holds, there is an embedding of $G$ into the group $\Bir_s(X_V)$ of symplectic birational transformations on $X_V$.
\end{theorem*}

Among the 375 groups mentioned in \Cref{thm: main thm for OG10 intro}, we have that 77 of them can be realised via LSV manifolds associated to cubic fourfolds in the sense of \Cref{thm: geo realise LSV}. We also find that 34 of the pairs can be realised via the twisted LSV construction from a cubic fourfold as in \cite{twistedLSV}. 
\begin{remark}
    Note that even though we know the action of $G\leq \Aut(V)$ on $H^4(X,\bZ)_{prim}$, we cannot (always) determine the action on $H^2(X_V,\bZ)$. The manifold $X_V$ admits a Lagrangian fibration $\pi:X_V\rightarrow \bP^5$, and has a distinguished divisor $\Theta$ \cite{LSV,sac2021birational}. There is a \textbf{rational} Hodge isometry between $\langle \Theta, \pi^*\calO_{\bP^5}(1)\rangle^\perp\subset H^2(X_V,\bZ)$ and $H^4(V,\bZ)_{prim}(-1)$. The action of the induced group $G\leq \Bir_s(X_V)$ on $H^2(X_V,\bZ)$ could be determined if this isometry can be upgraded to an integral Hodge isometry. The conjectural action of the 77 groups above are indicated in Table 4 of the ancillary files. 
\end{remark}

\subsection*{Outline of the paper} In \Cref{sec: prelims}, we recall the definitions and known results surrounding lattices, isometries and primitive extensions. We also collect results regarding isometries of the Borcherds lattice. In \Cref{sec: sympl bir}, we recall the Global Torelli theorem for IHS manifolds of $OG10$ type, and reduce the classification of finite groups of symplectic birational transformations to classifying certain isometry groups of the associated lattice. 
In \Cref{sec: stable sympl iso}, we classify finite groups of symplectic birational transformations acting trivially on $D_\Lambda$, proving \Cref{theo: sec 3 main}. In order to do so, we determine the isometry class of the coinvariant sublattices for some finite subgroups of isometries of $\bB$: those are either primitively embedded into the Leech lattice, or they are displayed in \Cref{appendix: exceptional stable,App: exceptional hearts}. The results of \Cref{theo: sec 3 main} are presented in \Cref{app: table of hearts}, \Cref{tab: lovely table 45}. 
In \Cref{sec: from stable sympl to sympl}, we explain our extension approach in order to classify finite groups of symplectic birational transformations acting nontrivially on $D_\Lambda$. 
We implement the algorithms from \Cref{sec: algorithms}, and we prove \Cref{thm: main thm for OG10 intro}. 
Finally, in \Cref{sec: geo interp}, we provide geometric realisations for some of the groups in the classification via known constructions from either a $K3$ surface, or a cubic fourfold. 
In particular, we prove \Cref{thm: geo realise LSV} and apply it to our classification. Our work is computer aided: all the data supporting our proofs are contained in an external Zenodo dataset \cite{database}, as well as several notebooks explaining our computations. 

\subsection*{Acknowledgements} The authors would like to thank Simon Brandhorst for his helpful comments and for his suggestion to use the Borcherds lattice, as in \Cref{subsec: Borcherds}. Specifically, we thank him for the discussions about the proofs of \Cref{propo: first characterisation finite subgroups,theorem exceptional lattices}. The authors would like to thank Simone Billi, Annalisa Grossi, Ljudmila Kamenova, Radu Laza and Giovanni Mongardi for helpful discussions and their comments. The authors would like to thank the OSCAR team for their support regarding the programming aspect of this project. Finally, the authors would like to thank the referee, whose comments greatly improved the manuscript.

\section{Preliminaries}\label{sec: prelims}
In this section, we recall some preliminary results. 
In \Cref{subsec: notation} we recall relevant notation and definitions regarding lattices and their isometries. 
In \Cref{subsec: equivariant extension} we recall Nikulin's theory of primitive  embeddings, highlighting the equivariant analogue. 
In \Cref{subsec: Borcherds} we introduce the Borcherds lattice, and recall the structure of its group of isometries.

\subsection{Lattices and isometries}\label{subsec: notation} Let $L$ be a lattice, which is a finitely generated free $\bZ$-module equipped with a nondegenerate, integer valued symmetric bilinear form. 
For two vectors $u,v\in L$, we denote by $u.v\in \bZ$ the image by the bilinear form, and we let $v^2 := v.v$. 
We assume that $L$ is \textbf{even} unless stated otherwise, i.e. $v^2\in 2\bZ$ for all $v\in L$. Finally, we denote by $O^+(L)\leq O(L)$ the group of orientiation-preserving isometries of $L$.

In what follows, all $ADE$ root lattices are assumed to be \textbf{negative definite}. The Leech lattice $\leech$ is the unique (up to isometry) negative definite even unimodular lattice of rank 24 that contains no $(-2)$-vectors. We denote by $U$ the hyperbolic plane lattice, which is the unique even unimodular lattice of rank 2, up to isometry.

We denote by $D_L:=L^\vee/L$ the discriminant group of $L$. For any lattice isometry $f\colon L\xrightarrow{\cong} L'$, we denote by $D_f\colon D_L\xrightarrow{\cong} D_{L'}$ the induced isometry.

\begin{definition}\label{defn: stable}
    Let $L$ be a lattice and let $G \leq O(L)$ be a group of isometries.  We call the \textbf{discriminant representation} of $G$ the morphism $G\to O(D_L)$. We denote by $\overline{G}$ its image and by $G^\#$ its kernel. Any element in $G^\#$ is said to be \textbf{stable} and we call $G$ \textbf{stable} if $G = G^\#$.
\end{definition}

Let $G\leq O(L)$ be a group of isometries. 
We denote by $L^G:=\{v\in L\mid g(v)=v, \,\, \forall g\in G\}$ and by $L_G:=(L^G)_L^\perp$ the associated invariant and coinvariant sublattices, respectively. 

\begin{definition}\label{defn: saturation}
    Let $L$ be a lattice, and let $H\leq G\leq O(L)$ be a chain of subgroups.
    We call the \textbf{saturation of $H$ in $G$}, denoted $\sat_G(H)$, the largest subgroup $H\leq \sat_G(H)\leq G$ such that $L^H = L^{\sat_G(H)}$. 
    We say that $H$ is \textbf{saturated} in $G$ if $\sat_G(H) = H$.
    
\end{definition} 

\begin{lemma}\label{sat implies stab sat}
    Let $L$ be a lattice and let $H\leq G\leq  O(L)$ be a chain of subgroups.
    If $H$ is saturated in $G$, then $H^\#$ is saturated in $G^\#$.
\end{lemma}

\begin{proof}
    Let $g\in G^\#$ such that $g$ acts trivially on $L^{H^\#}$. Since $L^{H}\subseteq L^{H^\#}$, we have that $g$ is the identity on $L^H$ and thus $g\in \Sat_G(H) = H$. Hence $g\in G^\#\cap H = H^\#$, and $\Sat_{G^\#}(H^\#) = H^\#$.
   
\end{proof}

\begin{definition}\label{definition lattice with isometry}
    We call a \textbf{lattice with isometry} any pair $(L, f)$ consisting of a lattice $L$ and an isometry $f\in O(L)$. Two such pairs $(L_1, f_1)$ and $(L_2, f_2)$ are \textbf{isomorphic} if there exists an isometry $\psi\colon L_1\xrightarrow{\cong} L_2$ such that $f_2 = \psi f_1\psi^{-1}$.

\end{definition}

\begin{definition}
    Let $L$ be a lattice and let $v\in L$ be a vector. We define the \textbf{divisibility} of $v$ in $L$, which we denote $\divi_L(v)$, the positive generator of the ideal $v.L$.
\end{definition}

\subsection{Primitive extensions}\label{subsec: equivariant extension}

For nondegenerate lattices, all morphisms are injective --- we therefore talk about \textbf{embeddings}.

\begin{definition}
    Let $L$ be a lattice.
    \begin{enumerate}
        \item Any sublattice $S\subseteq L$ is called \textbf{primitive} if the quotient $L/S$ is torsion-free.
        \item Any embedding $i\colon S\hookrightarrow L$ is called \textbf{primitive} if $i(S)\subseteq L$ is primitive.
        \item Two primitive sublattices $S_1, S_2\subseteq L$ are said \textbf{isomorphic} if there exists an isometry $\psi\in O(L)$ such that $\psi(S_1) = S_2$.
    \end{enumerate}
\end{definition}

Given an even lattice $S$, the proof of \cite[Proposition 1.15.1]{nikulin} describes a procedure to classify, up to isomorphism, primitive sublattices of an even lattice, with given signature and discriminant group, that are isometric to $S$.
The proof makes use of the notion of primitive extensions, which we recall now.

\begin{definition}
    For a lattice $L$ and a sublattice $N\subseteq L$, we say that $L$ is an \textbf{overlattice} of $N$ if $L$ and $N$ have the same rank, as $\bZ$-modules.
\end{definition}

Let $S$ and $T$ be even lattices and let $L$ be an overlattice of $S\oplus T$. Such an overlattice $S\oplus T\subseteq L$ is called a \textbf{primitive extension} if both composite embeddings $S\hookrightarrow S\oplus T\hookrightarrow L$ and $T\hookrightarrow S\oplus T\hookrightarrow L$ are primitive.

\begin{definition}
    We define a \textbf{glue map} between $S$ and $T$ to be an isomorphism of finite abelian groups
    \[ D_S\geq H_S\xrightarrow{\gamma} H_T\leq D_T\]
    such that $x^2+\gamma(x)^2\in 2\mathbb{Z}$ for all $x\in H_S$. We moreover call $H_S$ and $H_T$ the \textbf{glue domains} of $\gamma$.
\end{definition}

\begin{proposition}[{{{\cite[Proposition 1.4.1]{nikulin}}}}]
    Glue maps $D_S\geq H_S\xrightarrow{\gamma} H_T\leq D_T$ correspond bijectively to even primitive extensions $S\oplus T\subseteq L_\gamma$.
\end{proposition}

\begin{definition}
    We call $L_\gamma$ the \textbf{overlattice relative to $\gamma$}. In this situation, we also say that $H_S\leq D_S$ and $H_T\leq D_T$ are the glue domains of the primitive embeddings $S\hookrightarrow L$ and $T\hookrightarrow L$ respectively. 
\end{definition}

Let $(S, s)$ and $(T, t)$ be two even lattices with isometry, where $s\in O(S)$ and $t\in O(T)$. 
Let $D_S\geq H_S\xrightarrow{\gamma} H_T\leq D_T$ be a glue map. 
The glue map $\gamma$ is called \textbf{$(s, t)$-equivariant} if $H_S$ and $H_T$ are respectively $D_s$-stable and $D_t$-stable, and if it satisfies the \textbf{equivariant gluing condition:} 
\begin{equation}\label{eq:egc}\tag{EGC}
    \gamma\circ (D_{s})_{\mid H_S} = (D_{t})_{\mid H_T}\circ \gamma.
\end{equation}

\begin{proposition}[{{{\cite[Corollary 1.5.2]{nikulin}}}}]
The map $\gamma$ is $(s, t)$-equivariant if and only if $s\oplus t$ extends along the primitive extension $S\oplus T\subseteq L_{\gamma}$ to an isometry $f_\gamma$ of $L_{\gamma}.$
\end{proposition}

We  call $(L_\gamma, f_\gamma)$ an \textbf{equivariant primitive extension} of $(S, s)$ and $(T, t)$. 

\begin{definition}
    Let $(S_1, s_1)\oplus (T_1, t_1)\subseteq (L_1, f_1)$ and $(S_2, s_2)\oplus (T_2, t_2)\subseteq (L_2, f_2)$ be two equivariant primitive extensions. They are said to be \textbf{isomorphic} if there exists an isomorphism $\psi\colon (L_1, f_1)\to (L_2, f_2)$ which restricts to isomorphisms $\psi_S\colon (S_1,s_2)\to (S_2, s_2)$ and $\psi_T\colon (T_1, t_1)\to (T_2, t_2)$.
\end{definition}

\subsection{Borcherds lattice}\label{subsec: Borcherds}
A common technique to classify finite groups of symplectic birational transformations of IHS manifolds is to relate them to isometry groups of the Leech lattice $\leech$. Unfortunately, this lattice is too small for our purposes, and so we will require some results about isometries of the \textbf{Borcherds lattice:}
$$\bB:=U\oplus \leech.$$

\begin{remark}
    In \cite{bm24}, the authors defines the notion of Borcherds lattices, which are a particular class of hyperbolic even lattices with infinitely many simple $(-2)$-roots. They show that the Borcherds lattices of largest rank are isometric to $U\oplus \leech$, which motivates our definition. Note that our definition differs from the one given by Laza and Zheng in \cite{laza2019automorphisms}, where they define $U^2\oplus \bL$ as Borcherds lattice, which is not a Borcherds lattice in the sense of Brandhorst and Mezzedimi.
\end{remark}

The lattice $\mathbb{B}$ is even unimodular, of signature $(1,25)$, and it is unique in its genus.
Let us fix a basis $\{e,s\}$ for $U\subset \bB$ with $e^2=0, s^2=-2$ and $e. s=1$.

The group of isometries of $\mathbb{B}$ is known and it has been studied for instance by Conway in \cite[Chapter 27]{splg}. For the reader's convenience, we recall Conway's results, following an exposition of Brandhorst and Mezzedimi \cite{bm24}.

We denote by $\mathcal{P}$ one of the connected components of \[\left\{ x\in \bB\otimes \mathbb{R}\,\mid\, x^2 > 0\right\};\] we call it the \textbf{positive cone} of $\bB$. Note that $O(\bB) = \{\pm \id\}\times O^+(\bB)$ and $-\id$ does not preserve $\mathcal{P}$ --- we can see that the group $O^+(\bB)$ as the stabiliser of $\mathcal{P}$ in $O(\bB)$. 

Let us denote by $\Delta := \{r\in \bB\;\mid \; r^2=-2\}$ the set of \textbf{roots} of $\bB$
and denote by $W(\bB)\leq O^+(\bB)$ the subgroup generated by the reflections in the roots $r\in \Delta$. The group $W(\bB)$ is the so-called \textbf{Weyl group} of $\bB$, and it acts simply transitively on the set of connected components, or \textbf{chambers}, of
\[\Gamma := \mathcal{P}\setminus \bigcup_{r\in\Delta}r^\perp.\]
For any chamber $D\subseteq \Gamma$, we let $P_D := \{v\in\bB\;\mid\; v.x>0,\,\forall x\in \overline{D}\}$ and moreover, we define
\[ \Delta_D:= \{ r\in P_D\cap \Delta\;\mid\; r-r'\notin P_D\cap \Delta,\, \forall r'\in P_D\cap\Delta\}\]
the set of \textbf{simple roots of $D$} \cite[\S2.6]{bm24}.
\begin{lemma}
    There exists a chamber $D_0\subset \Gamma$ such that $e\in \overline{D_0}\cap \bB$ and the set $\{|e.r|\;\mid\; r\in \Delta_{D_0}\}$ is bounded. Moreover, the vector $e$ is the unique isotropic vector in $\overline{D_0}\cap \bB$ satisfying this property, and $e.r=1$ for all $r\in\Delta_{D_0}$
\end{lemma}
\begin{proof}
    Existence of an isotropic element with the required properties follows from \cite[Chapter 27]{splg}. Uniqueness follows from \cite[Theorem 3.7, Remark 3.8]{bm24}.
\end{proof}

For the rest of the paper, we refer to $e$ and $D_0$ as \emph{Conway's vector} and \emph{Conway's chamber}, respectively. We moreover denote by $\calD := \overline{D_0}$ the closure of $D_0$ in $\calP$, and we let $\Aut(\mathcal{D})$ be the subgroup of isometries of $O^+(\bB)$ preserving $\mathcal{D}$. Since $\mathcal{D}$ is the closure of a fundamental domain for the action of $W(\bB)$ on $\calP$, we can identify $\Aut(\mathcal{D})$ with $O^+(\bB)/W(\bB)$ \cite[\S2.6]{bm24}.

\begin{lemma}[{{{\cite[Theorem 3.7]{bm24}}}}]\label{first properties autD}
    The group $\Aut(\calD)$ is infinite and any element $h\in \Aut(\calD)$ fixes $e$.
\end{lemma}
We would like to study finite subgroups of $\Aut(\calD)$ in order to determine their coinvariant sublattices. For this, we will need to understand the structure of $\Aut(\mathcal{D})$.

\begin{definition}[Eichler--Siegel transformation]\label{eichlersiegel}
    For any $\lambda\in \bL$, we define \[\psi_\lambda\colon \bB\to \bB,\; x\mapsto x+ (x.\lambda)e - (x.e)\lambda -\frac{1}{2}(x.e)\lambda^2e.\]
\end{definition}
Note that $\psi_{\lambda}\in \Aut(\calD)$ by the proofs of \cite[Proposition 3.2, Theorem 4.7]{bm24}.
Further, the assignment
\[\psi\colon \bL\to \Aut(\mathcal{D}),\; \lambda\mapsto \psi_\lambda\]
is an injective group homomorphism (where we view $\bL$ as a torsion free abelian group of finite rank under addition). By the definition of Conway's vector $e$, which is isotropic, there is an exact sequence of lattices
\[0\to \bZ e\to e^\perp\to \bL\to 0\]
inducing an isometry $\kappa\colon e^\perp/\bZ e\xrightarrow{\cong}\bL$.

\begin{lemma}[{{{\cite[Chapter 27]{splg}}}}]\label{iso conway}
    We have that $\Aut(\calD)=\bL\rtimes O(\bL)$. 
\end{lemma}
\begin{proof}
    Since any isometry in $\Aut(\mathcal{D})$ fixes $e$ (\Cref{first properties autD}), the isometry $\kappa$ defines an orthogonal representation
\begin{equation}\label{defin pi}
    \pi\colon \Aut(\mathcal{D})\to O(\bL)
\end{equation}
which admits a section $\phi:O(\bL)\rightarrow \Aut(\calD)$, given by extending an isometry of $\bL$ to one of $\bB$ acting as the identity on $\bZ e+\bZ s\cong U.$
    According to \cite[Proposition 3.2]{bm24}, the following sequence is exact
\[ 0\to \bL\xrightarrow{\psi}\Aut(\mathcal{D})\xrightarrow{\pi} O(\bL)\to 1.\]  Since all nontrivial elements of $\bL$ have infinite order and $O(\bL)$ is of finite order, we have that $\bL\cap O(\bL)$ is trivial, as a subgroup of $\Aut(\mathcal{D})$, giving the claim.
\end{proof}

Hence, any element $h\in\Aut(\mathcal{D})$ can be uniquely written in the form $\psi_\lambda\circ \phi(g)$ for some $\lambda\in \bL$ and some $g\in O(\bL)$: we write
\[ h = (\lambda, g).\]
We record some properties of elements $h\in \Aut(\calD)$ for future use. The proof follows by direct calculation.
\begin{proposition}\label{propo: first comput}
    Let $\lambda\in \bL$ and let $g\in O(\bL)$. The following hold:
    \begin{enumerate}
        \item $\phi(g)\circ \psi_\lambda = \psi_{g(\lambda)}\circ \phi(g)$;
        \item $\psi^{-1}_\lambda = \psi_{-\lambda}$;
        \item $h := (\lambda, g)$ is of finite order if and only if $\lambda\in \bL_g$.
    \end{enumerate}
\end{proposition}

\begin{remark}\label{rem orders}
    A consequence of \Cref{propo: first comput} (3) is that for any $h := (\lambda, g)\in \Aut(\mathcal{D})$ of finite order, then the order of $h$ is the same as the order of $g$.
\end{remark}

\begin{theorem}\label{propo: first characterisation finite subgroups}
    Let $H\leq \Aut(\mathcal{D})$ be a subgroup, and let $G := \pi(H)\leq O(\bL)$ (\Cref{defin pi}). Then $H$ is of finite order if and only if there exists $n \in \bZ$ positive and $v\in \bL$ such that for all $h = (\lambda, g)\in H$, 
    \[ g(v) - v = n\lambda.\]
    Moreover, if $H$ is of finite order then $H\cong G$, and $n$ can be chosen to be $\#(H\cdot s)$.
\end{theorem}

\begin{proof}
    First remark that if $H$ is of finite order, then $H\cap \bL$ is the trivial subgroup of $\Aut(\mathcal{D})$. In particular, $\pi$ restricts to an isomorphism between $H$ and $G := \pi(H)$. Note moreover that for any $h = H\leq \Aut(\calD)$, we have $h(e) = e$ (\Cref{first properties autD})

    Suppose that $H$ is of finite order, and denote by $n:= \#(H\cdot s)$ the length of the orbit of $s$ under $H$. Since $H$ is finite, we have that $w := \sum_{r\in H\cdot s}r$ is fixed by $H$, and moreover $e.w = n(e.s) = n$. This implies that there exists $m\in \bZ$ and $v\in \bL$ such that
    \[w = me+ns+v.\]
    Since $H$ fixes $e$ and $w$, we have that $H$ fixes $ns+v$.
   In particular, for all $h = (\lambda, g)\in H$, we have
    \begin{equation}\label{eq: fix vector}
        0 = h(ns+v) - (ns+v) 
    = (g(v)-v-n\lambda) + \left(\lambda.g(v)-n\frac{\lambda^2}{2}\right)e \in \bL\oplus \bZ e.
    \end{equation} 
    This implies that $g(v)-v = n\lambda$, and since this holds for any $h = (\lambda, g)\in H$, we can conclude.

    Conversely, suppose that such $n>0$ and $v\in \bL$ exist, and let $h = (\lambda, g)\in H$ be arbitrary. Note that we have that
    \[n\lambda.(g(v)+v) = (g(v)-v).(g(v)+v) = 0.\]
    In particular, since $g(v)-v=n\lambda$, we observe that
    \[2g(v).\lambda = (g(v)-v+g(v)+v).\lambda = (g(v)-v).\lambda = n\lambda^2.\]
    According to \Cref{eq: fix vector}, we obtain that $ns+v\in \bB^H$. Since $e\in \bB^H$ too, one can find $m\in\bZ_{\geq0}$ large enough such that $me+ns+v\in \bB^H$ has positive norm. Since the lattice $\bB$ is hyperbolic, we deduce that $\bB_H$ is negative definite and the group $H$ acting faithfully on such a lattice must be finite. 
\end{proof}

\begin{remark}\label{rmk enough take order of group}
    Let $H\leq \Aut(\calD)$ be of finite order. Then the pair $(n, v)\in \bZ_{>0}\times \bL$ as in the statement of \Cref{propo: first characterisation finite subgroups} is not unique: in fact, one can rescale simultaneously $n$ and $v$ by any nonzero integer, and $v$ can be replaced by any element of $v+\Lambda^G$. In particular, we can always assume that $n=\#H$.
\end{remark}

\begin{remark}
    Let $H\leq \Aut(\calD)$ be finite, and let $G := \pi(H)\leq O(\bL)$ where $\pi\colon \Aut(\calD)\to O(\bL)$ is the representation defined by the isometry $\kappa\colon e^\perp/\bZ e\xrightarrow{\cong} \bL$. Since $H$ fixes $e$, the lattice $\bB_H\subseteq e^\perp$ and it does not contain $\bZ e$: we obtain therefore that $\bB_H$ embeds into $\bL_G$, and the two lattices have the same rank. However, this embedding is not necessarily primitive and, $\bB_H$ and $\bL_G$ are not always isometric.
\end{remark}

\begin{proposition}\label{most examples are Leech}
    Let $H\leq \Aut(\mathcal{D})$ be a finite subgroup and let $G := \pi(H)\cong H$. Then $\bB_H$ and $\bL_G$ are isometric if and only if there exists $v\in \bL$ such that $g(v)-v=\lambda$ for all $(\lambda, g)\in H$. If the previous does not hold, then there exists $n>1$ such that
    \[\det(D_{\bB_H}) = n^2\det(D_{\bL_G})\]
    holds.
\end{proposition}

\begin{proof}
   According to the proof of \Cref{propo: first characterisation finite subgroups}, we know that there exists $n>0$ and $v\in \bL$ such that
   \[ \bB^H = (\bZ e\oplus \bL^G)+\bZ(ns+v)\]
   and $g(v)-v = n\lambda$ for all $(\lambda, g)\in H$. 
   We define $H_v := \psi_v^{-1}\phi(G)\psi_v$: by definition of $\phi\colon O(\bL)\to \Aut(\calD)$ and the fact that $g(v)-v=n\lambda$ for all $(\lambda, g)\in H$, we see that elements of $H_v$ are of the form $(n\lambda, g)$ where $(\lambda,g)\in H$. We therefore already note that if $n=1$, i.e. $g(v)-v=\lambda$ for all $(\lambda,g)\in H$, then $H = H_v$ is conjugate to $\phi(G)$ in $\Aut(\calD)$ and thus
  \[ \bB_H = \bB_{H_v} \cong \bB_{\phi(G)} = \bL_G.\]
  Furthermore, by direct computations we infer that \(\bB^{H_v} = (\bZ e\oplus \bL^G)+\bZ(s+v).\) It hence follows that
   \[\det(D_{\bB^H}) = n^2\det(D_{\bB^{H_v}})\]
   holds. Since $\bB$ is unimodular, and $\bB_{H_v} \cong \bL_G$, we also obtain that 
    \[\det(D_{\bB_H}) = n^2\det(D_{\bL_G})\]
    holds too. From that, it is clear that if $\bB_H\cong \bL_G$, then $n=1$ and $g(v)-v=\lambda$ for all $(\lambda,g)\in H$.\qedhere

\end{proof}
\section{Symplectic birational transformations of IHS manifolds of \texorpdfstring{$OG10$}{OG10} type}\label{sec: sympl bir}
Our main aim is to classify finite groups of symplectic birational transformations of IHS manifolds of $OG10$ type. We recall the definition in \Cref{subsec: IHS manifolds}, along with the period map for IHS manifolds of  $OG10$ type. In \Cref{subsec: bir transf} we discuss birational transformations of such manifolds. Finally in \Cref{subsec: classification}, we explain our strategy for completing this classification --- namely, we relate the classification of finite groups of symplectic birational transformations to classifying certain isometry groups of an associated lattice.

\subsection{IHS manifolds of \texorpdfstring{$OG10$}{OG10} type}\label{subsec: IHS manifolds}
	An IHS manifold is a simply connected, compact, K\"ahler manifold $X$ such that $H^0(X,\Omega^2_X)$ is generated by a nowhere degenerate holomorphic $2$-form $\sigma_X$.
 In this paper, we are focused on IHS manifolds that are deformation equivalent to O'Grady's 10-dimensional exceptional example \cite{MR1703077}. 
 Such a manifold $X$ is said to be of $OG10$ type. 

 It is well known that $H^2(X,\bZ)$ admits a quadratic form $q_X$, known as the Beauville--Bogomolov--Fujiki form. The form $q_X$ is integral, nondegenerate, and the isometry class of the lattice $(H^2(X, \bZ), q_X)$ is invariant under deformation. By \cite{MR2349768}, for an IHS manifold of $OG10$ type, there is an isometry 
 \[\eta: (H^2(X,\bZ), q_X)\xrightarrow{\cong}\OG := U^3\oplus E_8^2\oplus A_2.\]
 A choice of such an isometry $\eta$ is called a \textbf{marking}, and we say that $(X, \eta)$ is a \textbf{marked IHS manifold}.
 Two such marked IHS manifolds $(X, \eta)$ and $(X', \eta')$ are called \textbf{equivalent} if there exists an isomorphism $f\colon X\xrightarrow{\cong} X'$ such that $\eta' = \eta\circ f^\ast$.

 Marked IHS manifolds of $OG10$ type are classified up to equivalence in a coarse moduli space which we denote $\mathcal{M}_{OG10}$. This space is neither Hausdorff nor connected, and two inseparable points in a same connected component define birational IHS manifolds of $OG10$ type \cite[Theorem 4.3]{huybrechts1997compact}. The \textbf{period map}
 \[\begin{array}{ccccc}
 \mathcal{P}&\colon&  \mathcal{M}_{OG10}&\to&\Omega_{OG10} := \left\{\mathbb{C}\omega\in\mathbb{P}(\OG\otimes\mathbb{C})\;\mid\;\omega^2 = 0, \; \omega. \overline{\omega} > 0\right\}\\
 &&(X, \eta) &\mapsto&[\eta_{\mathbb{C}}(\sigma_X)]
 \end{array}\]
  relates a marked IHS manifold with its Hodge structure, where $\eta_{\mathbb{C}}$ denotes $\eta$ extended over $\mathbb{C}$. We have the following crucial result:
  \begin{theorem}[{{{\cite[Theorem 8.1]{huybrechts1997compact}}}}]\label{lem:surj period map}
      The period map is a local homeomorphism, which is surjective onto $\Omega_{OG10}$ when restricted to any connected component of $\mathcal{M}_{OG10}$.
  \end{theorem}

 \subsection{Global Torelli}\label{subsec: bir transf}
 Let $X$ be an IHS manifold of $OG10$ type. We denote by $\Aut(X) \leq \Bir(X)$ the groups of automorphisms and birational transformations of $X$ respectively. A birational transformation $f\in \Bir(X)$ is well defined in codimension one, and so we obtain a Hodge isometry $f^*:H^2(X,\bZ)\rightarrow H^2(X,\bZ)$. 
\begin{definition}\label{symplectic}
	A birational transformation $f\in \Bir(X)$ is said to be \textbf{symplectic} if the induced action $f^*:H^2(X,\bC)\rightarrow H^2(X,\bC)$ acts trivially on $\sigma_X$. Otherwise, $f$ is said to be \textbf{nonsymplectic}. If $f$ is nonsymplectic with symplectic square, we say that $f$ is \textbf{antisymplectic}.
\end{definition}
Let $(X, \eta)$ be a marked IHS manifold of $OG10$ type. The marking $\eta$ gives rise to an orthogonal representation
\[ \eta_{\ast}\colon \text{Bir}(X)\to O(\OG), \,\,\,\,\, f\mapsto \eta (f^\ast)^{-1}\eta^{-1}\]
which is faithful \cite[Theorem 3.1]{mw17}. By \cite[Theorem 5.4]{onorati2021monodromy}, the image of $\eta_{\ast}$ lies in $O^+(\OG)$, the group of orientation-preserving isometries.

\begin{definition}We define the following:
    \begin{enumerate}
        \item A subgroup $H\leq O^+(\OG)$ is called \textbf{symplectic} if there exists a marked IHS manifold of $OG10$ type $(X, \eta)$ and a subgroup $G\leq \Bir_s(X)$ such that $\eta_{\ast}(G) = H$.
        \item Given a marked IHS manifold of $OG10$ type $(X, \eta)$, we call a subgroup $G\leq \Bir(X)$ \textbf{stable} or \textbf{saturated} if so is $\eta_{\ast}(G)\leq O^+(\OG)$ (see Definitions \ref{defn: stable}, \ref{defn: saturation}). 
    \end{enumerate}
\end{definition}
  \begin{remark}
      It is known that IHS manifolds of $OG10$ type admit no nontrivial symplectic automorphisms of finite order \cite{giovenzana2023symplectic}.
  \end{remark}

The aim of this paper is to classify finite groups of symplectic birational transformations of IHS manifolds of $OG10$ type. Using the surjectivity of the period map $\mathcal{P}$ and the injectivity of $\eta_{\ast}$ for any marked IHS manifolds of $OG10$ type $(X, \eta)$, our approach is to classify symplectic finite subgroups of $O^+(\OG)$. We explain now how to determine whether a given finite subgroup of $O^+(\OG)$ is symplectic.

Let again $(X, \eta)$ be a marked IHS manifold of $OG10$ type. Any birational transformation of $X$ preserves the birational K\"ahler cone $\mathcal{BK}(X)$; the structure of this cone for a manifold of $OG10$ type is well understood \cite{mongardi2020birational}. In particular, the walls of $\overline{\mathcal{BK}(X)}$ are defined by the hyperplanes $D^\perp\subset \calC(X)$, where $D$ is a \textbf{stably prime exceptional divisor} \cite[\S 5]{markman}, and $\calC(X)$ denotes the connected component of the positive cone of $X$ containing a K\"ahler class. 
  We define the following set of vectors:
 \[\markman:=\{v\in\OG : v^2=-2\}\cup \{v\in \OG : v^2=-6,\; \divi_{\OG}(v)=3\}.\vspace*{-0.5cm}\]
  \begin{proposition}\cite[Proposition 3.1]{mongardi2020birational}
      Let $(X, \eta)$ be a marked IHS manifold of $OG10$ type. Then $D\in\Pic(X)$ is stably prime exceptional if and only if $\eta(D)\in \markman.$
 \end{proposition}
  
  It follows that $\calB\calK(X)$ is contained in an \textbf{exceptional chamber}; that is, a component of \[\calC(X)\setminus \bigcup_{v\in \markman} v^\perp\](see \cite[Theorem 3.2]{mongardi2020birational}).
  Using this description, we can rephrase the Global Torelli theorem (due to Huybrechts, Markman and Verbitsky) in a way that is more suited for the study of symplectic birational transformations of $X$. This will provide us with criteria for when a finite group $H\leq O^+(\OG)$ is symplectic.

	\begin{lemma}[{{{\cite[Theorems 2.15 and 2.17]{grossi2020finite}}}}]\label{lem: crit for bir symp eff}
	Let $H\leq O^+(\Lambda)$ be finite. Then $H$ is symplectic if and only if both of the following hold:
	\begin{enumerate}
		\item $\Lambda_H$ is negative definite, and
		\item $\Lambda_H\cap \markman=\varnothing$.
	\end{enumerate}
\end{lemma}

\subsection{Classification problems}\label{subsec: classification}
Recall that $\Lambda:= U^3\oplus E_8^2\oplus A_2$ is the lattice isometric to the second integral cohomology lattice of any IHS manifold of $OG10$ type. We aim to classify finite groups $H\leq O^+(\Lambda)$ that are symplectic, i.e. induced by a finite group $G$ of symplectic birational transformations for IHS manifolds $X$ of $OG10$ type. By the discussion in \Cref{subsec: bir transf}, this is equivalent to classifying, up to conjugacy, finite subgroups $H\leq O^+(\Lambda)$ with negative definite coinvariant sublattice, and such that $\Lambda_H\cap\markman = \varnothing$. 

\begin{remark}\label{neg def coinv}
    According to \cite[Lemma 2.3]{grossi2020finite}, if $H\leq O(\Lambda)$ is such that $\Lambda_H$ is negative definite, then $H\leq O^+(\Lambda)$.
\end{remark}

\begin{remark}\label{only sat matters}
    For a classification purpose, it is enough to classify such groups $H$ which are saturated in $O^+(\Lambda)$. In fact, any finite subgroups of $O^+(\Lambda)$ is a subgroup of a saturated group. Moreover, if a finite subgroup $H\leq O^+(\Lambda)$ is symplectic, then according to  \Cref{defn: saturation} and \Cref{lem: crit for bir symp eff}, so is its saturation (see also \cite[Remark 3.18]{bh22}). 
\end{remark}

We use the action of such a group $H\leq O^+(\Lambda)$ on the discriminant group $D_\Lambda$ in order to aid this classification. 
Note that, as abelian groups, $D_\Lambda\cong \mathbb{Z}/3\mathbb{Z}$, and in particular, there is an exact sequence
\[ 1\to O^{+,\#}(\Lambda)\to O^+(\Lambda)\to O(D_\Lambda)\to 1\]
where $O(D_\Lambda)$ has order 2 generated by $-\textnormal{id}$. Hence, if $H\leq O^+(\Lambda)$ is a finite subgroup, we have an exact sequence
\begin{equation}\label{short exact sequence}
    1 \to H^\#\to H\to \mu_2
\end{equation}
where $H\to \mu_2$ is defined by the discriminant representation of $H$.

\begin{corollary}\label{cor: cases}
    Let $H\leq O^+(\Lambda)$. 
    Then one of the following is satisfied:
    \begin{enumerate}
    \item either $H^\#$ is trivial and $H$ is cyclic of order 2;
    \item $H = H^\#$ is nontrivial; or
    \item $H^\#$ is nontrivial and $[H:H^\#] = 2$.
\end{enumerate}
\end{corollary}
Note that the classification of symplectic finite groups satisfying case (1) has already been completed by the same authors in \cite{marqOG10}. We treat the two other cases inductively. In what follows, we start by showing how to construct representatives of $O^+(\Lambda)$-conjugacy classes of finite saturated subgroups of $O^{+,\#}(\Lambda)$, covering case (2). Then, we cover the last case (3) by adapting the extension approach of \cite{brandhorst_hashimoto,bh22} to \Cref{short exact sequence}.

\section{Finite groups of stable symplectic isometries}\label{sec: stable sympl iso}

In this section, we classify symplectic finite subgroups $H\leq O^+(\Lambda)$, up to conjugacy in $O^+(\Lambda)$, that satisfy \Cref{cor: cases} case (2), i.e., when the group $H = H^\#$ is stable. More precisely, we prove the following:

\begin{theorem}\label{theo: sec 3 main}
     Let $(X, \eta, G)$ be a triple consisting of a marked IHS manifold $(X, \eta)$ of $OG10$ type and a finite group of symplectic birational transformations $G\leq\Bir(X)$ such that $\eta_\ast(G)\leq O^+(\Lambda)$ is stable. Then, up to conjugacy, $\eta_\ast(G)\leq O^+(\Lambda)$ is contained in one of the 192 stable groups of the  dataset \cite{database}.
\end{theorem}

We outline the strategy to prove \Cref{theo: sec 3 main}. 
Let $H\leq O^{+,\#}(\Lambda)$ be a stable symplectic finite subgroup that we aim to classify. 
In \Cref{subsec: stable symplectic lattices}, we show that the group $H$ is completely determined by the negative definite primitive sublattice $C:=\Lambda_H\subseteq \Lambda$, and $H$ is identified with the group $O^\#(C)$, which fixes no nontrivial vector in $C$. 
Thus, we can classify such groups $H$ by classifying instead primitive sublattices $C$ of $\Lambda$ satisfying the above properties and which in addition satisfy $C\cap \markman =\varnothing.$ 
Such a lattice is called a \textbf{heart} of $\Lambda.$
In \Cref{subsec: isometry class of hearts} we show that a heart $C\subseteq \Lambda$ primitively embeds into the Borcherds lattice, as a coinvariant lattice for a finite subgroup $H\leq \Aut(\calD)\leq O^+(\bB).$
In \Cref{subsec: algor class}, we determine the isometry class of the hearts that embed primitively into $\bB$ but not into $\bL$, represented by the three lattices in \Cref{App: exceptional hearts} (see \Cref{coro exceptional hearts}).
Finally, in \Cref{subsec: results for stable sympl} we complete the classification of hearts of $\Lambda,$ and obtain 192 $O^+(\Lambda)$-conjugacy classes of saturated symplectic finite subgroups of $O^{+,\#}(\Lambda),$ completing the proof of \Cref{theo: sec 3 main}.

\subsection{Finite stable subgroups of \texorpdfstring{$O^+(\Lambda)$}{O+(Λ) }}\label{subsec: stable symplectic lattices}
Let $H\leq O^+(\Lambda)$ be a symplectic finite subgroup such that $H=H^\#$. In particular, we assume $H\leq O^{+,\#}(\Lambda)$. The following holds. 

\begin{lemma}\label{lem: stably sat}
        Let $H\leq O^{+,\#}(\Lambda)$ be a nontrivial subgroup. Then $H\to O(\Lambda_H)$ is injective with image lying in $O^\#(\Lambda_H)$. Moreover, $H$ is saturated in $O^{+,\#}(\Lambda)$ if and only if $H = O^\#(\Lambda_H)$ holds, seeing $H$ as a subgroup of $O(\Lambda_H)$.
\end{lemma}

\begin{proof}
    Let us embed $\Lambda$ primitively into the even unimodular lattice $M := U^5\oplus E_8^2$, with orthogonal complement $F\cong A_2(-1)$. 
    Since $H$ is stable, we can extend it to a group of isometries $\widetilde{H}\leq O(M)$ acting as the identity on $F$, in such a way that $M_{\widetilde{H}} = \Lambda_H$. 
    Note moreover that $F\oplus \Lambda^H\subseteq M^{\widetilde{H}}$ is an overlattice. Since $\widetilde{H}$ acts trivially on $M^{\widetilde{H}}$, we see that $\widetilde{H}\to O(M_{\widetilde{H}})$ is injective. 
    Moreover, since $M$ is unimodular, the equivariant gluing condition \Cref{eq:egc} tells us that $\widetilde{H}$ maps into $O^\#(M_{\widetilde{H}})$. 
    We therefore get an injective morphism \[ H\xrightarrow{\cong}\widetilde{H}\hookrightarrow O^\#(M_{\widetilde{H}}) = O^{\#}(\Lambda_H).\]
    Let $g\in O^\#(\Lambda_H) = O^\#(M_{\widetilde{H}})$. 
    We can extend $g$ to an isometry $\tilde{g}\in O(M)$ acting as the identity on $M^{\widetilde{H}}$.
    Since $F\oplus \Lambda^H\subseteq M^{\widetilde{H}}$, we have that $\widetilde{g}$ restricts to an isometry $h$ of $\Lambda = F^\perp_M$ fixing pointwise $\Lambda^H$: $h$ is the extension of $g$ to $O(\Lambda)$ with the identity on $\Lambda^H$. 
    Moreover $h\in O^{+, \#}(\Lambda)$ according to \Cref{eq:egc} and \Cref{neg def coinv}.
    In particular $h$ lies in the saturation of $H$ in $O^{+,\#}(\Lambda)$.

    Conversely, for any $h\in O^{+, \#}(\Lambda)$ fixing pointwise $\Lambda^H$, we can extend $h$ to an isometry $\widetilde{h}$ of $O(M)$ by extending with the identity on $F$. Similarly as before, the restrictions of $h$ and $\widetilde{h}$ to $\Lambda_H = M_{\widetilde{H}}$ coincide, and they lie in $O^\#(\Lambda_H)$. We therefore conclude that the saturation of $H$ in $O^{+, \#}(\Lambda)$ coincides with the extension of $O^\#(\Lambda_H)$ with the identity on $\Lambda^H$.
\end{proof}

\begin{lemma}\label{cor: prim embedding act free}
    Suppose that $H\leq O^{+,\#}(\Lambda)$ is a nontrivial symplectic finite subgroup. Then the coinvariant sublattice $C := \Lambda_H$ is even negative definite, it does not contain $(-2)$-vectors, and $O^\#(C)$ fixes no nontrivial vector in $C$.
\end{lemma}

\begin{proof}
    The two first statements follow from \Cref{lem: crit for bir symp eff}. For the last statement, we use \Cref{lem: stably sat} which tells us that $H$ maps injectively into $O^\#(C)$, and $C^H = \{0\}$ by definition. Hence, in particular, $O^\#(C)$ fixes no nontrivial vector in $C$.
\end{proof}

It follows that saturated symplectic finite subgroups $H$ of $O^{+, \#}(\Lambda)$ are completely determined by some primitive sublattices $C\subseteq \Lambda$ which are even negative definite, and such that $O^\#(C)$ fixes no nontrivial vector in $C$; here $H$ being defined as the extension of $O^\#(C)$ with the identity on $C^\perp_\Lambda$.
The fixed point free action of $O^\#(C)$ on the sublattice $C\subseteq \Lambda$ is an additional tool in our classification. We thus make the following definition:

\begin{definition}\label{defin: ss-pairs}
    Let $C$ be an even lattice. We say that $C$ is \textbf{stable symplectic} if it is negative definite, and $O^\#(C)$ fixes no nontrivial vector in $C$.
\end{definition}

By \Cref{cor: prim embedding act free}, in order to construct 
saturated symplectic finite subgroups $H$ of $O^{+,\#}(\Lambda)$, it suffices to construct
primitive sublattices $ C\subseteq \Lambda$  where $C$ is stable symplectic, and such that $C\cap\mathcal{W}^{pex} = \varnothing$. In order to ensure this last condition, we observe the following.

Let us denote by $\Pi :=U^3\oplus E_8^3$ the unique even unimodular lattice of signature $(3,27)$.

\begin{proposition}\label{primitive embedding og10 in extended borcherds}
Up to isomorphism, $\Pi$ admits a unique primitive sublattice isometric to $\Lambda$.
\end{proposition}
\begin{proof}
    We note that $E_6$ is the unique (up to isometry) lattice with signature $(0,6)$ and discriminant group $D_{E_6}\cong D_\Lambda(-1)$.
    Hence there exists a primitive embedding of $\Lambda$ into the unimodular lattice $\Pi$, with orthogonal complement $E_6$. Since $O(E_6)\to O(D_{E_6})$ is surjective, by \cite[Proposition 1.14.1]{nikulin} we see that such a primitive embedding is unique up to the action of $O(\Pi)$ and $O(\Lambda)$.
\end{proof}

\begin{lemma}\label{th:radu}
    Let $C\subseteq \Lambda$ be a stable symplectic primitive sublattice without $(-2)$-vectors. Then  $C\cap \markman=\varnothing$.
\end{lemma}
\begin{proof}
    Let us consider the succession of inclusions
    \[C\subseteq C\oplus E_6\subseteq \Lambda\oplus E_6\subseteq \Pi.\]
    We will show that $C\cap\markman = \varnothing$. First note that if $C$ is trivial, then the result necessary holds. In what follows, we assume that $C$, and therefore $O^\#(C)$, is nontrivial.

    Suppose that $C$ has a vector $v$ of square $-6$ such that $\divi_\Lambda(v) = 3$.
    Since the lattice $C$ is stable symplectic, there exists an isometry $g\in O^\#(C)$ such that $g(v)\neq v$ holds. Note that $g(v)\neq -v$ also holds: in fact, since $g$ is stable, we have equalities
    \[\frac{v}{3}+C = D_g\left(\frac{v}{3}+C\right) = \frac{g(v)}{3}+C\in D_C.\]
    If $g(v) = -v$ were to hold, we would have that $\frac{2v}{3}\in C$: however this vector has square $\frac{-8}{3}\notin \bZ$ contradicting that $C$ is an even lattice. Hence $v':=g(v)$ is not proportional to $v$, and it still has divisibility 3 in $\Lambda$ (divisibility is preserved under isometry, and $g\in O^\#(C)$ being stable can be seen as an isometry of $\Lambda$). Similarly to the proof of \cite[Theorem 4.5, $(iii) \implies (i)$]{laza2019automorphisms} it follows that the primitive closure $M$ of $E_6+\bZ v+\bZ v'$ in $\Pi$ is isometric to $E_8$. In particular, since $E_8$ has a unique sublattice isometric to $E_6$ (up to isometry), with orthogonal complement isometric to $A_2$, one concludes that $A_2\cong (E_6)^\perp_M$ embeds into $C$: this is a contradiction since $C$ contains no $(-2)$-vectors. Hence, $C\cap\markman=\varnothing$.
\end{proof}

By \Cref{th:radu}, we see that any stable symplectic primitive sublattice $C\subseteq \Lambda$ not containing $(-2)$-vectors is the coinvariant sublattice of a symplectic finite subgroup $H\leq O^{+, \#}(\Lambda)$. We have therefore reduced our problem to constructing and classifying stable symplectic primitive sublattices of $\Lambda$ not containing $(-2)$-vectors. 

\begin{definition}\label{defn: heart }
    A \textbf{heart} is any stable symplectic primitive sublattice $C\subseteq \Lambda$ that does not contain any $(-2)$-vectors.
\end{definition}

In the following sections, we explain how to recover the abstract isometry class of such hearts. We then make use of the following theorem to eventually classify saturated symplectic finite subgroup $H\leq O^{+, \#}(\Lambda)$ up to conjugacy in $O^+(\Lambda)$.

\begin{theorem}\label{main th embeddings}
    Let $C$ be a stable symplectic lattice. Then the set of $O^+(\Lambda)$-conjugacy classes of saturated finite subgroups $H\leq O^{+, \#}(\Lambda)$ such that $\Lambda_H\cong C$ is in bijection with the set of isomorphism classes of primitive sublattices $C'\subseteq \Lambda$ such that $C'\cong C$.
\end{theorem}

\begin{proof}
    First note that since $\Lambda$ has real signature $(3, 21)$, we have that $-\textnormal{id}_\Lambda$ has negative real spinor norm \cite[Example 4.1]{brandhorst2020prime}. Hence $O(\Lambda)/O^+(\Lambda)$ is generated by the coset of $-\textnormal{id}_\Lambda$ which is a central involution. In other terms, for any $f\in O(\Lambda)$ we have that either $f$ of $-f$ lies $O^+(\Lambda)$. Since the conjugation actions on $O(\Lambda)$ induced respectively by $f$ and $-f$  are the same, we deduce that $O^+(\Lambda)$-conjugacy classes and $O(\Lambda)$-conjugacy classes of finite subgroups $H\leq O^{+,\#}(\Lambda)$ coincide. 
    The proof follows then from \Cref{lem: stably sat} which tells us that any saturated subgroup of $O^{+, \#}(\Lambda)$ is uniquely determined by its coinvariant sublattice.
\end{proof}

\begin{remark}\label{several embeddings}
    One may note that a consequence of \Cref{main th embeddings} is that there might be several conjugacy classes of subgroups $H\leq O^{+,\#}(\Lambda)$ whose coinvariant sublattice is abstractly isometric to a fixed stable symplectic lattice $C$.
\end{remark}

\subsection{Isometry classes of hearts of \texorpdfstring{$\Lambda$}{Λ}}\label{subsec: isometry class of hearts}
Recall that  $\mathbb{B} := U\oplus \mathbb{L}$, where $\bL$ is the Leech lattice, with fixed basis $\{e,s\}$ for $U\subseteq \mathbb{B}$ such that $e^2 = 0$, $s^2 = -2$ and $e.s = 1$. We denote again by $\calD$ the closure of Conway's chamber $D_0$.

\begin{lemma}\label{lem: embed in B}
    Let $C\subseteq \Lambda$ be a heart.
    Then $C$ embeds primitively into $\bB$. Moreover, the group $O^\#(C)$ is isomorphic to a subgroup $H\leq \Aut(\calD)$ so that $\bB_{H}\cong C$.
\end{lemma}
\begin{proof}
    Since $C$ is negative definite of rank at most 21, and since
\[\rank(C) + l(D_C) \leq \rank(\Lambda)+l(D_\Lambda) < 26\]
\cite[Proposition 1.15.1]{nikulin}, we see that $C$ embeds primitively into $\bB$ \cite[Corollary 1.12.3]{nikulin}. 

Let us fix $j\colon C\hookrightarrow \bB$ such a primitive embedding. Since $\bB$ is unimodular, \Cref{eq:egc} tells us that we can extend $O^\#(C)$ with the identity on $N := j(C)^\perp_\bB$ to a group $\widetilde{H}\leq O^+(\bB)$. In particular $\bB_{\widetilde{H}} = j(C) \cong C$. Since $C$ does not contain $(-2)$-vectors, $N\otimes \bR$ intersect a chamber $D$ of the positive cone of $\bB$. The group $\widetilde{H}$ acting trivially on $N$, we have that $\widetilde{H}$ preserves a vector in $D$ and thus, $\widetilde{H}$ preserves the entire chamber $D$. Hence $\widetilde{H}\leq \Aut(\overline{D})$. Since both $\Aut(\overline{D})$ and $\Aut(\calD)$ are isomorphic to $O^+(\bB)/W(\bB)$, we obtain that $\widetilde{H}$ is $O^+(\bB)$-conjugate to a subgroup $H\leq\Aut(\calD)$. In particular $\bB_H\cong \bB_{\widetilde{H}}\cong C$.
\end{proof}

\begin{remark}\label{len:hearts of B}
    Let $H\leq \Aut(\calD)$ be a nontrivial subgroup. According to the proof of \Cref{propo: first characterisation finite subgroups}, the group $H$ is finite if and only if $\bB_H$ is stable symplectic and does not contain roots of $\bB$.
\end{remark}


Therefore, as abstract lattices, we can view each heart $C\subseteq \Lambda $ as the coinvariant sublattice of a finite subgroup $H\leq \Aut(\calD)$. As already noted in \Cref{most examples are Leech}, in some cases these are actually primitively embedded into the Leech lattice. Stable symplectic sublattices of the Leech lattice are known and well-understood \cite{H_hn_2016}.
Unfortunately, there exist hearts $C\subset\Lambda$ that embeds primitively into $\bB$, but not $\bL$ --- we call such a hearts \textbf{exceptional}. In the next section, we describe a procedure to recover the abstract isometry class of such exceptional hearts. We then apply \Cref{main th embeddings} to this list of abstract lattices, and to H\"ohn--Mason list of stable symplectic sublattices of $\bL$, to classify conjugacy classes of hearts in $\Lambda$. 

\subsection{Exceptional hearts}\label{subsec: algor class}In this section, we aim to classify the remaining exceptional hearts: that is, the hearts $C'\subseteq \Lambda$ that embed primitively into $\bB$ but not into $\bL$. In order to do so, we will use properties of isometries in $\Aut(\mathcal{D})$, discussed in \Cref{subsec: Borcherds}. \Cref{counterexample} proves that such exceptional hearts do in fact exist.

\begin{example}\label{counterexample}
See the Notebook ``Counterexample" in the dataset \cite{database} for explicit computations and proofs of the following statements. Let $C$ be the negative definite even lattice with Gram matrix
    \[
\scriptscriptstyle{\begin{pmatrix}
-4 &  2 &  2 & -2 & -2 &  2 & -2 &  2 &  2 &  0 &  0 &  0 &  0 &  0 &  2 &  2 & -1 &  1 \\
 2 & -4 & -2 &  0 &  1 & -1 &  2 & -2 & -2 & -1 & -1 &  1 &  1 & -1 & -1 & -2 & -1 & -2 \\
 2 & -2 & -4 &  1 &  0 & -2 &  2 & -2 & -2 & -1 & -1 & -1 &  1 & -1 &  0 & -2 &  0 & -1 \\
-2 &  0 &  1 & -4 & -2 &  0 &  0 &  0 &  0 &  1 &  1 &  0 & -1 &  1 &  0 &  2 &  0 & -1 \\
-2 &  1 &  0 & -2 & -4 &  1 &  0 &  0 &  0 &  1 &  1 & -1 & -1 &  1 &  2 &  2 & -1 &  1 \\
 2 & -1 & -2 &  0 &  1 & -4 &  2 & -2 & -2 &  0 &  0 &  0 &  0 &  0 & -2 & -1 &  1 & -1 \\
-2 &  2 &  2 &  0 &  0 &  2 & -4 &  2 &  2 & -1 &  1 & -1 & -1 &  1 &  0 &  0 & -1 &  2 \\
 2 & -2 & -2 &  0 &  0 & -2 &  2 & -4 & -1 &  1 & -1 &  1 & -1 &  1 & -1 & -1 &  0 & -1 \\
 2 & -2 & -2 &  0 &  0 & -2 &  2 & -1 & -4 &  0 &  1 & -1 &  0 &  0 & -1 & -1 &  0 & -1 \\
 0 & -1 & -1 &  1 &  1 &  0 & -1 &  1 &  0 & -4 & -1 & -1 &  2 & -2 &  0 & -2 & -1 &  0 \\
 0 & -1 & -1 &  1 &  1 &  0 &  1 & -1 &  1 & -1 & -4 &  2 &  2 & -1 &  1 & -1 &  0 & -1 \\
 0 &  1 & -1 &  0 & -1 &  0 & -1 &  1 & -1 & -1 &  2 & -4 &  0 &  0 &  0 &  0 &  0 &  1 \\
 0 &  1 &  1 & -1 & -1 &  0 & -1 & -1 &  0 &  2 &  2 &  0 & -4 &  2 & -1 &  1 &  0 &  1 \\
 0 & -1 & -1 &  1 &  1 &  0 &  1 &  1 &  0 & -2 & -1 &  0 &  2 & -4 &  1 & -1 &  0 & -1 \\
 2 & -1 &  0 &  0 &  2 & -2 &  0 & -1 & -1 &  0 &  1 &  0 & -1 &  1 & -4 & -1 &  1 & -1 \\
 2 & -2 & -2 &  2 &  2 & -1 &  0 & -1 & -1 & -2 & -1 &  0 &  1 & -1 & -1 & -4 &  0 & -1 \\
-1 & -1 &  0 &  0 & -1 &  1 & -1 &  0 &  0 & -1 &  0 &  0 &  0 &  0 &  1 &  0 & -4 &  1 \\
 1 & -2 & -1 & -1 &  1 & -1 &  2 & -1 & -1 &  0 & -1 &  1 &  1 & -1 & -1 & -1 &  1 & -4 
\end{pmatrix}}
.\]
The lattice $C$ lies in the genus $\II_{(0,18)}3^{-7}$,  $O^\#(C)\cong C_3\times C_3$ fixes no nontrivial vector in $C$, and $C$ contains no $(-2)$-vectors. Using \cite[Proposition 1.15.1]{nikulin}, one can check that the lattice $C$ admits a primitive embedding into $\Lambda$ with orthogonal complement isometric to $U(3)^{ 3}$. 

However, we remark that $\text{rank}(C)+l(D_C) = 25$: this implies that $C$ does not embed primitively into the Leech lattice, but it does embed primitively into $\bB$.
\end{example}

In order to determine, up to isometry, the stable symplectic sublattices of $\bB$ not containing roots that do not embed primitively into the Leech lattice, we know from \Cref{most examples are Leech} that we need to look for finite subgroups $H\leq\Aut(\calD)$ such that there does not exist any $v\in \bL$ so that $g(v)-v=\lambda$ for all $(\lambda, g)\in H$ --- we refer to such groups as \textbf{exceptional}. 
Note that since the isometry class of $\bB_H$ is preserved under conjugation of $H$ by any element of $O(\bB)$, we describe a procedure to recover at least one representative for each $\Aut(\calD)$-conjugacy class of finite exceptional subgroups of $\Aut(\calD)$.\bigskip

Recall from the proof of \Cref{iso conway} that we have an exact sequence
\[ 0\to \bL\xrightarrow{\psi}\Aut(\mathcal{D})\xrightarrow{\pi} O(\bL)\to 1\]
where $\pi$, which is induced by the isometry $e^\perp/\bZ e\cong \bL$, admits a section $\phi\colon O(\bL)\to \Aut(\calD)$. Moreover, the same result tells us that $\Aut(\calD) = \bL\rtimes O(\bL)$ so any element $h\in \Aut(\calD)$ of finite order can be written as a pair $(\lambda, g)$ where $g := \pi(h)$ and $\lambda\in\bL_g$ (\Cref{propo: first comput}). The section $\phi$ above sends any $g\in O(\bL)$ to $\phi(g) := (0, g)\in\Aut(\calD)$.

Let $G\leq O(\bL)$ be a subgroup.
We define a $\bZ$-linear map
\[p_G\colon \bL\to\prod_{g\in G}\bL_g,\; v\mapsto (g(v)-v)_{g\in G},\]
whose kernel is exactly $\bL^G$. Let us denote by $m$ the order of $G$. We define moreover
\[\bL\to \prod_{g\in G}\bL_g/m\bL_g,\; v\mapsto (g(v)-v+m\bL_g)_{g\in G}\]
whose kernel is denoted by $K(G)$. We observe that $K(G)$ contains $\bL^G+m\bL$. We denote by $A(G) := K(G)/(m\bL+\bL^G)$ --- it is a finite abelian group. We have seen in \Cref{propo: first characterisation finite subgroups} and \Cref{rmk enough take order of group} that for any finite subgroup $H\leq \Aut(\calD)$ satisfying $\pi(H) = G$, there exists a vector $v\in \bL$ such that every element $h = (\lambda, g)\in H$ satisfies
\[g(v) - v = m \lambda\]
and $\lambda\in \bL_g$. In particular, $v\in K(G)$ and it is uniquely determined by $H$, up to translation by a vector in $\ker(p_G) = \bL^G$.

\begin{lemma}\label{procedure lemma}
    Let $H, H'\leq \Aut(\calD)$ be finite subgroups such that $\pi(H) = \pi(H') = G$, and let $v,v'\in K(G)$ be associated vectors. Then the following items are equivalent:
    \begin{enumerate}
        \item the groups $H, H'$ are $\bL$-conjugate in $\Aut(\calD)$;
        \item the vectors $v$ and $v'$ define the same class in $A(G)$.
    \end{enumerate}
\end{lemma}

\begin{proof}
    Let us suppose that there exists $\mu\in \bL$ be such that $\psi_\mu H\psi_\mu^{-1} = H'$, where $\psi_{\mu}$ is the Eichler--Siegel transformation associated to $\mu$ (see \Cref{eichlersiegel}). Then, for all $g\in G$, we have
    \[\psi_\mu\left(\frac{g(v)-v}{m}, g\right)\psi_\mu^{-1} = \left(\frac{g(v')-v'}{m}, g\right)\]
    which is equivalent to
    \[g(v-v') - (v-v') = m(g(\mu)-\mu).\]
    Thus, we conclude that $v-v'-m\mu\in \bL^g$ for all $g\in G$, meaning exactly that $v-v'\in \bL^G+m\bL$. The converse holds similarly, by reversing the order of the arguments.
\end{proof}

\begin{corollary}\label{procedure corollary}
    The group $A(G)$ is trivial if and only if for all $H\leq \Aut(\calD)$ finite such that $\pi(H) = G$, the lattices $\bB_H$ and $\bL_G$ are isometric.
\end{corollary}

\begin{proof}
    If $A(G)$ is trivial, we know from \Cref{procedure lemma} that any $H\leq \Aut(\calD)$ finite such that $\pi(H) = G$ is $\bL$-conjugate to $\phi(G)$: in particular $\bB_H\cong \bB_{\phi(G)} = \bL_G$.

    Now, if $A(G)$ is nontrivial, then there is a vector $v\in \bL\setminus (\bL^G+m\bL)$ such that, for all $g\in G$, $\lambda_g := \frac{g(v)-v}{m}\in \bL_g$ and, $\{\lambda_g\}_{g\in G}$ does not lie in the image of $\bL$ by $p_G$. Therefore, according to \Cref{most examples are Leech}, the finite subgroup $H := \{(\lambda_g, g)\;\mid\; g\in G\}$ satisfies that $\bB_H\not\cong \bL_G$.
\end{proof}

\begin{remark}\label{procedure remark}
    Let $H'\leq \Aut(\calD)$ be finite such that $G' := \pi(H')$ is $O(\bL)$-conjugate to $G$. Let $f\in O(\bL)$ be such that $fG'f^{-1} = G$. Then the group
    \[H := \left\{(f(\lambda'),\, fg'f^{-1})\;\mid\; (\lambda', g')\in H'\right\} = (0, f)H'(0, f^{-1})\]
    is conjugate to $H'$ and satisfies $\pi(H) = fG'f^{-1} = G$.
\end{remark}

Therefore, in order to construct at least one representative for each $\Aut(\calD)$-conjugacy class of finite exceptional subgroups $H$ of $\Aut(\calD)$,  we proceed as follows:

\begin{enumerate}
    \item we start by fixing a stable symplectic sublattice $C$ of the Leech lattice $\bL$ \cite{H_hn_2016};
    \item we compute a set $\mathcal{G}$ of representatives of $O(\bL)$-conjugacy classes of finite subgroups $G$ of $O^\#(C)$ such that $C^G=\{0\}$ (\Cref{procedure remark});
    \item for any $G\in \mathcal{G}$,  we compute $A(G)$: if this is trivial, we try a new group (\Cref{procedure corollary});
    \item for every $[v]\in A(G)$ nontrivial, we define $H:= \left\{\left(\frac{g(v)-v}{\#G}, g\right)\;\mid\; g\in G\right\}\leq \Aut(\calD)$.
\end{enumerate}
Note that according to \Cref{procedure lemma}, the $\Aut(\calD)$-conjugacy class of $H$ in step (4) does not depend on a choice of a representative for the nontrivial class $[v]\in A(G)$.

\begin{theorem}\label{theorem exceptional lattices}
    Let $H\leq \Aut(\calD)$ be an exceptional finite subgroup so that $\bB_H$ has rank at most 21. Then, $\bB_H$ is abstractly isometric to one of the 101 stable symplectic lattices in the folder ``exceptional" of the dataset \cite{database}.
\end{theorem}

\begin{proof}
    We apply the previous procedure to the list of stable symplectic sublattices of the Leech lattice, of rank at most 21. For each such sublattices $C\leq \bL$, we enumerate conjugacy classes of subgroups $G\leq O^\#(C)$ so that $C^G = \{0\}$ and $A(G)$ is nontrivial. By doing so, we obtain at least one representative for every $\Aut(\calD)$-conjugacy class of exceptional finite subgroups of $H\leq \Aut(\calD)$ satisfying $\pi(H) = G$ (\Cref{procedure lemma}, \Cref{procedure remark}). We collect the coinvariant sublattices $\bB_H$ associated to each such exceptional subgroup of $H\leq \Aut(\calD)$, and we keep only one representative for each isometry class of lattices they represent. We record information about such lattices $\bB_H$ in \Cref{appendix: exceptional stable}, \Cref{tab: exceptional}.
\end{proof}

For our purposes, we only need those stable symplectic lattices of $\bB$ that can embed into $\Lambda$. Thus, we conclude the following:

\begin{corollary}\label{coro exceptional hearts}
    Let $C\subseteq \Lambda$ be an exceptional heart. Then $C$ is abstractly isometric to one of the three stable symplectic lattices given in \Cref{App: exceptional hearts}.
\end{corollary}

\begin{proof}
    According to the proof of \cite[Proposition 1.15.1]{nikulin}, a stable symplectic lattice $C$ embeds primitively into $\Lambda$ only if for all prime number $p$
    \[ \rank(C)+l(D_p) \leq 24+\delta_{3,p}\]
    where $l(D_p)$ is the length of the $p$-Sylow subgroup of $D_C$, and $\delta_{3,p} = 1$ if $p=3$ and 0 otherwise. In fact, it follows from the fact that $D_{\Lambda}$ has order 3 and any $\bZ_3$-lattice admits an orthogonal basis.
    We deduce from \Cref{tab: exceptional} that there are only three possible isometry classes of exceptional heart $C\subseteq \Lambda$.
\end{proof}

\begin{remark}
    If $C\subseteq \Lambda$ is an exceptional heart, \Cref{coro exceptional hearts} tells us what its isometry class can be. However, we have not proved yet that each of the three are realised as an exceptional heart of $\Lambda$.
\end{remark}

One observes that for all sublattices $C\subseteq \bB$ presented in \Cref{tab: exceptional}, we have $\text{rank}(C)+l(D_C)>24$. In particular, the following holds.

\begin{corollary}\label{main th stable symplectic}
    Let $C$ be a stable symplectic lattice of rank at most 21 and not containing $(-2)$-vectors. Then $C$ embeds primitively into the Leech lattice $\bL$ so that $\bL_{O^\#(C)} \cong C$ if and only if 
    $$\textnormal{rank}(C)+l(D_C)\leq 24.$$
\end{corollary}

\begin{proof}
    Already note that if $C$ embeds primitively into the Leech lattice, then according to \cite[Corollary 1.12.3]{nikulin} we have that \(\rank(C)+l(D_C)\leq \rank(\bL) = 24.\)
    
    Now suppose that $\rank(C)+l(D_C)\leq 24$. By this assumption, we know that $C$ embeds primitively into $\bB$ and this in a unique way up to the action $O(\bB)$ \cite[Proposition 1.14.1]{nikulin}. In particular, we can choose a primitive embedding $j\colon C\hookrightarrow \bB$ such that the complement $N:= j(C)^\perp_\bB$ intersects the interior of $\mathcal{D}$. Therefore, as in the proof of \Cref{lem: embed in B}, we have that the finite group $H\leq O^+(\bB)$ obtained by extending $O^\#(C)$ with the identity on $N$ preserves $\mathcal{D}$. Thus $C$ is the coinvariant sublattice of a finite subgroup of $\Aut(\calD)$. There are two cases: either $H$ is exceptional, or $C$ is isometric to $\bL_{\pi(H)}$.
    However, the former is not possible since \Cref{tab: exceptional} tells us that the coinvariant sublattices $N$ of exceptional finite subgroups of $\Aut(\calD)$ satisfy
    \[\rank(N)+l(D_N)\geq 25.\qedhere\]
\end{proof}

\subsection{Results and comments}\label{subsec: results for stable sympl}
Now that we have determined the possible isometry classes for the hearts of $\Lambda$, we can conclude our classification by testing whether each such lattices embeds primitively into $\Lambda$, and in how many ways up to the action of $O(\Lambda)$ (\Cref{main th embeddings}).

\begin{proposition}\label{propo: th 4.1 lattice}
    There are exactly 192 $O^+(\Lambda)$-conjugacy classes of saturated symplectic finite subgroups $H\leq O^{+, \#}(\Lambda)$. 
\end{proposition}
\begin{proof}
    According to \Cref{main th embeddings} and \Cref{th:radu}, it suffices to compute representatives for the isomorphism classes of hearts $C\subseteq \Lambda$. According to \Cref{lem: embed in B}, \Cref{len:hearts of B}, \Cref{most examples are Leech} and \Cref{coro exceptional hearts}, such a heart is abstractly isometric either to a stable symplectic sublattice of the Leech lattice $\mathbb{L}$, or to one of the three exceptional lattices in \Cref{App: exceptional hearts}. Recall that the former are classified in \cite{H_hn_2016}. We complete the classification by applying \cite[Proposition 1.15.1]{nikulin} to each such lattices $C$. We use an implementation of the algorithm described in the proof of \cite[Proposition 1.15.1]{nikulin} by the second author in \cite[QuadFormAndIsom]{oscar-book} to compute a representative for each isomorphism class of primitive sublattices of $\Lambda$ isometric to $C$. The results of such computations are available in the folder ``data" of the dataset \cite{database}: we give an overview of this dataset in \Cref{app: table of hearts}, \Cref{tab: lovely table 45}. \qedhere
\end{proof}

In each case from \Cref{propo: th 4.1 lattice}, the groups $H\leq O^{+,\#}(\Lambda)$ are saturated in $O^{+,\#}(\Lambda)$ but they might not be saturated in $O^+(\Lambda)$. Whether such a group $H$ is saturated in $O^+(\Lambda)$ can be tested as follows.

\begin{lemma}\label{lemma saturation}
    Let $C\subseteq \Lambda$ be a negative definite primitive sublattice, and let $P(C)$ be the largest subgroup of $O(C)$ which can be extended with the identity on $C^\perp_\Lambda$. Suppose that $P(C)$ fixes no nontrivial vector in $C$. Then, for any subgroup $H\leq P(C)$ fixing no nontrivial vector in $C$, we have that
    \[\Sat_{O^+(\Lambda)}(H) = P(C),\]
    where we see $H$ and $P(C)$ as subgroups of $O^+(\Lambda)$ by extending with the identity on $F:= C^\perp_\Lambda$.
\end{lemma}

\begin{proof}
    Since $H$ fixes no nontrivial vector in $C$, we have that $\Lambda_H = C$ --- hence $\Sat_{O^+(\Lambda)}(H)$ is the pointwise stabiliser of $F:= C^\perp_\Lambda$ in $O^+(\Lambda)$. Now, by definition, any element $f$ of $\Sat_{O^+(\Lambda)}(H)$ acts trivially on $F$, so its restriction $f_{\mid C}$ to $C$ lies in $P(C)$. Conversely, by definition, any element of $P(C)$ extends to an isometry of $\Lambda$ with trivial spinor norm (\Cref{neg def coinv}) and acting trivially on $F$. 
\end{proof}

\begin{remark}\label{rmk saturation}
With the notation from \Cref{lemma saturation}, let $ D_F\geq H_F\xrightarrow{\gamma} H_C\leq D_C$ be the glue map associated to the primitive extension $F\oplus C\subseteq \Lambda$. Let us denote by $\pi_C\colon O(C)\to O(D_C)$ the discriminant representation of $O(C)$, whose image is denoted by $G_C:= \pi_C(O(C))$. Let moreover $G_C(H_C)$ be the stabiliser of $H_C$ in $G_C$, and let us finally denote $K_C := \ker(G_C(H_C)\to O(H_C))$. Then it follows from \Cref{eq:egc} that $P(C) = \pi_C^{-1}(K_C)$. In particular, since $C$ is negative definite, we have that
    \[\#\Sat_{O^+(\Lambda)}(H) = \#\pi_C^{-1}(K_C) = \#K_C\times\#\ker(\pi_C) = \displaystyle\frac{\#K_C\#O(C)}{\#G_C}.\]
\end{remark}

In the case of nonsaturated groups, there is a way of characterising the saturation of $O^\#(C)$ in $O^+(\Lambda)$.

\begin{proposition}\label{possible saturations}
    Let $H \leq O^{+,\#}(\Lambda)$ be saturated, and let $G$ be the saturation of $H$ in $O^+(\Lambda)$. Then $G^\#=H$ and $[G:H] \leq 2$.
\end{proposition}

\begin{proof}
    The fact that $G^\#=H$ follows from $H$ being saturated in $O^{+,\#}(\Lambda)$. We conclude by remarking that $G/H$ embeds into $\text{im}(O(\Lambda)\to O(D_\Lambda))$ which has order 2.
\end{proof}

\begin{proof}[Proof of \Cref{theo: sec 3 main}]
    This follows from \Cref{lem: crit for bir symp eff} together with \Cref{propo: th 4.1 lattice} for the list of conjugacy classes of groups, and \Cref{lemma saturation} and \Cref{rmk saturation} for determining the saturated ones (\Cref{only sat matters}).
\end{proof}

Let $X$ be an IHS manifold of $OG10$ type and let $f\in\Bir_s(X)$ be of prime order $p$. It follows that $\eta_*(f)$ acts trivially on the discriminant group $D_\Lambda$ unless $p=2$. Using the results of \Cref{theo: sec 3 main} and the classification of symplectic birational involutions \cite{marqOG10}, we can extract a complete classification of prime order symplectic birational transformations of manifolds of $OG10$ type. We display for convenience of interested readers.

\begin{theorem}\label{th: prime order}
    Let $X$ be an IHS manifold of $OG10$ type, $f\in \Bir_s(X)$ a symplectic birational transformation of prime order $p$. Then the induced action of $H:=\langle \eta_*(f)\rangle$ on $\Lambda$ appears in \Cref{fig:cyclic prime}.
\end{theorem}

\begin{proof}
The classification of such $H$ for $p=2$ follows from \cite{marqOG10}.
Suppose now that $p$ is odd. In particular, we have that $H$ is stable of prime order $p$, and therefore $C:=\Lambda_H$ is $p$-elementary (see for instance \cite[Lemma 2.8]{grossi2020finite}). Hence, by \Cref{propo: th 4.1 lattice}, we see that either $C$ embeds primitively into the Leech lattice, or $p = 3$ and $C$ is isometric to the lattice defined in  \Cref{counterexample}. 
\begin{enumerate}
    \item In the former case, we know that $C$ embeds primitively into the Leech lattice, so we can use the classification of conjugacy classes of prime order isometries of the Leech lattice in \cite{H_hn_2016}. For each class of prime order isometries $g$ of the Leech lattice, we know that $ \mathbb{L}_g$ is stable symplectic and we moreover have that $(\bL_g, g_{\mid \bL_g})$ is unique up to isomorphism. Hence, we can infer from \Cref{tab: lovely table 45} whether $\bL_g$ embeds primitively into $\Lambda$. If it does, then we know that $\Lambda$ has a prime order isometry $f\in O^{+,\#}(\Lambda)$ with $\Lambda_f\cong \bL_g$, and \Cref{tab: lovely table 45} also tells us what is the genus of $F := \Lambda^f$. Note that one can apply \cite[Algorithm 2]{bh22} to the triple of lattices with isometry $((F, \id), (\bL_g, g_{\mid \bL_g}), (\Lambda, \id))$, with $k=p$, to obtain the number of $O^+(\Lambda)$-conjugacy classes of such isometries $f$.
    \item In the latter case, one can use an algorithm of Plesken--Souvignier \cite{ps97} to show that the stable symplectic lattice $C$ defined in \Cref{counterexample} admits a unique conjugacy class of fixed-point free isometry $g$ of order 3.
    By applying again \cite[Algorithm 2]{bh22} to the triple of lattices with isometry $((U(3)^{3}, \id), (C, g), (\Lambda, \id))$ with $p=k=3$, and we obtain that there are two $O^+(\Lambda)$-conjugacy classes of cyclic subgroups of $O^{+, \#}(\Lambda)$ generated by an isometry $f$ of order 3 such that $\Lambda_f\cong C$.
\end{enumerate}
In all case, since the lattices $C$ we consider are stable symplectic without $(-2)$-vectors, we know that the isometries $f\in O^{+, \#}(\Lambda)$ we have constructed are symplectic (\Cref{th:radu}). We refer to the Notebook ``Prime" in the folder ``verification" of the dataset \cite{database} for more details about the computations described in (1) and (2). 
\end{proof}

\begin{remark}
    We define the two following rank 2 even lattices:
    \[ K_7 := \begin{pmatrix}
        4&-1\\-1&2
    \end{pmatrix},\quad\quad H_5 := \begin{pmatrix}
        2&1\\1&-2
    \end{pmatrix}\]
    which are respectively positive definite and indefinite, of respective determinant 7 and $-5$.
\end{remark}

For each finite cyclic group $H$ as in \Cref{th: prime order}, we display in \Cref{fig:cyclic prime} the invariant sublattice $\Lambda^H$ and the genus of the coinvariant sublattice $g(\Lambda_H)$. Moreover, the number of $O^+(\Lambda)$-conjugacy classes of such a group, for fixed $(\Lambda^H, g(\Lambda_H))$, is given in the last column of the table.
{\small\centering\setlength{\tabcolsep}{8pt}
\renewcommand\arraystretch{1.5}
\begin{table}[!t]
    \centering
    \begin{tabular}{cccc}
         $\#H$&$\Lambda^H$&$g(\Lambda_H)$&nb classes \\
         \hline
         \rowcolor{lightgray!40!white}1&$\Lambda$&\textemdash&1\\
         2&$U^3\oplus D_4^3$&$\II_{(0,6)}2^{-6}3^1 $& 1\\
         \rowcolor{lightgray!40!white}2&$U^3\oplus A_2\oplus E_8(2)$&$\II_{(0,8)}2^8$&1\\
         2&$U^2\oplus A_1\oplus A_1(-1)\oplus E_8(2)$ &$\II_{(0,10)}2^{-10}_43^1$&1\\
         \rowcolor{lightgray!40!white}2&$U(2)^2\oplus A_1\oplus A_1(-1)\oplus E_6(2)$&$\II_{(0,12)}2^{12}_4$&1\\
         2&$A_1^9\oplus A_1(-1)^3$&$\II_{(0,12)}2^{-12}_23^1$&1\\
         \rowcolor{lightgray!40!white}2&$A_1^5\oplus A_1(-1)^3$&$\II_{(0,16)}2^{-8}_63^1$&1\\
         3&$U\oplus U(3)^2\oplus E_6$&$\II_{(0,12)}3^6$&1\\
         \rowcolor{lightgray!40!white}3&$U(3)^3\oplus E_6$&$\II_{(0,12)}3^6$&1\\
         3&$U(3)^3\oplus A_2$&$\II_{(0,16)}3^8$&1\\
         \rowcolor{lightgray!40!white}3&$U\oplus U(3)^2$&$\II_{(0,18)}3^5$&1\\
         3&$U(3)^3$&$\II_{(0,18)}3^5$&1\\
         \rowcolor{lightgray!40!white}3&$U(3)^3$&$\II_{(0,18)}3^{-7}$&2\\
         5&$U\oplus U(5)^2\oplus A_2$&$\II_{(0,16)}5^4$&1\\
         \rowcolor{lightgray!40!white}5&$H_5\oplus A_2(-5)$&$\II_{(0,20)}5^3$&1\\
         7&$U(7)\oplus K_7\oplus A_2$&$\II_{(0,18)}7^3$&1\\
         \rowcolor{lightgray!40!white}11&$U\oplus A_2(-11)$&$\II_{(0,20)}11^2$&1\\
         \hline
    \end{tabular}
    \caption{Classification of prime order symplectic birational transformations for IHS manifolds of $OG10$ type}
    \label{fig:cyclic prime}
\end{table}}

\section{From stable symplectic to symplectic}\label{sec: from stable sympl to sympl}
It remains to classify symplectic finite subgroups $H\leq O^+(\Lambda)$ that satisfy \Cref{cor: cases} case (3), i.e. groups $H$ where $H^\#\leq H$ is an index 2 subgroup. In particular, there exists a short exact sequence (\Cref{short exact sequence}):
$$1\rightarrow H^\#\rightarrow H\rightarrow  \mu_2\rightarrow 1,$$ where $H^\#\leq O^{+,\#}(\Lambda).$ In this section, we adapt an extension approach of \cite{brandhorst_hashimoto,bh22} in order to classify those symplectic finite subgroups $H$ whose subgroup $H^\#$ is contained in the classification of \Cref{sec: stable sympl iso}. More precisely, we prove the following:

\begin{theorem}\label{thm: main thm for OG10 intro}
     Let $(X, \eta, G)$ be a triple consisting of a marked IHS manifold $(X, \eta)$ of $OG10$ type and a finite group of symplectic birational transformations $G\leq\Bir_s(X)$. Then, up to conjugacy, $\eta_\ast(G)\leq O(\Lambda)$ is contained in one of the 375 saturated groups of the dataset \cite{database}.
\end{theorem}

We outline our strategy to prove \Cref{thm: main thm for OG10 intro}. 
In \Cref{subsec:BHH extension} we explain our adaptation of the Brandhorst--Hashimoto--Hofmann extension approach. Briefly, we develop a method to extend symplectic finite groups $H^\#$ that act trivially on the discriminant group $D_\Lambda$ by $\mu_2$ to obtain a subgroup $H\leq O^+(\Lambda)$ that is both symplectic and saturated. 
Note that for any generator $hH^\#$ of $H/H^\#\cong \mu_2$, the isometry $h$ 
determines an equivariant primitive extension $$(\Lambda ^{H^\#}, h|_{\Lambda^{H^\#}})\oplus (\Lambda_{H^\#}, h|_{\Lambda_{H^\#}})\subseteq (\Lambda, h).$$ 
Since the group $H$ is generated by $H^\#$ and $h$, knowing $\Lambda_{H^\#}\subseteq \Lambda$, the pair $(\Lambda^{H^\#}, h|_{\Lambda^{H^\#}})$ and the previous equivariant primitive extension, is enough to reconstruct $H\leq O^+(\Lambda)$.
The main result of the section is \Cref{th:classth}, which provides us with a way to construct all possible extensions of a symplectic finite subgroup $H^\#\leq O^{+, \#}(\Lambda)$ starting from the primitive sublattice $\Lambda_{H^\#}\subseteq \Lambda.$ 
In \Cref{sec: algorithms}, we explain the algorithms to undertake this procedure, and prove that for a given primitive sublattice $\Lambda_{H^\#}\subseteq \Lambda$ they return a complete set of conjugacy classes of $(\Lambda, H)$ with $H\leq O^+(\Lambda)$ symplectic finite subgroup extending $H^\#$ as in \Cref{short exact sequence}. 
We prove \Cref{thm: main thm for OG10 intro} in \Cref{subsec: results for BHH approach} --- our classification is contained in the folder ``data" of the dataset \cite{database}.

\subsection{Brandhorst--Hashimoto--Hofmann extension approach}\label{subsec:BHH extension}
In what follows, we adapt the \emph{extension approach} as described in \cite{brandhorst_hashimoto,bh22} to our context. We do not expect that the reader is familiar with the content of the previously cited works. Our presentation of what follows is self contained.

Let again $\Lambda:= U^3\oplus E_8^2\oplus A_2$. Let $H\leq O^+(\Lambda)$ be finite such that $H^\#$ is nontrivial and $H\neq H^\#$. Note that $\Lambda$ is unique in its genus, and for any lattice $\Lambda'\cong \Lambda$, we denote by $\markman(\Lambda')$ to be the set $\varphi(\markman)$ for any isometry $\varphi\colon \Lambda\xrightarrow{\cong}\Lambda'$.
We have seen that there is an exact sequence (\ref{short exact sequence})
\[1\to H^\#\to H\to \mu_2\to 1\]
where the character $H\to \mu_2$ is induced by the discriminant representation $H\to O(D_\Lambda)$. In particular, $[H:H^\#] = 2$ and there exists a nonstable isometry $h\in H\setminus H^\#$ such that $H/H^\#$ is generated by $hH^\#$. The restriction $a$ of $h$ to $\Lambda^{H^\#}$ has order $2$, except when $H^\#$ is not saturated in $O^+(\Lambda)$ and $\Lambda^{H^\#} = \Lambda^{H}$, in which case $a = \textnormal{id}$ (see \Cref{possible saturations}). 
In both cases, we observe that the definition of $a\in O(\Lambda^{H^\#})$ does not depend on a choice of a representative of $hH^\#$. However, for any $h'\in hH^\#$, the restriction $b$ of $h'$ to $\Lambda_{H^\#}$ does depend on the choice of $h'$, and it determines an equivariant primitive extension $(\Lambda^{H^\#}, a)\oplus (\Lambda_{H^\#}, b)\subseteq (\Lambda, h')$ with associated glue map $\gamma$.

\begin{remark}
    Even though the equivariant primitive extension 
    \[(\Lambda^{H^\#}, a)\oplus (\Lambda_{H^\#}, b)\subseteq (\Lambda, h')\]
    depends on the choice of $h'$, the associated glue map $\gamma$ does not depend on any $h'\in hH^\#$ or even on $a$. Indeed, $\gamma$ is the glue map of the primitive extension $\Lambda^{H^\#}\oplus \Lambda_{H^\#}\subseteq \Lambda$.
\end{remark}

\begin{definition}\label{heart for symplectic}
    Let $H$ be a symplectic finite subgroup of $O^+(\Lambda)$. We define:
    \begin{enumerate}
        \item the \textbf{heart} of $H$ to be the stable symplectic lattice $\Lambda_{H^\#}$;
        \item the \textbf{head} of $H$ to be the pair $(\Lambda^{H^\#}, a)$;
        \item the \textbf{spine} of $H$ to be the glue map $\gamma$.
    \end{enumerate}
\end{definition}

\begin{remark}
    For a finite subgroup $H\leq O^+(\Lambda)$, if $H$ is saturated in $O^+(\Lambda)$ then $H^\#$ is saturated $O^{+, \#}(\Lambda)$ (\Cref{sat implies stab sat}). In particular, the group $H^\#$ is completely characterised by the primitive sublattice $\Lambda_{H^\#}\subseteq \Lambda$, and it is equal to $O^\#(\Lambda_{H^\#})$ (\Cref{lem: stably sat}). Moreover, the sublattice $\Lambda_{H^\#}\subseteq \Lambda$ is a heart, as defined in \Cref{defn: heart }.
\end{remark}

\begin{proposition}\label{same chara have same extension data2}
    Let $H_1,H_2\leq O^+(\Lambda)$ be two symplectic finite subgroups and let $\psi\in O^+(\Lambda)$ conjugate $H_1$ and $H_2$. Then:
    \begin{enumerate}
        \item $\psi$ restricts to an isometry between the respective hearts of $H_1$ and $H_2$;
        \item $\psi$ induces an isomorphism between the respective heads of $H_1$ and $H_2$.
    \end{enumerate}
\end{proposition}

 \begin{proof}
     The first claim follows by definition of $\psi$, and the second claim follows from the fact that the definition of the heads does not depend on the choice of a representative for the generators of $H_1/H_1^\#$ and $H_2/H_2^\#$ respectively. 
\end{proof}

From \Cref{same chara have same extension data2}, one sees that two conjugate symplectic finite subgroups of $O^+(\Lambda)$ share the same heart and head. We would like to measure to which extent the converse holds, and describe an algorithm to compute representatives of conjugacy classes of such groups with given heart and head.

\begin{definition}\label{defin: head}
   Let $C\subseteq \Lambda$ be a heart, and let $F := C^\perp_\Lambda$. A \textbf{head} of $C$ is a pair $(F,a)$ where the isometry $a\in O(F)$ has order $n=1,2$, and $F_a$ is negative definite or trivial. We call $n\in \{1,2\}$ the order of the head.
\end{definition}

Now, for each head $(F,a)$ of a given heart  $C$, we would like to classify isomorphism classes of equivariant primitive extensions
\begin{equation}\label{eq:eqprimext}
    (F,a)\oplus (C,b)\subseteq (\Lambda_\gamma, h), \text{ where } \Lambda_\gamma\cong \Lambda \text{ and } D_h \neq \textnormal{id},
\end{equation}
such that $b\notin O^\#(C)$ if $a=\id_F$. Indeed, we know that the extension of $O^\#(C)$ with the identity on $F$ is stable.

\begin{definition}\label{defin:spine}
    Let $C$ be a heart and let $(F, a)$ be a head of $C$ of order $n\in\{1,2\}$. A \textbf{spine} between $C$ and $(F,a)$ is a glue map 
    \[D_F\geq H_F \xrightarrow{\gamma} H_C\leq D_C\] 
    such that
    \begin{enumerate}
        \item $D_aH_F \leq H_F$;
        \item there exists $b\in O(C)$ such that $b\notin O^\#(C)$ if $a = \textnormal{id}_F$, $D_bH_C\leq H_C$ and $\gamma$ is $(a, b)$-equivariant;
        \item the equivariant primitive extension $(F, a)\oplus (C, b)\subseteq (\Lambda_\gamma, h)$ is such that
        \begin{enumerate}
            \item $\Lambda_\gamma \cong \Lambda$;
            \item $h\in O^+(\Lambda_\gamma)\setminus O^\#(\Lambda_\gamma)$;
            \item $(\Lambda_\gamma)_{H_\gamma}\cap\mathcal{W}^{pex}(\Lambda_\gamma) = \varnothing$ where $H_\gamma := \langle O^\#(C),\, h\rangle$.
        \end{enumerate}
    \end{enumerate}
    If it exists, we call the isometry $b$ a \textbf{companion} of $a$ along $\gamma$.
\end{definition}

It is not hard to see that the previous definitions are such that the heart $C := \Lambda_{H^\#}$ of any symplectic finite subgroup $H\leq O^+(\Lambda)$ is a heart and its head $(\Lambda^{H^\#}, a)$ is a head of $C$. 

For two lattices $L_1, L_2$ and two groups of isometries $G_i\leq O(L_i)$ ($i=1,2$), we call the pairs $(L_1, G_1)$ and $(L_2, G_2)$ \textbf{conjugate} if there exists an isometry $\psi\colon L_1\to L_2$ such that $G_2 = \psi G_1\psi^{-1}$. The following holds.

\begin{theorem}\label{th:BHP}
    Let $C$ be a heart and let $(F, a)$ be a head of $C$ of order $n\in\{1,2\}$. Let $D_F\geq H_F \xrightarrow{\gamma} H_C\leq D_C$ be a spine between $C$ and $(F,a)$, and let $b\in O(C)$ be a companion of $a$ along $\gamma$. Then
    \begin{enumerate}
        \item $H_\gamma/O^\#(C)$ is a cyclic group of order 2;
        \item $H_\gamma\leq O^+(\Lambda_\gamma)$ is symplectic and $\overline{H_\gamma} \leq O(D_{\Lambda_\gamma})$ has order 2;
        \item the definition of $H_\gamma$ is independent of the choice of $b$;
        \item the conjugacy class of $(\Lambda_\gamma, H_\gamma)$ is invariant under the action of $O(C)$ and $O(F, a)$.
    \end{enumerate}
\end{theorem}
\begin{proof}
    We let $h := a\oplus b$ where $a$ is an isometry of $F$ of order at most 2 and $b \notin O^\#(C)$ if $a = \textnormal{id}_F$.
    Since $O(D_{\Lambda_\gamma})$ has order 2, we have that $D_{h}^2$ is trivial. Thus $h^2 = \id_F\oplus b^2$ must lie in $O^\#(C)$ because $O^\#(C)$ is saturated in $O^{+,\#}(\Lambda_\gamma)$.
    By definition of $a$, we moreover know that $h$ has negative definite coinvariant sublattice, which implies that $h\in O^+(\Lambda_\gamma)$ by \Cref{neg def coinv}. 
    It follows that $H_\gamma := \langle h,\, O^\#(C)\rangle\leq O^+(\Lambda_\gamma)$ also has negative definite coinvariant sublattice, and by definition of $\gamma$, we obtain that $H_\gamma$ is symplectic (\Cref{lem: crit for bir symp eff}). 
    Moreover $O^\#(C) = H_\gamma^\#$ is normal in $H_\gamma$, so we have that $H_\gamma/O^\#(C)=\langle hO^\#(C)\rangle $ has order 2.
    
    Now, if one had chosen $b\neq b'\in O(C)$ such that $H_C$ is $D_{b'}$-stable and $\gamma \circ (D_a)_{\mid H_F} = (D_{b'})_{\mid H_C}\circ\gamma$, giving rise to an isometry $h'\in O(\Lambda_\gamma)$, then $h^{-1}h' = \textnormal{id}_F\oplus b^{-1}b'$ would act trivially on $F$ with  $D_{h^{-1}h'} = D_{h^{-1}}D_{h'}$ trivial. 
    Therefore $b' \in bO^\#(C)$ since $O^\#(C)$ is saturated in $O^{+, \#}(\Lambda_\gamma)$ and thus $\langle O^\#(C),\, h\rangle = \langle O^\#(C),\, h'\rangle$. 
    
    Finally, let $\psi\in O(F, a)$ and let $c\in O(C)$. 
    Then, for any spine 
    \[D_{F}\geq H_{F} \xrightarrow{\gamma_1} H_C\leq D_C\]
    between $C$ and $(F, a)$ with companion $b$ of $a$, the glue map
    \[D_{F}\geq D_{\psi}^{-1}H_F \xrightarrow{\gamma_2 := D_c\circ\gamma_1\circ D_{\psi}} D_cH_C\leq D_C\]
    is a spine between $C$ and $(F, a)$, and $cbc^{-1}$ is a companion of $a$ along $\gamma_2$. 
    By \cite[Corollary 1.5.2]{nikulin}, $\gamma_1$ and $\gamma_2$ indeed defines isomorphic equivariant primitive extensions and since $O^\#(C)$ is normal in $O(C)$, the corresponding resulting isometry $\psi^{-1}\oplus c\colon L_{\gamma_1}\xrightarrow{\cong}L_{\gamma_2}$ conjugates $\langle O^\#(C), a\oplus b\rangle\leq O^+(\Lambda_{\gamma_1})$ and $\langle O^\#(C), a\oplus cbc^{-1} \rangle\leq O^+(\Lambda_{\gamma_2})$. 
    The rest of the proof follows from the \Cref{defin:spine}.
\end{proof}

Hence, the definitions of hearts, heads and spines are consistent with \Cref{heart for symplectic}. The following Theorem gives a converse to \Cref{same chara have same extension data2}, and it is a stated in such a way that we can describe an effective algorithm out of it. 

\begin{theorem}\label{th:classth}
    Let $C$ be a heart, let $(F,a)$ be a head of $C$ of order $n\in\{1,2\}$ and let $S$ be the set of spines between $C$ and $(F, a)$. Then the double cosets in
    \[ O(C)\backslash S \slash O(F, a)\]
    are in bijection with the conjugacy class of pairs $(\Lambda_\gamma, H_\gamma)$ where $\Lambda_\gamma\cong \Lambda$ and $H_\gamma\leq O^+(\Lambda_\gamma)$ is a symplectic finite subgroup such that $H_\gamma^\#$ is saturated in $O^{+, \#}(\Lambda_\gamma)$, the subgroup $\#\overline{H_\gamma}\leq O(D_{\Lambda_\gamma})$ has order 2, the heart of $H_\gamma$ is isometric to $C$ and the head of $H_\gamma$ is isomorphic to $(F, a)$. 
\end{theorem}

\begin{proof}
    One directions follows from \Cref{th:BHP}, while the other is a direct consequence of \Cref{same chara have same extension data2}.
    The converse follows from  \Cref{same chara have same extension data2} and \Cref{th:BHP}.
\end{proof}

Based on \Cref{th:classth}, the strategy behind the extension approach can be summarised in the following way.
\begin{enumerate}
    \item Start with a given heart $C\subseteq \Lambda$, as classified in \Cref{propo: th 4.1 lattice}.
    \item Determine a set $\mathcal{F}_C(n)$ of representatives of isomorphism classes of heads $(F, a)$ of $C$ of given order $n=1,2$. Such pairs $(F, a)$ must satisfy that $F\cong C^\perp_\Lambda$, $a\in O(F)$ has order $n$ and $F_a$ is negative definite or trivial.
    \item For each heart $C$ and candidate head $(F, a)\in \mathcal{F}_C(n)$ of order $n=1,2$, compute and classify equivariant primitive extensions $(F,a)\oplus (C, b)\subseteq (\Lambda', h)$ where $\Lambda'\cong\Lambda$, $D_h\neq \id$ and where $b\notin O^\#(C)$ if $a=\id_F$.
\end{enumerate}
Step (3) is done by constructing spines representing each of the double cosets described in \Cref{th:classth}.

\begin{remark}
    From a technical point of view, the classification of hearts $C$ in step (1) has been completed in \Cref{sec: stable sympl iso}. The classification of lattices with isometry $(F, a)$ in step (2) has been solved algorithmically by Brandhorst and Hofmann in their work \cite{bh22}. Finally, the computation and classification of spines for step (3) can also be carried out systematically thanks to Nikulin's theory on primitive embeddings \cite{nikulin}. Both of these procedures have been implemented by the second author in \cite[QuadFormAndIsom]{oscar-book}.
\end{remark}

\subsection{Results and comments}\label{subsec: results for BHH approach}
We implement the extension procedure described in \Cref{th:classth} to the list of hearts of $\Lambda$ determined in \Cref{theo: sec 3 main}, and we obtain the following.

\begin{theorem}\label{th:table2}
    Let $\mathcal{H}$ be an $O^+(\Lambda)$-conjugacy class of symplectic finite subgroups $H\leq O^+(\Lambda)$. Then a representative of $\mathcal{H}$ is computable. Moreover, the folder    ``data" of the dataset \cite{database} contains representatives for each such conjugacy class.
\end{theorem}
\begin{proof}
    This follows from \Cref{th:BHP,th:classth}. The representatives contained in the dataset \cite{database} have been computed using  \Cref{alg:one,alg:two}, which have been implemented in Oscar \cite{oscar-book}. See \Cref{sec: algorithms} for more details on the algorithmic aspect of the theorem.
\end{proof}

From this, we can conclude our classification of finite subgroups of symplectic birational transformations for IHS manifolds of $OG10$ type.

\begin{proof}[Proof of \Cref{thm: main thm for OG10 intro}]
    The classification follows from \Cref{th:table2}. We obtain 934 pairs $(\Lambda',H)$ with $\Lambda'\cong \Lambda$ and $H^\#\leq O^{+, \#}(\Lambda')$ saturated. These are contained in the dataset \cite{database}. We extract 375 pairs $(\Lambda',H)$ with $H$ saturated in $O^+(\Lambda')$ using \Cref{rmk saturation}, which are described in Table 5 of the ancillary files. 
\end{proof}

According to the data compiled in the dataset \cite{database}, we obtain the largest cohomological actions (i.e. groups with nontrivial action on second cohomology) for IHS manifolds of known deformation type. Recall that a finite group $G\leq \Bir(X)$ is \textbf{mixed} if it contains both nontrivial symplectic and nonsymplectic birational transformations.

\begin{proposition}\label{theorem maximality}
    The largest cohomological actions for IHS manifolds of all known deformation types have order \textbf{6531840} in the symplectic case, and \textbf{39191040} in the mixed case.
\end{proposition}

\begin{proof}
    According to Tables 4 and 5 of the ancillary files, the largest symplectic finite subgroup $H\leq O^+(\Lambda)$ we have determined is isomorphic to a semidirect product $\text{PSU}(4,3)\rtimes C_2$ (Id ``163a.1"), and its order is $\#H=6531840$. Note that this group is saturated in $O^+(\Lambda)$. We prove that this group is maximal (in size) among all the finite groups of symplectic isometries of $H^2(X, \bZ)$ for $X$ an IHS manifold of known deformation type. 
    
    By similar arguments as the ones used in this paper, the maximal order in the $K3^{[n]}$ cases ($n\geq 1$) is bounded above by 491520 (\cite[Table 2]{H_hn_2016}). For an IHS manifold $X$ of $Kum_n$ type ($n\geq 2$) or $OG6$ type, the negative signature of $H^2(X, \bZ)$ is $k = 4$ or $k=5$ respectively. By the known bounds on the order of finite subgroups of $\GL_k(\bZ)$ for $k=4,5$, we know that the maximal orders in those cases is smaller than 200000.

    Let $H= \text{PSU}(4,3)\rtimes C_2\leq O^+(\Lambda).$ The invariant sublattice $F:= \Lambda^H$ has Gram matrix $\scriptscriptstyle{\begin{pmatrix}2&-1&0\\-1&2&0\\0&0&4\end{pmatrix}}$ in its standard basis. This lattice admits an order 6 isometry $a\in O(F)$ given by $$\scriptscriptstyle{\begin{pmatrix}0&-1&0\\1&1&0\\0&0&1\end{pmatrix}},$$
    with characteristic polynomial $\Phi_1\Phi_6$. Using \Cref{eq:egc}, one can show that there exists an isometry $b\in O(\Lambda_H)$ such that we have an equivariant primitive extension
    \[ (F, a)\oplus (\Lambda_H, b)\subseteq (\Lambda', c)\]
    where $\Lambda'\cong \Lambda$. Since $\Lambda_H$ is negative definite, such an isometry $b$ is effectively computable using an algorithm of Plesken--Souvignier \cite{ps97}.  The invariant sublattice $(\Lambda')^c$ has signature $(1, \ast)$ and the isometry $c\in O^+(\Lambda')$. Moreover, $N := (\ker\Phi_1(c)\Phi_6(c))^\perp_{\Lambda'}\subseteq F^\perp_{\Lambda'} = \Lambda_H$ satisfies that $N\cap \mathcal{W}^{pex}(\Lambda')$: in particular, similarly to \cite[Proposition 3.3]{bh22}, we can conclude that $c\in O^+(\Lambda')$ is effective and $\langle H, \,c\rangle$ is an effective finite subgroup of $O^+(\Lambda')$ (in the sense of \cite[Definition 3.2]{bh22}). In particular, by the Strong Torelli theorem and the surjectivity of the period map, $\langle H, \,c\rangle$ has a faithful action by birational transformations on an IHS manifolds of $OG10$ type. As before, we claim that the action of this group on cohomology has maximal order for all known deformation types. See the Notebook ``Maximal" of the dataset \cite{database} for the computational details.
\end{proof}

\begin{remark}
    The group $\text{PSU}(4,3)$ is a well-known group, and it appears to act faithfully on some other special objects. For instance, the authors in \cite{os24} show that $\text{PSU}(4,3)$ is the largest finite group acting symplectically on a supersingular K3 surface of Artin invariant 1. Another recent occurrence of this group is in \cite{yyz24} where it was shown that the automorphism group of the most symmetric sextic fourfold is a degree 2 extension of $\text{PSU}(4, 3)$ (such sextic fourfold has been known since Todd \cite{tod50}, and it is intrinsically related to the Coxeter--Todd lattice $K_{12}$, see \cite{cs93}).
\end{remark}

\section{Geometrical Interpretations}\label{sec: geo interp}
In this section, we provide geometric realisations for some of the groups listed in the classification of \Cref{sec: from stable sympl to sympl}. 
First, in \Cref{subsec: induced from a cubic fourfold}, we investigate the possible groups that can be induced from a cubic fourfold via the LSV construction of \cite{LSV}, \cite{sac2021birational}.  In particular, we prove:
\begin{theorem}\label{thm: geo realise LSV}
   Let $X$ be an IHS manifold of $OG10$ type and let $G\leq \Bir_s(X)$ be a finite group of symplectic birational transformations. Suppose that $\Lambda^G\cong U\oplus \Gamma$ holds, for some lattice $\Gamma$. Then there exists some smooth cubic fourfold $V$ and an embedding $j\colon G\hookrightarrow \Aut(V)$ such that:
    \begin{enumerate}
        \item either $G$ acts trivially on the discriminant group $D_\Lambda$, and $j(G)\leq \Aut_s(V)$;
        \item or $G=\langle G_s, \phi\rangle$ with $j(G_s)\leq \Aut_s(V)$, and $j(\phi)\in\Aut(V)\setminus\Aut_s(V)$ is antisymplectic.
    \end{enumerate}
    The pair $(G_s, \Lambda_{G_s})$ occurs in the classification of \cite{laza2019automorphisms}.

    Conversely, for any smooth cubic fourfold $V$, any LSV manifold $X_V$ associated to $V$ and any finite subgroup $G\leq \Aut(V)$ so that $[G:G_s]\leq 2$ holds, there is an embedding of $G$ into the group $\Bir_s(X_V)$ of symplectic birational transformations on $X_V$.
\end{theorem}
\begin{proof}
The proof follows from combining Propositions \ref{prop: split U}, \ref{prop: cubic to OG10} and \ref{prop: cubic to OG10 non-symp}.
\end{proof}

In \Cref{subsec: twisted LSV}, we use similar techniques of \cite{bg24} to investigate when a group of symplectic birational transformations of an IHS manifold of $OG10$ type is induced from a cubic fourfold via the twisted LSV construction of \cite{twistedLSV}. Such a group is isomorphic to a stable symplectic finite subgroup $H\leq O^{+,\#}(\Lambda)$ whose heart $\Lambda_H$ embeds primitively into the Leech lattice, as in \Cref{sec: stable sympl iso}; we list the Id's of the pairs from \Cref{tab: lovely table 45} that can be induced in this way in \Cref{cor: twisted LSV}.

Next, in \Cref{subsec: numerical k3}, we apply the criteria developed recently in \cite{felisetti2023ogrady} to investigate when a group of symplectic birational transformations of an IHS manifold of $OG10$ type is induced from an underlying $K3$ surface. Such manifolds are necessarily numerical moduli spaces; that is they are birational to a certain moduli space of sheaves on a $K3$ surface.

\subsection{LSV manifolds}\label{subsec: induced from a cubic fourfold}
Let $V\subset \bP^5$ be a smooth cubic fourfold. By the construction of \cite{LSV,sac2021birational} there exists an IHS manifold $X_V$ and a Lagrangian fibration $\pi:X_V\rightarrow (\bP^5)^{\vee}$ that compactifies the intermediate Jacobian fibration associated to $V$. Further, the manifold $X_V$ is an IHS manifold of $OG10$ type. We call such a manifold an \textbf{LSV manifold associated to V.} 

\begin{remark}
    Note that such a compactification $X_V$ may not be unique --- indeed, if the cubic fourfold $V\subset \bP^5$ is not general then uniqueness is not guaranteed. In particular, a cubic fourfold with nontrivial automorphism group is not Hodge general.
\end{remark}
We prove \Cref{thm: geo realise LSV} in a series of propositions. First, we establish the existence of cubic fourfolds with specified group actions. 

\begin{proposition}\label{prop: split U}
    Let $X$ be an IHS manifold of $OG10$ type and let $G\leq \Bir_s(X)$ be a finite group of symplectic birational transformations. Suppose that $\Lambda^G\cong U\oplus \Gamma$ holds, for some lattice $\Gamma$. Then there exists some smooth cubic fourfold $V$ and an embedding $j\colon G\hookrightarrow \Aut(V)$ such that:
    \begin{enumerate}
        \item either $G$ acts trivially on the discriminant group $D_\Lambda$, and $j(G)\leq \Aut_s(V)$;
        \item or $G=\langle G_s, \phi\rangle$ with $j(G_s)\leq \Aut_s(V)$, and $j(\phi)\in\Aut(V)\setminus\Aut_s(V)$ is antisymplectic.
    \end{enumerate}
    The pair $(G_s, \Lambda_{G_s})$ occurs in the classification of \cite{laza2019automorphisms}.
\end{proposition}

\begin{proof}
    This follows a similar strategy as \cite[Theorems 4.1 and 5.1]{marqOG10}. In what follows, we identify $G$ with a subgroup of $O(\Lambda)$ by fixing a marking of $X$.

    Let $U_1:=U$ be such that $\Lambda^G =U_1\oplus \Gamma\hookrightarrow \Lambda$, and denote by $L:=(U_1)^\perp_{\Lambda(-1)}$. 
    Then $L$ is an even lattice of signature $(20, 2)$ and (recalling that $ADE$ lattices are assumed negative definite), $$L\cong U^2\oplus E_8^2(-1)\oplus A_2(-1).$$ 
    The group $G$ restricts to $G\leq O(L)$ with $$L^G\cong \Gamma(-1) \text{ and } L_G\cong \Lambda_G(-1).$$ 
    We choose a weight 4 pure Hodge structure $H$ of $K3$ type on $L$, such that $H^{2,2}\cap L = \Lambda_G(-1)$.  
    In particular, $H^{3,1}\subset L^G$, and since $\Lambda_G\cap \markman=\varnothing$ we can apply the Global Torelli Theorem for cubic fourfolds (\cite{voisintorelli}, \cite[Prop 1.3]{ZHENG_2019}). 
    We obtain a smooth cubic fourfold $V$ with $H^4(V,\bZ)_{prim}\cong H$ as Hodge structures.

    First, we assume that $G$ acts trivially on $D_\Lambda$, and hence on $D_L.$  
    In order to conclude that $G$ embeds in $\Aut_s(V)$, we need to extend $G$ to a group of isometries of $H^4(V,\bZ)$ fixing the square of the hyperplane class $h^2\in H^4(V,\bZ)$.
    We have that $L\oplus \langle h^2\rangle\subset H^4(V,\bZ)$. Since $G$, seen as a subgroup of $O(L)$, is stable, we can extend $G$ with $\id_{\langle h^2\rangle}$ to an isometry group of $H^4(V,\bZ)$; it follows by the Torelli theorem for cubic fourfolds that $G$ embeds into $\Aut(V)$. 
    Further, in this case it follows that $G$ acts faithfully symplectically on $V$.

    Next, we assume that $G$ does not act trivially on $D_\Lambda$. 
    In other words, there exists a short exact sequence as in \Cref{short exact sequence}:
    $$1\rightarrow G^\#\rightarrow G\rightarrow \mu_2\rightarrow 1.$$ 
    We can conclude that $G=\langle G^\#, a\rangle$, where $a\in O(\Lambda)$ acts nontrivially on $D_\Lambda$.
    Here, $a$ has even order and $\Lambda^a$ splits $U_1$ as well, since $\Lambda_a$ is negative definite. 
    We can first embed $G^\#$ into $\Aut_s(V)$ as above. It remains to extend the isometry $a$.
    
    We restrict $a$ to $L$ to obtain an isometry of $L$, but acting nontrivially on the discriminant group $D_L$. As in the proof of \cite[Theorem 4.3]{marqOG10}, we can instead extend the isometry $-a\in O(L)$ to an isometry $b$ of $H^4(V,\bZ)$, acting trivially on $\langle h^2\rangle,$ and by the Torelli theorem we conclude that $b$ is induced by an automorphism of $V$, which we still denote by $b\in\Aut(V)$. However, this automorphism $b$ is now antisymplectic: the isometry $a$ acted trivially on $H^{3,1},$ hence $b$ acts by $-\textnormal{id}|_{H^{3,1}}$. The group $\langle G_s, b\rangle \cong G$ is thus a subgroup of $\Aut(V)$, satisfying the statement of the Proposition.

    Finally, since $G_s$ is a group of symplectic automorphisms of a cubic fourfold, the pair $(G_s,\Lambda_{G_s})$ is a Leech pair, and occurs in the classification of \cite{laza2019automorphisms}.
\end{proof}

Next, we investigate the birational transformations on an LSV induced from the associated cubic fourfold. It was observed in \cite[\S 3.1]{sac2021birational} (see also \cite{LPZ}) that any automorphism of the cubic fourfold $V$ induces a birational transformation of any associated LSV manifold $X_V$. Moreover, a symplectic automorphism of $V$ induces a symplectic birational transformation of $X_V$.

\begin{proposition}\label{prop: cubic to OG10}
Let $V\subset \bP^5$ be a smooth cubic fourfold, and $G\leq\Aut_s(V)$ be a finite subgroup of symplectic automorphisms of $V$. Then the group of symplectic birational transformations of any LSV manifold $X_V$ associated to $V$ contains a finite subgroup isomorphic to $G$.
\end{proposition}
\begin{proof}
    This is essentially the argument of \cite[\S 3.1]{sac2021birational} and \cite[Theorem 4.3]{marqOG10} --- we include the proof for completeness.

The automorphism group $G$ acts on the universal family of hyperplane sections of $V$, and thus on the Donagi--Markman fibration $\calJ_U\rightarrow U$, where $U\subset (\bP^5)^{\vee}$ parametrises smooth hyperplane sections of $V$.
On a compactification, we obtain a group of birational transformations $G\leq \Bir(X_V)$ that leaves the sublattice $\langle \Theta, \pi^*\calO_{\bP^5}(1)\rangle\cong U$ invariant (here $\Theta$ is the relative theta divisor). 
Since $G$ acts symplectically on $V$, the group $G\leq\Bir(X_V)$ acts symplectically on $X_V$ by \cite[Lemma 3.2]{sac2021birational}. 
\end{proof}

Due to the presence of the Lagrangian fibration $\pi$ which always admits a section, an LSV manifold $X_V$ always admits a birational involution $\tau\in \Bir(X_V)$ that acts  by $x\mapsto -x$ on the smooth fibers of $\pi$.

\begin{remark}
    Note that $\tau\in\Bir(X_V)$ is moreover antisymplectic since it preserves the Lagrangian fibration and acts by $-1$ on the $H^1$ of a smooth fiber of $\pi$ (compare with \cite[Proposition 3.1]{FMOS}).
\end{remark}

In \cite{marqOG10}, we prove that an antisymplectic involution $\phi\in \Bir(X_V)$ induced from an antisymplectic involution of the underlying cubic fourfold $V$ could be composed with $\tau$, producing a symplectic birational involution $\phi\circ\tau\in \Bir_s(X_V)$ which is nonstable. 

\begin{remark}
    Note that $\tau\in \Bir(X_V)$ commutes with all transformations induced from $\Aut(V)$. Indeed, $\tau$ acts trivially on the base $(\bP^5)^\vee$, whereas any automorphism induced from $f\in \Aut(V)$ will either permute fibers, or induce an automorphism of an invariant fiber that will commute with $\tau$.
\end{remark}

\begin{proposition}\label{prop: cubic to OG10 non-symp}
    Let $V\subset \mathbb{P}^5$ be a smooth cubic fourfold, and $G\leq \Aut(V)$ be a finite subgroup of automorphisms whose symplectic subgroup $G_s := G\cap\Aut_s(V)$ has index 2 in $G$. Then the group of symplectic birational transformations of any LSV manifold $X_V$ associated to $V$ contains a finite subgroup isomorphic to $G$.
\end{proposition}
\begin{proof}
    The proof is similar to the proof of \Cref{prop: cubic to OG10}, however it is done in two steps. We have an exact sequence
    $$1\rightarrow G_s\rightarrow G\rightarrow \mu_2\rightarrow 1,$$ where $G_s\leq \Aut_s(V).$ We apply \Cref{prop: cubic to OG10} to the group $G_s$, and see that $G_s\leq \Bir_s(X_V)$ for any LSV manifold $X_V$ associated to $V$.
    Next, we let $g\in G\setminus G_s$ such that the quotient $G/G_s$ is generated by $gG_s$. Then $g$ has even order, and we denote $\widehat{g}\in\Bir(X_V)$ the antisymplectic transformation induced by $g$. Then $\tau\circ\widehat{g}$ is a symplectic birational transformation of $X_V$ with order the order of $g$. We can conclude that $(\tau\circ\widehat{g})^2=\widehat{g}^2$ lies in $G_s\leq \Bir(X_V)$ and $G\cong\langle G_s, \tau\circ\widehat{g}\rangle\leq \Bir(X_V)$.
\end{proof}

 We see that as a consequence of \Cref{prop: cubic to OG10 non-symp}, the LSV construction allows to realise larger groups using the extra involution $\tau$.
We illustrate the phenomenon described in \Cref{prop: cubic to OG10 non-symp} with the two following examples:

\begin{example}[The Clebsch cubic]\label{ex: clebsch}
    Let $V=\{x_0^3+x_1^3+x_2^3+x_3^3+x_4^3+x_5^3-(x_0+x_1+x_2+x_3+x_4+x_5)^3=0\}\subset \bP^5$ be the Clebsch diagonal cubic fourfold.  This has $\Aut(V)=S_7$, with group of symplectic automorphisms $H=A_7$. In particular, $\Aut(V)/H\cong \bZ/2$, generated by the antisymplectic involution $\phi_1$ (in the notation of \cite{marquand2022cubic}) given in these coordinates by exchanging $x_0, x_1$. 

    Let $X_V$ be an associated LSV manifold and let $\tau$ be the LSV involution. By \Cref{prop: cubic to OG10 non-symp}, we see that $S_7\cong \langle \widehat{H}, \tau\circ\widehat{\phi_1}\rangle \leq\Bir_s(X_V)$. This case appears under the Id ``108a.1" in the dataset \cite{database} (see also Tables 4 and 5 in the ancillary files). 
    
    As a further remark, since $\tau\in \Bir(X_V)$ is a nonsymplectic involution, the group $\langle \widehat{H}, \tau, \widehat{\phi_1}\rangle$ is a subgroup of $\Bir(X_V)$ with mixed action, containing $S_7\leq \Bir_s(X_V)$ as a subgroup of index two.
\end{example}

\begin{example}[The Fermat cubic]\label{ex: fermat}
    Let $V=\{x_0^3+x_1^3+x_2^3+x_3^3+x_4^3+x_5^3=0\}\subset\bP^5$ be the Fermat cubic fourfold. By \cite{laza2019automorphisms}, the group of symplectic automorphisms of $V$ is isomorphic to $H= C_3^4\rtimes A_6$, and the quotient $\Aut(V)/H=\langle \phi_1H, \psi H\rangle$ has order $6$. Here, $\phi_1$ is the antisymplectic involution exchanging $x_0,x_1$, and $\psi$ is the order 3 nonsymplectic automorphism given by $x_0\mapsto \zeta_3 x_0$. 

    Let $X_V$ be an associated LSV manifold, and let $\tau$ be the LSV involution. By \Cref{prop: cubic to OG10 non-symp}, we have that $C_3^4\rtimes S_6\cong\langle \widehat{H}, \tau\circ\widehat{\phi_1}\rangle\leq\Bir_s(X_V)$. This case appears under the Id ``101a.1" in the dataset \cite{database} (see also Tables 4 and 5 in the ancillary files).
    
    Again, $\tau\in \Bir(X_V)$ is an antisymplectic involution, and $\psi\in \Aut(V)$ induces a nonsymplectic birational transformation $\widehat{\psi}\in \Bir(X_V)$ of order 3. The group $G':=\langle \widehat{H}, \tau, \widehat{\phi_1}, \widehat{\psi} \rangle$ is a subgroup of the full birational group $\Bir(X_V)$ with mixed action, and $G'$ contains $\langle \widehat{H}, \tau\circ\widehat{\phi_1}\rangle$ as a subgroup of index $6$ acting symplectically on $X_V$. The group $G'$ has order $349920=2\cdot174960$ and it is the largest finite group acting faithfully by birational transformations on IHS manifolds of $OG10$ type we have been able to geometrically realise.
    
    Note that the subgroup $\langle H, \phi_1. \psi\rangle$ of order 174960 acts on the Fano variety of lines $F(V)$ as shown in \cite{wawak}. The author realises this group as the largest finite group acting faithfully by regular automorphisms on an IHS manifold of $K3^{[2]}$ type. This was also obtained independently by \cite{comparin2023irreducible}.
\end{example}

\subsection{Twisted LSV manifolds}\label{subsec: twisted LSV}

Let $V\subset \bP^5$ be a smooth cubic fourfold. 
By the construction of \cite{twistedLSV}, there exists another IHS manifold $X_V^t$ of $OG10$ type, equipped with a Lagrangian fibration $\pi^t: X_V^t\rightarrow (\bP^5)^\vee$ whose fiber over a general hyperplane section $Y_H:=V\cap H$ is the torsor $\mathrm{Jac}^1(Y_H)$ parametrising degree 1 cycles. 
We call such a manifold $X_V^t$ a \textbf{twisted LSV manifold.} Any automorphism of $V$ induces a birational transformation of $X_V^t$ acting trivially on the discriminant group of $\Lambda$ (see for example \cite[Remark 5.4]{bg24}).
In \cite{bg24}, the authors investigate when an IHS manifold of $OG10$ type is birational to a twisted LSV manifold, and use their criterion to investigate when a nonsymplectic automorphism is induced from a cubic fourfold (in a similar manner to \ref{prop: cubic to OG10}). By a small adaption of the arguments in \cite[Proposition 5.2]{bg24}, one obtains the following proposition.

\begin{proposition}\label{prop: twisted LSV criteria}
    Let $X$ be an IHS manifold of $OG10$ type, and $G\leq \Bir_s(X)$ be a finite group with $H:=\eta_*(G)$ acting trivially $D_\Lambda$. Suppose that there is a primitive embedding $U(3)\hookrightarrow \Lambda^H$ such that the composition $U(3)\hookrightarrow \Lambda$ has glue domain $\bZ/3\bZ$. Then there exists a smooth cubic fourfold $V$ with $G\leq \Aut_s(V)$ such that $(H, \Lambda_H)$ occurs in the classification of \cite{laza2019automorphisms}. 
\end{proposition}
\begin{proof}
    Consider the lattice $L:=U(3)^\perp\subset \Lambda;$ this lattice is isometric to $U^2\oplus E_8(-1)^2\oplus A_2(-1)$. Note that $H:=\eta_*(G^\#)$ acts on $L$ with $L_H\cong \Lambda_H(-1)$. We follow the proof of \Cref{prop: split U} to obtain a smooth cubic fourfold $V$ with $H^4(V,\bZ)_{prim}$ Hodge isomorphic to the lattice $L$ with induced Hodge structure. Since $H$ acts trivially on the discriminant group, we extend $H$ to an isometry group of $H^4(V,\bZ)_{prim}$ as in \Cref{prop: split U} and conclude that $G\leq\Aut_s(V).$ It follows that the pair $(H, \Lambda_H)$ occurs in the classification of \cite{laza2019automorphisms}.
\end{proof}

\begin{corollary}\label{cor: twisted LSV}
    A symplectic finite subgroup $H\leq O^+(\Lambda)$ can be realised by an induced action on a twisted LSV manifold associated to a cubic fourfold $V$ if and only if $H$ is stable, and the \textnormal{Id} number, as listed in \Cref{tab: lovely table 45}, is one of the following:
    \begin{align*}
        \textnormal{Id}\in \{&1, 2, 3, 4b, 7b, 9, 13, 15b, 18b, 19b, 20, 29b, 31, 35b, 39b, 44b, 46b, 47c,\\& 52, 53, 55, 68b, 72b, 77, 82b, 84,85b, 87,101b, 108b, 109b, 119, 120, 128b \}.
    \end{align*} 
\end{corollary}
\begin{proof}
    Let $X$ be an IHS manifold of $OG10$ type with $G\leq \Bir_s(X)$ such that $H:=\eta_*(G)$ is one of the Ids above. Then by \Cref{prop: twisted LSV criteria} there exists a cubic fourfold $V$ with $G\leq \Aut_s(V).$ One then applies \cite[Proposition 5.6]{bg24} to obtain a twisted LSV manifold $X_V$ associated to $V$ with a compatible action of $G.$

    Conversely, if $H\leq O(\Lambda)$ is induced on a twisted LSV by a cubic fourfold it must act trivially on the discriminant group, and is thus stable. The only groups $H$ listed in \Cref{tab: lovely table 45} that satisfy the assumptions of \Cref{prop: twisted LSV criteria} are those listed above, verified by direct computation.
\end{proof}

\begin{remark}
Let $V\subset \bP^5$ be the Fermat cubic, and let $G=\Aut_s(V)\cong C_3^4\rtimes A_6.$ Then both $X_V$ and $X_V^t$ inherit a group of symplectic birational automorphisms isomorphic to $G$, but the action on the second cohomology is different. Namely, the action on $H^2(X_V,\bZ)$ is given by the entry 101a in \Cref{tab: lovely table 45}, whereas the action on $H^2(X_V^t, \bZ)$ is given by 101b. We point out that $X_V$ and $X_V^t$ are birational \cite[Remark 5.3]{bg24}, but the pairs $(X_V, G)$ and $(X_V^t,G)$ are not birational conjugate.

\end{remark}

\subsection{Numerical moduli spaces}\label{subsec: numerical k3}

Using the close connection between cubic fourfolds and manifolds of $OG10$ type, we have geometrically realised certain group actions of symplectic birational transformations. Of course, it is natural to ask the same question using automorphisms of $K3$ surfaces.

In \cite{felisetti2023ogrady}, the authors develop a criterion for when an IHS manifold of $OG10$ type is birational to a moduli space of sheaves on a $K3$ surface. Further, they provide another criterion for when a group of birational transformations of such a manifold is induced by a group of automorphisms of an underlying $K3$ surface. In this subsection, we will apply these criteria to our classification of saturated birational effective groups $G$. We begin by recalling the relevant definitions of \cite{felisetti2023ogrady}, and refer to them for more details.

Let $S$ be a $K3$ surface. Recall that the \textbf{Mukai lattice}  $\tilde{H}(S,\bZ)= H^0(S,\bZ)\oplus H^2(S,\bZ)\oplus H^4(S,\bZ)$. We have that $\tilde{H}(S,\bZ)\cong \Lambda_{24}:=U^4\oplus E_8^2$.

\begin{definition}\cite[Definition/Theorem 3.8]{felisetti2023ogrady}
    A \textbf{numerical moduli space} of $OG10$ type is a marked IHS manifold of $OG10$ type $(X,\eta)$ such that there exists a primitive class $\sigma\in H^{1,1}(X)$ with $\sigma^2=-6$ and $\divi_{\Lambda}(\eta(\sigma))=3$, and the Hodge embedding $\eta(\sigma)^\perp_\Lambda\hookrightarrow \Lambda_{24}$ embeds a copy of $U$ in $\Lambda_{24}^{1,1}$ as a direct summand.

    Equivalently, a numerical moduli space is an IHS manifold $X$ of $OG10$ type that is birational to $\tilde{M}_V(S,\Theta)$ for some $K3$ surface $S,$ a Mukai vector $v=2w$ where $w$ is primitive of square $2$, and a $v$-generic polarisation $\theta.$
\end{definition}
\begin{definition}
    Let $(X, \eta)$ be a marked IHS manifold of $OG10$ type and $G\leq\Bir(X)$ a finite subgroup. We say the group $G$ is:
    \begin{itemize}
        \item \textbf{$K3$-induced} if there exists a $K3$ surface $S$ with an injective group homomorphism $i:G\hookrightarrow \Aut(S)$, a $G$-invariant Mukai vector $v\in\tilde{H}(S,\bZ)^G$ and a $v$-generic polarisation $\theta$ on $S$ such that the action induced by $G$ on $\tilde{M}_v(S,\theta)$ via $i$ coincides with the given action of $G$ on $X$.
        \item \textbf{numerically $K3$-induced} if there exists a $G$-invariant class $\sigma\in \NS(X)$ with that $\sigma^2=-6$ and $\divi_{\Lambda}(\eta(\sigma))=3$ such that given the Hodge embedding $\eta(\sigma)_\Lambda^\perp\hookrightarrow \Lambda_{24}$, the induced action of $G$ on $\Lambda_{24}$ is such that the $(1,1)$-part of $\Lambda_{24}^G$ contains $U$ as a direct summand.
    \end{itemize}
\end{definition}
Notice that if a group $G\leq \Bir_s(X)$ is $K3$-induced, then its action on the discriminant group of $H^2(X,\bZ)$ is necessarily trivial. 

Let $\mathcal{S}$ denote the set of stable symplectic sublattices $C$ of the Leech lattice $\mathbb{L}$ which embeds primitively into the $K3$ lattice $\Lambda_{K3} := U^3\oplus E_8^2$ (i.e., those that occur in \cite[Table 10.2]{hashimoto}).
\begin{theorem}
    Let $(X, \eta)$ be a marked IHS manifold of $OG10$ type. Assume that $(X, \eta)$ is a numerical moduli space. Then a finite subgroup $G\leq\Bir_s(X)$ is $K3$-induced if and only if $\Lambda_{G}$ is isometric to a lattice in the set $\mathcal{S}.$ 
\end{theorem}
\begin{proof}
    If $G\leq \Bir_s(X)$ is a $K3$-induced group of symplectic birational transformations, by definition there exists a $K3$ surface $S$ with an injective group homomorphism $i:G\hookrightarrow \Aut(X)$. We use our classification of saturated groups of symplectic birational transformations. As the action of a $K3$-induced automorphism on the discriminant group is trivial, we can restrict to our classification of finite groups of stable symplectic isometries as in \Cref{sec: stable sympl iso}. 

    If an action is $K3$-induced, then it is numerically $K3$-induced \cite[Proposition 6.7]{felisetti2023ogrady} and in particular there exists a primitive invariant class $\sigma\in H^{1,1}(X,\bZ)$ with $\sigma^2=-6$ and divisibility three; i.e $\eta(\sigma)\in\Lambda^H$, where $H := \eta_\ast(G)$. After embedding $\eta(\sigma)^\perp_\Lambda\hookrightarrow \Lambda_{24}$ and extending the action of $H$, we obtain that $\Lambda_H = (\Lambda_{24})_H$. In particular, since $G$ is numerically $K3$-induced, we know that $H$ acts on $\Lambda_{K3} = U^\perp_{\Lambda_{24}}$ in such a way that $(\Lambda_{K3})_H = (\Lambda_{24})_H$. Hence $\Lambda_H$ is isometric to a lattice in $\mathcal{S}$. 

    For the converse, we let $G$ be a finite group of symplectic birational transformations of $X$ such that $\Lambda_H$ is isometric to a lattice in $\mathcal{S}$, where $H=\eta_*(G)$.
     We prove that $G$ is numerically $K3$-induced. Since by assumption $(X, \eta)$ is a numerical moduli space, there exists a class $\sigma\in \NS(X)$ with $\sigma^2=-6$, and $\divi_{\Lambda}(\eta(\sigma))=3$ such that, given the Hodge embedding $\eta(\sigma)_\Lambda^\perp\hookrightarrow \Lambda_{24}$, the lattice $\Lambda_{24}^{1,1}$ contains $U$ as a direct summand. Such a copy of $U$ must be contained in $\Lambda^H$, since $\Lambda_H$ is negative definite. Notice also that $\eta(\sigma)\in\markman(X)$; it follows from \Cref{lem: crit for bir symp eff} that $\eta(\sigma)\not\in\Lambda_H$. Thus $\eta(\sigma)$ must lie in $\Lambda^H$; in other words, it is invariant, and we see that $G$ is numerically $K3$-induced. The result follows from \cite[Theorem 6.8]{felisetti2023ogrady}
\end{proof}
\newpage

\appendix

\section{Exceptional stable symplectic sublattices of \texorpdfstring{$\bB$}{B}}\label{appendix: exceptional stable}
Each entry in \Cref{tab: exceptional} corresponds to an abstract isometry class of exceptional stable symplectic sublatticse $C$ of $\bB$ without $(-2)$-vectors and of rank at most 21. For each entry we give:
\begin{itemize}
    \item the rank of the lattice $C$;
    \item the length $l(D_C)$ of the discriminant group $D_C$ of $C$;
    \item the genus $g(C)$ of the lattice $C$;
    \item the order of the stable subgroup $O^\#(C)$ of isometries of $C$;
    \item a description of the group $O^\#(C)$, its Id in the Small Group Library \cite{small_grp_lib}, or nothing (\textemdash) if none of the previous are available.
\end{itemize}
Note that some of the lattices represented in \Cref{tab: exceptional} are in the same genus and have the same stable subgroup of isometries. However, the entries represent pairwise nonisometric lattices.

{
\small\centering\setlength{\tabcolsep}{1pt}
\renewcommand\arraystretch{1.5}
\rowcolors{1}{lightgray!40!white}{white}
\begin{longtable}{ccccc|ccccc}

\rowcolor{white}\caption{Exceptional stable symplectic sublattices of $\bB$ without $(-2)$-vectors}\label{tab: exceptional}\\
 $\text{rk}(C)$&$l(D_C)$&$g(C)$&$\#O^\#(C)$&$O^\#(C)$&$\text{rk}(C)$&$l(D_C)$&$g(C)$&$\#O^\#(C)$&$O^\#(C)$\\
\hline
\endfirsthead

\rowcolor{white}\caption[]{Exceptional stable symplectic sublattices of $\bB$ without $(-2)$-vectors (continued)}\\
 $\text{rank}(C)$&$l(D_C)$&$g(C)$&$\#O^\#(C)$&$O^\#(C)$&$\text{rank}(C)$&$l(D_C)$&$g(C)$&$\#O^\#(C)$&$O^\#(C)$\\
\hline
\endhead
16&10&$\II_{(0,16)}2^{10}_0$&2&$C_2$&20&6&$\II_{(0,20)}2^24^4_4$&128&$C_2\times D_8^2$\\
16&10&$\II_{(0,16)}2^84^2$&8&$C_2^3$&20&6&$\II_{(0,20)}2^24^4_4$&128&[128, 1135]\\
16&10&$\II_{(0,16)}2^{10}$&32&$C_2^5$&20&6&$\II_{(0,20)}2^24^4_4$&128&[128, 2216]\\
17&9&$\II_{(0,17)}2^88^1_7$&16&$C_2^4$&20&6&$\II_{(0,20)}2^63^2$&192&$C_2^3\times S_4$\\
18&7&$\II_{(0,18)}3^{-7}$&9&$C_3^2$&20&6&$\II_{(0,20)}2^{-2}4^{-4}_0$&256&[256, 29598]\\
18&8&$\II_{(0,18)}2^4_64^4$&4&$C_2^2$&20&6&$\II_{(0,20)}2^24^4_4$&256&[256, 56089]\\
18&8&$\II_{(0,18)}2^4_64^4$&8&$C_2^3$&20&6&$\II_{(0,20)}2^44^{-1}_38^{-1}_5$&256&[256, 53380]\\
18&8&$\II_{(0,18)}2^4_64^4$&8&$C_2\times C_4$&20&6&$\II_{(0,20)}2^{-4}4^{-2}$&768&[768, 1090235]\\
18&8&$\II_{(0,18)}2^{-4}4^{-4}_2$&16&$C_2^4$&20&6&$\II_{(0,20)}2^{-4}4^{-2}$&4096&\textemdash\\
18&8&$\II_{(0,18)}2^{-4}4^{-4}_2$&16&$C_2\times D_8$&21&4&$\II_{(0,21)}2^{-3}_73^4$&18&$C_3\times S_3$\\
18&8&$\II_{(0,18)}2^{-6}4^{-2}_2$&32&$C_2^5$&21&5&$\II_{(0,21)}2^2_24^1_18^2_0$&16&$C_2\times D_8$\\
18&8&$\II_{(0,18)}2^64^2_6$&32&$C_2^5$&21&5&$\II_{(0,21)}2^3_53^5$&18&$C_3\rtimes S_3$\\
18&8&$\II_{(0,18)}2^{-6}4^{-2}_2$&64&$C_2^3\times D_8$&21&5&$\II_{(0,21)}2^3_78^{-2}$&24&$C_2\times A_4$\\
19&7&$\II_{(0,19)}2^5_58^2$&8&$D_8$&21&5&$\II_{(0,21)}2^2_24^{-3}_1$&32&[32, 44]\\
19&7&$\II_{(0,19)}2^24^5_5$&8&$D_8$&21&5&$\II_{(0,21)}2^24^{-1}_58^{-2}_2$&32&$C_2^2\wr C_2$\\
19&7&$\II_{(0,19)}2^24^5_5$&16&$C_2^4$&21&5&$\II_{(0,21)}2^24^{-1}_38^{-2}_4$&32&$C_8\rtimes C_2^2$\\
19&7&$\II_{(0,19)}2^4_64^{-2}8^{-1}_3$&16&$C_2^4$&21&5&$\II_{(0,21)}2^{-3}_34^1_116^{-1}_3$&32&$C_8\rtimes C_2^2$\\
19&7&$\II_{(0,19)}2^{-4}_24^{-2}8^1_7$&16&$C_2\times D_8$&21&5&$\II_{(0,21)}2^4_48^1_7$&32&$2^{1+4}_-$\\
19&7&$\II_{(0, 19)}2^{-2}4^{-5}_1$&16&$C_2\times D_8$&21&5&$\II_{(0,21)}2^24^1_18^2_2$&64&$D_8^2$\\
19&7&$\II_{(0,19)}2^{-4}4^{-2}8^1_1$&32&$C_2^2\times D_8$&21&5&$\II_{(0,21)}2^{-2}_24^2_68^1_7$&64&$D_8^2$\\
19&7&$\II_{(0,19)}2^{-4}4^{-2}_28^1_7$&32&$C_2^2\times D_8$&21&5&$\II_{(0,21)}4^5_3$&64&$C_2^3\times D_8$\\
19&7&$\II_{(0,19)}2^44^{-2}_68^{-1}_3$&32&$C_2^2\wr C_2$&21&5&$\II_{(0,21)}2^24^{-1}_38^{-2}_4$&64&$C_2^3\times D_8$\\
19&7&$\II_{(0,19)}2^{-6}8^{-1}_5$&64&$C_2^6$&21&5&$\II_{(0,21)}2^4_616^{-1}_5$&64&$C_2\times C_8\rtimes C_2^2$\\
19&7&$\II_{(0,19)}2^44^3_5$&64&$C_2\times C_2^2\wr C_2$&21&5&$\II_{(0,21)}2^{-4}_216^{-1}_5$&64&[64, 124]\\
19&7&$\II_{(0,19)}2^{-6}8^{-1}_5$&128&$C_2^2\times C_2^2\wr C_2$&21&5&$\II_{(0.21)}2^{-2}_44^{-1}_38^{-2}_4$&64&[64, 134]\\
19&7&$\II_{(0,19)}2^44^3_5$&128&$C_2^2\times 2^{1+4}_+$&21&5&$\II_{(0,21)}2^3_78^{-2}$&64&[64, 134]\\
20&5&$\II_{(0,20)}5^5$&5&$C_5$&21&5&$\II_{(0,21)}2^3_78^{-2}$&64&[64, 211]\\
20&5&$\II_{(0,20)}3^{-4}9^1$&108&[108, 40]&21&5&$\II_{(0,21)}2^{-4}_28^1_73^1$&96&$C_2^2\times S_4$\\
20&6&$\II_{(0,20)}2^{-2}4^{-4}_0$&4&$C_4$&21&5&$\II_{(0,21)}2^48^1_73^2$&96&$C_2^2\times S_4$\\
20&6&$\II_{(0,20)}2^2_64^{-4}_6$&4&$C_2^2$&21&5&$\II_{(0,21)}2^{-2}4^1_78^{-2}_4$&128&$D_8\wr C_2$\\
20&6&$\II_{0,20)}2^{-6}_43^2$&6&$C_6$&21&5&$\II_{(0,21)}2^3_78^{-2}$&128&[128, 2317]\\
20&6&$\II_{(0,20)}2^{-4}_04^2$&16&$C_2\times Q_8$&21&5&$\II_{(0,21)}2^24^{-3}_13^{-1}$&192&$C_2\times C_2^2\rtimes S_4$\\
20&6&$\II_{(0,20)}2^{-4}_28^2_6$&16&$C_2\times D_8$&21&5&$\II_{(0,21)}2^24^{-2}_48^{-1}_3$&256&[256, 16883]\\
20&6&$\II_{(0,20)}2^4_68^2_6$&16&$C_4\bigcirc D_8$&21&5&$\II_{(0,21)}2^24^{-2}_28^{-1}_5$&256&[256, 16888]\\
20&6&$\II_{(0,20)}2^{-5}_116^1_7$&16&$C_4\bigcirc D_8$&21&5&$\II_{(0,21)}4^5_3$&256&[256, 25886]\\
20&6&$\II_{(0,20)}2^3_54^1_78^2$&16&$D_{16}$&21&5&$\II_{(0,21)}2^24^2_28^1_1$&256&[256, 51978]\\
20&6&$\II_{(0, 20)}2^{-2}3^{-6}$&18&$C_3\times S_3$&21&5&$\II_{(0,21)}2^24^2_28^1_1$&384&[384, 18235]\\
20&6&$\II_{(0,20)}2^24^3_58^1_7$&32&$C_2^5$&21&5&$\II_{(0,21)}2^24^3_3$&384&[384, 20089]\\
20&6&$\II_{(0,20)}2^{-2}4^3_78^{-1}_5$&32&$C_2^2\times D_8$&21&5&$\II_{(0,21)}2^{-2}4^{-3}_7$&384&[384, 20089]\\
20&6&$\II_{(0,20)}2^{-2}_64^{-4}_2$&32&$C_2^2\times D_8$&21&5&$\II_{(0,21)}2^48^{-1}_53^{-1}$&384&[384, 20100]\\
20&6&$\II_{(0,20)}2^4_68^{-2}_2$&32&$C_2^2\times D_8$&21&5&$\II_{(0,21)}2^{-2}4^{-2}_68^1_1$&512&\textemdash\\
20&6&$\II_{(0,20)}2^4_68^2_6$&32&[32, 34]&21&5&$\II_{(0,21)}2^24^3_3$&512&\textemdash\\
20&6&$\II_{(0, 20)}2^24^4_4$&48&$C_2\times S_4$&21&5&$\II_{(0,21)}2^{-4}8^1_13^1$&576&$S_4^2$\\
20&6&$\II_{(0,20)}2^24^4_4$&64&$C_2\times C_2^2\wr C_2$&21&5&$\II_{(0,21)}2^48^{-1}_53^{-1}$&768&[768, 1089108]\\
20&6&$\II_{(0,20)}2^24^4_4$&64&$C_2\times C_2^2\wr C_2$&21&5&$\II_{(0,21)}2^{-4}4^1_13^1$&768&[768, 1090213]\\
20&6&$\II_{(0,20)}2^24^4_4$&64&$C_2\wr C_2^2$&21&5&$\II_{(0,21)}2^{-2}4^{-3}_7$&1536&\textemdash\\
20&6&$\II_{(0,20)}2^24^{-3}_38^{-1}_5$&64&$C_2\wr C_2^2$&21&5&$\II_{(0,21)}2^24^3_3$&1536&\textemdash\\
20&6&$\II_{(0, 20)}2^{-4}8^{-2}_4$&64&$D_8^2$&21&5&$\II_{(0,21)}2^{-4}8^{-1}_3$&4608&\textemdash\\
20&6&$\II_{(0,20)}2^{-4}4^2_23^1$&96&$C_2^2\times S_4$&21&5&$\II_{(0,21)}2^{-2}4^{-3}_7$&8192&\textemdash\\
20&6&$\II_{(0,20)}2^{-4}4^2_63^{-1}$&96&$C_2\times C_2^2\rtimes A_4$&21&5&$\II_{(0,21)}2^{-4}8^{-1}_3$&24576&\textemdash\\
20&6&$\II_{(0,20)}2^{-4}4^1_78^{-1}_5$&128&$C_2^3\wr C_2$&&&&&\\
\hline
\end{longtable}}

\section{Isometry classes of potential exceptional hearts of \texorpdfstring{$\Lambda$}{Λ} }\label{App: exceptional hearts}
We display the Gram matrices for the exceptional stable symplectic sublattices of $\bB$ without $(-2)$-vectors which could potentially embed primitively into the lattice $\Lambda:=U^3\oplus E_8^2\oplus A_2$. These are the only lattices represented in \Cref{tab: exceptional} for which $\text{rank}(-)+l(D_-) \leq 25$.

\begin{center}
\resizebox{\textwidth}{!}{$E_{18} := \scriptscriptstyle{\begin{pmatrix}
-4 &  2 &  2 & -2 & -2 &  2 & -2 &  2 &  2 &  0 &  0 &  0 &  0 &  0 &  2 &  2 & -1 &  1 \\
 2 & -4 & -2 &  0 &  1 & -1 &  2 & -2 & -2 & -1 & -1 &  1 &  1 & -1 & -1 & -2 & -1 & -2 \\
 2 & -2 & -4 &  1 &  0 & -2 &  2 & -2 & -2 & -1 & -1 & -1 &  1 & -1 &  0 & -2 &  0 & -1 \\
-2 &  0 &  1 & -4 & -2 &  0 &  0 &  0 &  0 &  1 &  1 &  0 & -1 &  1 &  0 &  2 &  0 & -1 \\
-2 &  1 &  0 & -2 & -4 &  1 &  0 &  0 &  0 &  1 &  1 & -1 & -1 &  1 &  2 &  2 & -1 &  1 \\
 2 & -1 & -2 &  0 &  1 & -4 &  2 & -2 & -2 &  0 &  0 &  0 &  0 &  0 & -2 & -1 &  1 & -1 \\
-2 &  2 &  2 &  0 &  0 &  2 & -4 &  2 &  2 & -1 &  1 & -1 & -1 &  1 &  0 &  0 & -1 &  2 \\
 2 & -2 & -2 &  0 &  0 & -2 &  2 & -4 & -1 &  1 & -1 &  1 & -1 &  1 & -1 & -1 &  0 & -1 \\
 2 & -2 & -2 &  0 &  0 & -2 &  2 & -1 & -4 &  0 &  1 & -1 &  0 &  0 & -1 & -1 &  0 & -1 \\
 0 & -1 & -1 &  1 &  1 &  0 & -1 &  1 &  0 & -4 & -1 & -1 &  2 & -2 &  0 & -2 & -1 &  0 \\
 0 & -1 & -1 &  1 &  1 &  0 &  1 & -1 &  1 & -1 & -4 &  2 &  2 & -1 &  1 & -1 &  0 & -1 \\
 0 &  1 & -1 &  0 & -1 &  0 & -1 &  1 & -1 & -1 &  2 & -4 &  0 &  0 &  0 &  0 &  0 &  1 \\
 0 &  1 &  1 & -1 & -1 &  0 & -1 & -1 &  0 &  2 &  2 &  0 & -4 &  2 & -1 &  1 &  0 &  1 \\
 0 & -1 & -1 &  1 &  1 &  0 &  1 &  1 &  0 & -2 & -1 &  0 &  2 & -4 &  1 & -1 &  0 & -1 \\
 2 & -1 &  0 &  0 &  2 & -2 &  0 & -1 & -1 &  0 &  1 &  0 & -1 &  1 & -4 & -1 &  1 & -1 \\
 2 & -2 & -2 &  2 &  2 & -1 &  0 & -1 & -1 & -2 & -1 &  0 &  1 & -1 & -1 & -4 &  0 & -1 \\
-1 & -1 &  0 &  0 & -1 &  1 & -1 &  0 &  0 & -1 &  0 &  0 &  0 &  0 &  1 &  0 & -4 &  1 \\
 1 & -2 & -1 & -1 &  1 & -1 &  2 & -1 & -1 &  0 & -1 &  1 &  1 & -1 & -1 & -1 &  1 & -4 
\end{pmatrix}}$}
\end{center}

\begin{center}
\resizebox{\textwidth}{!}{$E_{20} := \scriptscriptstyle{\begin{pmatrix}
    -4 & -2 & -1 & 1 & -1 & -1 & -1 & -1 & -1 & -1 & 1 & -1 & -2 & -2 & -1 & -2 & 1 & 1 & -1 & 0 \\ 
-2 & -4 & 1 & 2 & 1 & -2 & 1 & 1 & 1 & 1 & -1 & -2 & -1 & -2 & -1 & 0 & -1 & -1 & 0 & -1 \\ 
-1 & 1 & -4 & -2 & -2 & 2 & 0 & -2 & -2 & -2 & 0 & 1 & 0 & 0 & 1 & -1 & 0 & 2 & -2 & -1 \\ 
1 & 2 & -2 & -4 & -2 & 2 & -1 & 0 & -2 & -2 & 1 & 2 & 0 & 0 & 0 & 1 & -1 & 2 & -1 & -1 \\ 
-1 & 1 & -2 & -2 & -4 & 0 & 0 & 0 & -1 & -1 & 2 & 2 & 0 & 0 & -1 & -1 & 0 & 2 & -2 & 0 \\ 
-1 & -2 & 2 & 2 & 0 & -4 & 1 & 2 & 2 & 2 & 1 & -1 & -1 & -1 & -2 & 0 & 0 & -1 & 0 & 1 \\ 
-1 & 1 & 0 & -1 & 0 & 1 & -4 & 0 & -1 & -1 & 1 & 1 & -2 & -1 & -1 & 0 & 1 & 1 & 0 & 1 \\ 
-1 & 1 & -2 & 0 & 0 & 2 & 0 & -4 & -2 & -2 & 0 & -1 & 1 & 1 & 2 & -2 & 1 & 1 & 0 & 0 \\ 
-1 & 1 & -2 & -2 & -1 & 2 & -1 & -2 & -4 & -2 & 1 & 0 & 0 & -1 & 1 & -1 & 1 & 1 & 0 & 0 \\ 
-1 & 1 & -2 & -2 & -1 & 2 & -1 & -2 & -2 & -4 & 0 & 1 & -1 & 0 & 1 & -1 & 0 & 2 & 0 & -1 \\ 
1 & -1 & 0 & 1 & 2 & 1 & 1 & 0 & 1 & 0 & -4 & 0 & 0 & 0 & 1 & 1 & -1 & -1 & 1 & -1 \\ 
-1 & -2 & 1 & 2 & 2 & -1 & 1 & -1 & 0 & 1 & 0 & -4 & 0 & -1 & 0 & 0 & 0 & -1 & 1 & 0 \\ 
-2 & -1 & 0 & 0 & 0 & -1 & -2 & 1 & 0 & -1 & 0 & 0 & -4 & -2 & -2 & 0 & 0 & 1 & 0 & 0 \\ 
-2 & -2 & 0 & 0 & 0 & -1 & -1 & 1 & -1 & 0 & 0 & -1 & -2 & -4 & -2 & 0 & 0 & 0 & 0 & 0 \\ 
-1 & -1 & 1 & 0 & -1 & -2 & -1 & 2 & 1 & 1 & 1 & 0 & -2 & -2 & -4 & 1 & 0 & 0 & 0 & 0 \\ 
-2 & 0 & -1 & 1 & -1 & 0 & 0 & -2 & -1 & -1 & 1 & 0 & 0 & 0 & 1 & -4 & 1 & 1 & -1 & 1 \\ 
1 & -1 & 0 & -1 & 0 & 0 & 1 & 1 & 1 & 0 & -1 & 0 & 0 & 0 & 0 & 1 & -4 & 0 & -1 & -1 \\ 
1 & -1 & 2 & 2 & 2 & -1 & 1 & 1 & 1 & 2 & -1 & -1 & 1 & 0 & 0 & 1 & 0 & -4 & 2 & 1 \\ 
-1 & 0 & -2 & -1 & -2 & 0 & 0 & 0 & 0 & 0 & 1 & 1 & 0 & 0 & 0 & -1 & -1 & 2 & -4 & 0 \\ 
0 & -1 & -1 & -1 & 0 & 1 & 1 & 0 & 0 & -1 & -1 & 0 & 0 & 0 & 0 & 1 & -1 & 1 & 0 & -4 
\end{pmatrix}}$}
\end{center}

\begin{center}
    \resizebox{\textwidth}{!}{$E_{21}:= \scriptscriptstyle{\begin{pmatrix}
    -4 & 2 & 2 & -1 & 2 & -2 & 2 & 1 & -1 & -2 & -2 & -2 & 1 & 0 & -2 & -2 & -1 & -2 & 2 & -2 & 2 \\ 
2 & -4 & -2 & 2 & -1 & 0 & 0 & 1 & 2 & 1 & 2 & 2 & 1 & -1 & 1 & 2 & 2 & 2 & -2 & 1 & -1 \\ 
2 & -2 & -4 & 2 & -2 & 1 & 0 & 1 & 0 & 0 & 2 & 1 & -1 & 1 & 0 & 1 & 0 & 0 & -1 & 2 & 0 \\ 
-1 & 2 & 2 & -4 & 0 & -1 & -1 & -1 & -1 & -1 & -2 & -1 & -1 & 1 & -1 & -2 & -1 & -1 & 1 & -2 & -1 \\ 
2 & -1 & -2 & 0 & -4 & 2 & -2 & 1 & -1 & 0 & 0 & 0 & -1 & 0 & 1 & 1 & -1 & 0 & -2 & 1 & -2 \\ 
-2 & 0 & 1 & -1 & 2 & -4 & 1 & 1 & 1 & -1 & -1 & -1 & 1 & 0 & -1 & -1 & 0 & 0 & 1 & -1 & 1 \\ 
2 & 0 & 0 & -1 & -2 & 1 & -4 & 0 & -1 & 0 & 0 & 1 & 0 & -1 & 2 & 1 & -1 & 1 & -1 & 0 & -2 \\ 
1 & 1 & 1 & -1 & 1 & 1 & 0 & -4 & 0 & 2 & 1 & 1 & -1 & 1 & 0 & 0 & 1 & 0 & 0 & 0 & -1 \\ 
-1 & 2 & 0 & -1 & -1 & 1 & -1 & 0 & -4 & -2 & -1 & -1 & 0 & 1 & 0 & -1 & -2 & -2 & 1 & -1 & 0 \\ 
-2 & 1 & 0 & -1 & 0 & -1 & 0 & 2 & -2 & -4 & -2 & -1 & 0 & 0 & -1 & -1 & -2 & -2 & 2 & -2 & 1 \\ 
-2 & 2 & 2 & -2 & 0 & -1 & 0 & 1 & -1 & -2 & -4 & -2 & 0 & 0 & -1 & -1 & -2 & -1 & 1 & -2 & 0 \\ 
-2 & 2 & 1 & -1 & 0 & -1 & 1 & 1 & -1 & -1 & -2 & -4 & -1 & 1 & -1 & -2 & -2 & -1 & 1 & 0 & 1 \\ 
1 & 1 & -1 & -1 & -1 & 1 & 0 & -1 & 0 & 0 & 0 & -1 & -4 & 2 & -1 & -1 & -1 & -1 & 0 & 1 & 0 \\ 
0 & -1 & 1 & 1 & 0 & 0 & -1 & 1 & 1 & 0 & 0 & 1 & 2 & -4 & 2 & 1 & 0 & 1 & -1 & 0 & -1 \\ 
-2 & 1 & 0 & -1 & 1 & -1 & 2 & 0 & 0 & -1 & -1 & -1 & -1 & 2 & -4 & -1 & 0 & -2 & 1 & -1 & 2 \\ 
-2 & 2 & 1 & -2 & 1 & -1 & 1 & 0 & -1 & -1 & -1 & -2 & -1 & 1 & -1 & -4 & -1 & -1 & 2 & -1 & 1 \\ 
-1 & 2 & 0 & -1 & -1 & 0 & -1 & 1 & -2 & -2 & -2 & -2 & -1 & 0 & 0 & -1 & -4 & -1 & 1 & 0 & 0 \\ 
-2 & 2 & 0 & -1 & 0 & 0 & 1 & 0 & -2 & -2 & -1 & -1 & -1 & 1 & -2 & -1 & -1 & -4 & 1 & -1 & 1 \\ 
2 & -2 & -1 & 1 & -2 & 1 & -1 & 0 & 1 & 2 & 1 & 1 & 0 & -1 & 1 & 2 & 1 & 1 & -4 & 2 & -2 \\ 
-2 & 1 & 2 & -2 & 1 & -1 & 0 & 0 & -1 & -2 & -2 & 0 & 1 & 0 & -1 & -1 & 0 & -1 & 2 & -4 & 0 \\ 
2 & -1 & 0 & -1 & -2 & 1 & -2 & -1 & 0 & 1 & 0 & 1 & 0 & -1 & 2 & 1 & 0 & 1 & -2 & 0 & -4
\end{pmatrix}}$}
\end{center}

\section{Conjugacy classes of hearts of \texorpdfstring{$\Lambda$}{Λ}}\label{app: table of hearts}

Each entry in \Cref{tab: lovely table 45} corresponds to an isomorphism class of hearts $C$ of $\Lambda$ (\Cref{sec: stable sympl iso}). For each entry, we give:
\begin{itemize}
    \item the label of the associated lattice $C$. If $C$ embeds primitively into the Leech lattice $\bL$, the Id corresponds to the one of the associated stable symplectic sublattice of $\bL$ as given in \cite[Table 2]{H_hn_2016}. Otherwise, if $C$ is exceptional, the Id corresponds to the name of the associated lattice in \Cref{App: exceptional hearts}. In the case where $\Lambda$ has several isomorphism classes of primitive sublattices abstractly isometric to $C$ (\Cref{main th embeddings}), we add letters to distinguish each class;
    \item a description of the group $O^\#(C)$, its Id in the Small Group Library \cite{small_grp_lib}, or its order;
    \item the genus of the invariant sublattice $\Lambda^{O^\#(C)}=C^\perp_\Lambda$, following the convention of \cite[Chapter 15]{splg};
    \item whether the group $O^\#(C)$, which is saturated in $O^{+, \#}(\Lambda)$, is saturated in $O^+(\Lambda)$.
\end{itemize} 

\begin{remark}
     For an IHS manifold $X$ of deformation type $OG10$, the \textbf{wall divisors} of $X$ which are not prime exceptional correspond to vectors in $H^{1,1}(X, \mathbb{R})\cap H^2(X, \mathbb{Z})$ which are of square $-4$, or of square $-24$ and divisibility $3$ \cite[Proposition 5.4]{mongardi2020birational}. Excluding the pairs 
     \[\{(194\text{a}, 194\text{b}), (200\text{b}, 200\text{c}), (203\text{a}, 203\text{b}), (208\text{b},208\text{c})\}\]
     all entries in the dataset \cite{database} are uniquely determined by
    \begin{enumerate}
        \item the isometry class of stable symplectic sublattice $C\subseteq \Lambda$;
        \item the isometry class of the orthogonal complement of $C^\perp_\Lambda$;
        \item the number of vectors of square $-4$ in $C$;
        \item the number of vectors of square $-24$ in $C$ which have divisibility 3 in $\Lambda$.
    \end{enumerate}
\end{remark}

{
\small\centering\setlength{\tabcolsep}{3pt}
\renewcommand\arraystretch{1.5}
\begin{longtable}{cccc|cccc}

\caption{Hearts of $\Lambda$}\\
 Id&$O^\#(C)$&$g(C^\perp_\Lambda)$&Saturated&Id&$O^\#(C)$&$g(C^\perp_\Lambda)$&Saturated\\
\hline
\endfirsthead

\caption[]{Hearts of $\Lambda$ (continued)}\\
 Id&$O^\#(C)$&$g(\Lambda_{\textnormal{OG10}}^H)$&Saturated&Id&$O^\#(C)$&$g(\Lambda_{\textnormal{OG10}}^H)$&Saturated\\
\hline
\endhead

\rowcolor{lightgray!40!white}1 & $C_1$ & $ \II_{(3, 21)}3^{1} $ & true & 108b & $A_7$ & $ \II_{(3, 1)}3^{-2}5^{1}7^{-1} $ & true \\
2 & $C_2$ & $ \II_{(3, 13)}2^{8}3^{1} $ & true & 109a & $[1944, 3559]$ & $ \II_{(3, 1)}2^{2}_{6}3^{2} $ & false \\
\rowcolor{lightgray!40!white}3 & $C_2^2$ & $ \II_{(3, 9)}2^{-6}4^{-2}3^{1} $ & true & 109b & $[1944, 3559]$ & $ \II_{(3, 1)}2^{2}_{6}3^{-4} $ & true \\
4a & $C_3$ & $ \II_{(3, 9)}3^{-5} $ & false & 110 & $[1920, 240993]$ & $ \II_{(3, 1)}4^{-1}_{3}8^{1}_{1}3^{1}5^{-1} $ & true \\
\rowcolor{lightgray!40!white}4b & $C_3$ & $ \II_{(3, 9)}3^{7} $ & true & 111 & $[1344, 11686]$ & $ \II_{(3, 1)}4^{2}_{6}3^{1}7^{-1} $ & true \\
5 & $C_2$ & $ \II_{(3, 9)}2^{-12}_{0}3^{1} $ & true & 112 & $[1152, 155478]$ & $ \II_{(3, 1)}8^{-2}_{2}3^{2} $ & true \\
\rowcolor{lightgray!40!white}6 & $C_2^3$ & $ \II_{(3, 7)}2^{6}4^{2}_{6}3^{1} $ & true & 114 & $[972, 812]$ & $ \II_{(3, 1)}2^{-2}3^{-3} $ & false \\
7a & $S_3$ & $ \II_{(3, 7)}2^{-2}3^{4} $ & false & 116 & $[768, 1090135]$ & $ \II_{(3, 1)}2^{2}_{2}8^{-2}_{6}3^{1} $ & true \\
\rowcolor{lightgray!40!white}7b & $S_3$ & $ \II_{(3, 7)}2^{-2}3^{-6} $ & true & 118a & $S_6$ & $ \II_{(3, 1)}2^{-2}3^{-1}5^{1} $ & false \\
9 & $[4, 1]$ & $ \II_{(3, 7)}2^{-2}_{2}4^{4}3^{1} $ & true & 118b & $S_6$ & $ \II_{(3, 1)}2^{-2}3^{3}5^{1} $ & true \\
\rowcolor{lightgray!40!white}10 & $C_2^4$ & $ \II_{(3, 6)}2^{6}8^{1}_{7}3^{1} $ & true & 119 & $M_{10}$ & $ \II_{(3, 1)}2^{-1}_{5}4^{1}_{1}3^{2}5^{1} $ & true \\
13 & $D_8$ & $ \II_{(3, 6)}4^{5}_{7}3^{1} $ & true & 120 & $L_2(11)$ & $ \II_{(3, 1)}3^{1}11^{2} $ & true \\
\rowcolor{lightgray!40!white}15a& $A_{3,3}$ & $ \II_{(3, 5)}3^{-3}9^{1} $ & false & 121a & $[576, 8654]$ & $ \II_{(3, 1)}4^{1}_{1}8^{1}_{7}3^{-1} $ & false \\
15b & $A_{3,3}$ & $ \II_{(3, 5)}3^{5}9^{1} $ & true & 121b & $[576, 8654]$ & $ \II_{(3, 1)}4^{1}_{1}8^{1}_{7}3^{3} $ & true \\
\rowcolor{lightgray!40!white}17 & $C_2\times D_8$ & $ \II_{(3, 5)}2^{2}4^{4}_{0}3^{1} $ & true & 122 & $[500, 23]$ & $ \II_{(3, 1)}3^{1}5^{3} $ & true \\
18a & $D_{12}$ & $ \II_{(3, 5)}2^{4}3^{-3} $ & false & 123a & $[384, 20097]$ & $ \II_{(3, 1)}2^{-2}_{6}4^{-2} $ & false \\
\rowcolor{lightgray!40!white}18b & $D_{12}$ & $ \II_{(3, 5)}2^{4}3^{5} $ & true & 123b & $[384, 20097]$ & $ \II_{(3, 1)}2^{-2}_{6}4^{-2}3^{-2} $ & true \\
19a & $A_4$ & $ \II_{(3, 5)}2^{-2}4^{-2}3^{-1} $ & false & 124 & $[384, 18134]$ & $ \II_{(3, 1)}2^{-3}_{7}16^{1}_{1}3^{1} $ & true \\
\rowcolor{lightgray!40!white}19b & $A_4$ & $ \II_{(3, 5)}2^{-2}4^{-2}3^{3} $ & true & 128a & $\Gamma\textnormal{L}_2(\mathbb{F}_4)$ & $ \II_{(3, 1)}3^{1}5^{-2} $ & false \\
20 & $D_{10}$ & $ \II_{(3, 5)}3^{1}5^{4} $ & true & 128b & $\Gamma\textnormal{L}_2(\mathbb{F}_4)$ & $ \II_{(3, 1)}3^{-3}5^{-2} $ & true \\
\rowcolor{lightgray!40!white}22 & $S_3$ & $ \II_{(3, 5)}3^{-7} $ & true & 129 & $C_2\times L_3(2)$ & $ \II_{(3, 1)}2^{2}3^{1}7^{2} $ & true \\
23 & $C_2^2$ & $ \II_{(3, 5)}2^{-4}_{6}4^{4}_{6}3^{1} $ & true & 131 & $[192, 1494]$ & $ \II_{(3, 1)}2^{2}_{2}8^{-2}_{6}3^{1} $ & true \\
\rowcolor{lightgray!40!white}24 & $C_2^2\rtimes A_4$ & $ \II_{(3, 4)}2^{-4}8^{1}_{7}3^{2} $ & true & 133 & $C_2\times M_9$ & $ \II_{(3, 1)}2^{-4}_{0}3^{1}9^{1} $ & true \\
26 & $C_2^2\wr C_2$ & $ \II_{(3, 4)}2^{2}4^{2}_{2}8^{1}_{7}3^{1} $ & true & 134 & $\Gamma\textnormal{L}_1(\mathbb{F}_9)$ & $ \II_{(3, 1)}2^{1}_{7}4^{1}_{7}3^{2}9^{1} $ & true \\
\rowcolor{lightgray!40!white}28 & $2_+^{1+4}$ & $ \II_{(3, 4)}4^{5}_{1}3^{1} $ & true & 137a & $S_5$ & $ \II_{(3, 1)}2^{-4}_{6}5^{-1} $ & false \\
29a & $S_4$ & $ \II_{(3, 4)}4^{3}_{5}3^{-1} $ & false & 137b & $S_5$ & $ \II_{(3, 1)}2^{-4}_{6}3^{-2}5^{-1} $ & true \\
\rowcolor{lightgray!40!white}29b & $S_4$ & $ \II_{(3, 4)}4^{3}_{5}3^{3} $ & true & 138 & $[108, 17]$ & $ \II_{(3, 1)}3^{-1}9^{2} $ & true \\
30 & $D_8$ & $ \II_{(3, 4)}2^{4}_{2}4^{-2}8^{1}_{7}3^{1} $ & true & 141a & $N_{72}$ & $ \II_{(3, 1)}2^{4}_{2}9^{1} $ & false \\
\rowcolor{lightgray!40!white}31 & $Q_8$ & $ \II_{(3, 4)}2^{3}_{5}8^{-2}3^{1} $ & true & 141b & $N_{72}$ & $ \II_{(3, 1)}2^{4}_{2}3^{-2}9^{1} $ & true \\
35a & $[486, 249]$ & $ \II_{(3, 3)}3^{4} $ & false & 143 & $[64,257]$ & $ \II_{(3, 1)}2^{1}_{1}4^{-1}_{3}8^{2}3^{1} $ & true \\
\rowcolor{lightgray!40!white}35b & $[486, 249]$ & $ \II_{(3, 3)}3^{-6} $ & true & 144 & $A_5$ & $ \II_{(3, 1)}2^{-2}3^{-3} $ & true \\
37 & $C_4^2\rtimes A_4$ & $ \II_{(3, 3)}2^{-2}8^{-2}_{2}3^{1} $ & true & 148a & $C_2\times S_4$ & $ \II_{(3, 1)}2^{-2}_{2}4^{-1}_{3}8^{1}_{1} $ & false \\
\rowcolor{lightgray!40!white}38 & $C_2^2\rtimes S_4$ & $ \II_{(3, 3)}2^{2}4^{-1}_{5}8^{1}_{7}3^{2} $ & true & 148b & $C_2\times S_4$ & $ \II_{(3, 1)}2^{2}_{0}4^{1}_{1}8^{1}_{1}3^{-2} $ & true \\
39a & $A_{4,3}$ & $ \II_{(3, 3)}4^{-2}3^{-2} $ & false & 149 & $C_2\times F_5$ & $ \II_{(3, 1)}2^{-4}_{0}3^{1}5^{2} $ & true \\
\rowcolor{lightgray!40!white}39b & $A_{4,3}$ & $ \II_{(3, 3)}4^{-2}3^{4} $ & true & 150 & $[36, 13]$ & $ \II_{(3, 1)}2^{-4}_{0}3^{-3} $ & false \\
41 & $[64, 266]$ & $ \II_{(3, 3)}2^{-2}_{6}4^{4}3^{1} $ & true & 151a & $S_3^2$ & $ \II_{(3, 1)}2^{4}_{6}3^{2} $ & false \\
\rowcolor{lightgray!40!white}43 & $\Gamma_{25}a_1$ & $ \II_{(3, 3)}4^{3}_{1}8^{1}_{1}3^{1} $ & true & 151b & $S_3^2$ & $ \II_{(3, 1)}2^{4}_{6}3^{2} $ & false \\
44a & $A_5$ & $ \II_{(3, 3)}2^{-2}5^{-2} $ & false & 151c & $S_3^2$ & $ \II_{(3, 1)}2^{4}_{6}3^{-4} $ & true \\
\rowcolor{lightgray!40!white}44b & $A_5$ & $ \II_{(3, 3)}2^{-2}3^{-2}5^{-2} $ & true & 152 & $C_2\times QD_{16}$ & $ \II_{(3, 1)}2^{2}_{0}4^{-1}_{5}16^{1}_{7}3^{1} $ & true \\
45a & $C_2\times S_4$ & $ \II_{(3, 3)}2^{-2}4^{-2}_{2}3^{-1} $ & false & 154a & $C_2^2\times S_3$ & $ \II_{(3, 1)}2^{2}_{0}4^{-2}_{4}3^{1} $ & false \\
\rowcolor{lightgray!40!white}45b & $C_2\times S_4$ & $ \II_{(3, 3)}2^{-2}4^{-2}_{2}3^{3} $ & true & 154b & $C_2^2\times S_3$ & $ \II_{(3, 1)}2^{2}_{0}4^{-2}_{4}3^{-3} $ & true \\
46a & $C_3^2\rtimes C_4$ & $ \II_{(3, 3)}2^{-2}_{2}3^{-1}9^{1} $ & false & 157a & $D_{24}$ & $ \II_{(3, 1)}2^{-2}_{2}4^{-2}3^{2} $ & false \\
\rowcolor{lightgray!40!white}46b & $C_3^2\rtimes C_4$ & $ \II_{(3, 3)}2^{-2}_{2}3^{3}9^{1} $ & true & 157b & $D_{24}$ & $ \II_{(3, 1)}2^{-2}_{2}4^{-2}3^{-4} $ & true \\
47a & $S_3^2$ & $ \II_{(3, 3)}2^{-2}3^{2}9^{1} $ & false & 161 & $D_{12}$ & $ \II_{(3, 1)}2^{-4}_{0}3^{-3} $ & true \\
\rowcolor{lightgray!40!white}47b & $S_3^2$ & $ \II_{(3, 3)}2^{-2}3^{2}9^{1} $ & false & 163a & $PSU(4,3)$ & $ \II_{(3, 0)}4^{1}_{1}3^{-1} $ & false \\
47c & $S_3^2$ & $ \II_{(3, 3)}2^{-2}3^{-4}9^{1} $ & true & 163b & $PSU(4,3)$ & $ \II_{(3, 0)}4^{1}_{1}3^{3} $ &  true \\
\rowcolor{lightgray!40!white}48 & $[32, 34]$ & $ \II_{(3, 3)}2^{-2}_{6}4^{4}3^{1} $ & true & 165a & $M_{22}$ & $ \II_{(3, 0)}4^{-1}_{3}3^{1}11^{-1} $ & true \\
52 & $C_7\rtimes C_3$ & $ \II_{(3, 3)}3^{1}7^{-3} $ & true & 165b & $M_{22}$ & $ \II_{(3, 0)}4^{-1}_{3}3^{1}11^{-1} $ & true \\
\rowcolor{lightgray!40!white}53 & $F_5$ & $ \II_{(3, 3)}2^{2}_{6}3^{1}5^{3} $ & true & 167a &  $PSU(3,5)$ & $ \II_{(3, 0)}2^{-1}_{5}3^{1}5^{-2} $ & true \\
55 & $QD_{16}$ & $ \II_{(3, 3)}2^{-1}_{5}4^{-1}_{5}8^{-2}3^{1} $ & true & 167b & $PSU(3,5)$ & $ \II_{(3, 0)}2^{-1}_{5}3^{1}5^{-2} $ & true \\
\rowcolor{lightgray!40!white}56 & $A_4$ & $ \II_{(3, 3)}2^{2}_{2}4^{-4}3^{1} $ & true & 169 & $\#58320$  & $ \II_{(3, 0)}2^{1}_{7}3^{2}9^{-1} $ & true \\
58 & $C_2^3$ & $ \II_{(3, 3)}2^{-2}_{4}4^{-3}_{3}8^{-1}_{3}3^{1} $ & true & 170 & $\#40320$  & $ \II_{(3, 0)}4^{-1}_{5}3^{2}7^{1} $ & true \\
\rowcolor{lightgray!40!white}59 & $C_2\times C_4$ & $ \II_{(3, 3)}2^{-4}_{0}8^{2}_{6}3^{1} $ & true & 171a & $\#40320$ & $ \II_{(3, 0)}2^{-3}_{7}3^{1}7^{1} $ & true \\
60a & $S_3$ & $ \II_{(3, 3)}2^{-6}_{6}3^{1} $ & false & 171b & $\#40320$  & $ \II_{(3, 0)}2^{-3}_{7}3^{1}7^{1} $ & true \\
\rowcolor{lightgray!40!white}60b & $S_3$ & $ \II_{(3, 3)}2^{-6}_{6}3^{-3} $ & true & 172 & $\#40320$  & $ \II_{(3, 0)}8^{1}_{7}3^{1}7^{-1} $ & true \\
61a & $S_3$ & $ \II_{(3, 3)}2^{-6}_{2}3^{3} $ & false & 175a &$A_8$& $ \II_{(3, 0)}4^{1}_{7}5^{1} $ & false \\
\rowcolor{lightgray!40!white}61b & $S_3$ & $ \II_{(3, 3)}2^{-6}_{2}3^{-5} $ & true & 175b & $A_8$ & $ \II_{(3, 0)}4^{1}_{7}3^{-2}5^{1} $ & true \\
63a & $C_6$ & $ \II_{(3, 3)}2^{6}_{0}3^{-2} $ & false & 178a & $\#11520$  & $ \II_{(3, 0)}2^{2}_{2}8^{1}_{1} $ & false \\
\rowcolor{lightgray!40!white}63b & $C_6$ & $ \II_{(3, 3)}2^{6}_{0}3^{4} $ & true & 178b & $\#11520$  & $ \II_{(3, 0)}2^{-2}_{6}8^{-1}_{5}3^{-2} $ & true \\
68a & $[972, 776]$ & $ \II_{(3, 2)}2^{1}_{7}3^{3} $ & false & 180 & $\#10752$  & $ \II_{(3, 0)}2^{2}_{2}16^{-1}_{3}3^{1} $ & true \\
\rowcolor{lightgray!40!white}68b & $[972, 776]$ & $ \II_{(3, 2)}2^{1}_{7}3^{-5} $ & true & 182 & $M_{11}$ & $ \II_{(3, 0)}2^{-1}_{5}3^{2}11^{1} $ & true \\
69 & $M_{20}$ & $ \II_{(3, 2)}2^{-2}8^{1}_{7}3^{1}5^{-1} $ & true & 183 & $\#5760$  & $ \II_{(3, 0)}8^{1}_{7}3^{2}5^{-1} $ & true \\
\rowcolor{lightgray!40!white}70 & $F_{384}$ & $ \II_{(3, 2)}4^{-1}_{5}8^{-2}_{2}3^{1} $ & true & 184 & $\#4608$  & $ \II_{(3, 0)}2^{1}_{1}8^{-2}3^{1} $ & true \\
72a & $A_6$ & $ \II_{(3, 2)}4^{-1}_{3}3^{-1}5^{1} $ & false & 186 & $\#3888$ & $ \II_{(3, 0)}4^{1}_{7}3^{2} $ & false \\
\rowcolor{lightgray!40!white}72b & $A_6$ & $ \II_{(3, 2)}4^{-1}_{3}3^{3}5^{1} $ & true & 187a & $\#3888$ & $ \II_{(3, 0)}2^{-3}_{1}3^{1} $ & false \\
73a & $A_{4,4}$ & $ \II_{(3, 2)}2^{2}8^{1}_{7}3^{-1} $ & false & 187b & $\#3888$ & $ \II_{(3, 0)}2^{-3}_{1}3^{-3} $ & true \\
\rowcolor{lightgray!40!white}73b & $A_{4,4}$ & $ \II_{(3, 2)}2^{2}8^{1}_{7}3^{3} $ & true & 191 & $[1944, 3536]$ & $ \II_{(3, 0)}2^{3}_{3}3^{-2} $ & false \\
74 & $H_{192}$ & $ \II_{(3, 2)}4^{-2}_{6}8^{1}_{7}3^{2} $ & true & 193a & $[1440, 5844]$ & $ \II_{(3, 0)}2^{2}_{0}4^{1}_{1}3^{1}5^{1} $ & true \\
\rowcolor{lightgray!40!white}76a & $T_{192}$ & $ \II_{(3, 2)}4^{-3}_{1} $ & false & 193b & $[1440, 5844]$ & $ \II_{(3, 0)}2^{2}_{0}4^{1}_{1}3^{1}5^{1} $ & true \\
76b & $T_{192}$ & $ \II_{(3, 2)}4^{-3}_{1}3^{-2} $ & true & 194a & $[1440, 5841]$ & $ \II_{(3, 0)}2^{-3}_{7}3^{2}5^{1} $ & true \\
\rowcolor{lightgray!40!white}77 & $L_2(7)$ & $ \II_{(3, 2)}4^{1}_{7}3^{1}7^{2} $ & true & 194b & $[1440, 5841]$ & $ \II_{(3, 0)}2^{-3}_{7}3^{2}5^{1} $ & true \\
80 & $[128, 1759]$ & $ \II_{(3, 2)}2^{2}_{2}4^{-2}8^{1}_{1}3^{1} $ & true & 197a & $[768, 1086051]$ & $ \II_{(3, 0)}2^{1}_{7}4^{1}_{1}16^{-1}_{5}3^{1} $ & true \\
\rowcolor{lightgray!40!white}82a & $S_5$ & $ \II_{(3, 2)}4^{-1}_{5}5^{-2} $ & false & 197b & $[768, 1086051]$ & $ \II_{(3, 0)}2^{1}_{7}4^{1}_{1}16^{-1}_{5}3^{1} $ & true \\
82b & $S_5$ & $ \II_{(3, 2)}4^{-1}_{5}3^{-2}5^{-2} $ & true & 200a & $S_6$ & $ \II_{(3, 0)}2^{2}_{6}4^{-1}_{3}3^{-1} $ & false \\
\rowcolor{lightgray!40!white}84 & $M_9$ & $ \II_{(3, 2)}2^{3}_{5}3^{2}9^{1} $ & true & 200b & $S_6$ & $ \II_{(3, 0)}2^{2}_{6}4^{-1}_{3}3^{3} $ & true \\
85a & $N_{72}$ & $ \II_{(3, 2)}4^{1}_{7}3^{-1}9^{1} $ & false & 200c & $S_6$ & $ \II_{(3, 0)}2^{2}_{6}4^{-1}_{3}3^{3} $ & true \\
\rowcolor{lightgray!40!white}85b & $N_{72}$ & $ \II_{(3, 2)}4^{1}_{7}3^{3}9^{1} $ & true & 201 & $\textnormal{AGL}_2(\mathbb{F}_3)$ & $ \II_{(3, 0)}2^{-1}_{5}3^{-1}9^{1} $ & true \\
87 & $T_{48}$ & $ \II_{(3, 2)}2^{1}_{1}8^{-2}3^{2} $ & true &203a & $C_2\times\textnormal{A}\Gamma\textnormal{L}_1(\mathbb{F}_9)$ & $ \II_{(3, 0)}2^{-2}_{2}4^{1}_{7}3^{1}9^{1} $ & true \\
\rowcolor{lightgray!40!white}89 & $\Aut(D_8)$ & $ \II_{(3, 2)}2^{-3}_{7}8^{2}3^{1} $ & true & 203b & $C_2\times\textnormal{A}\Gamma\textnormal{L}_1(\mathbb{F}_9)$ & $ \II_{(3, 0)}2^{-2}_{2}4^{1}_{7}3^{1}9^{1} $ & true \\
92a & $S_4$ & $ \II_{(3, 2)}2^{4}_{2}8^{1}_{7} $ & false & 205a & $C_2\times S_5$ & $ \II_{(3, 0)}2^{2}_{0}4^{-1}_{3}5^{-1} $ & false \\
\rowcolor{lightgray!40!white}92b & $S_4$ & $ \II_{(3, 2)}2^{4}_{2}8^{1}_{7}3^{-2} $ & true & 205b & $C_2\times S_5$ & $ \II_{(3, 0)}2^{2}_{0}4^{-1}_{3}3^{-2}5^{-1} $ & true \\
94 & $C_2\times Q_8$ & $ \II_{(3, 2)}2^{-4}_{0}16^{1}_{7}3^{1} $ & true & 205c & $C_2\times S_5$ & $ \II_{(3, 0)}2^{2}_{0}4^{-1}_{3}3^{-2}5^{-1} $ & true \\
\rowcolor{lightgray!40!white}95 & $C_2\times D_8$ & $ \II_{(3, 2)}2^{2}_{0}4^{-1}_{3}8^{2}_{0}3^{1} $ & true & 207 & $\textnormal{A}\Gamma\textnormal{L}_1(\mathbb{F}_8)$ & $ \II_{(3, 0)}2^{1}_{1}8^{-2}3^{1} $ & true \\
97a & $D_{12}$ & $ \II_{(3, 2)}2^{4}_{0}4^{-1}_{3}3^{1} $ & false & 208a & $S_3\times S_4$ & $ \II_{(3, 0)}2^{2}_{6}8^{-1}_{3}3^{1} $ & false \\
\rowcolor{lightgray!40!white}97b & $D_{12}$ & $ \II_{(3, 2)}2^{4}_{0}4^{-1}_{3}3^{-3} $ & true & 208b & $S_3\times S_4$ & $ \II_{(3, 0)}2^{2}_{6}8^{-1}_{3}3^{-3} $ & true \\
98a & $D_{12}$ & $ \II_{(3, 2)}2^{-4}_{2}4^{-1}_{3}3^{-2} $ & false & 208c & $S_3\times S_4$ & $ \II_{(3, 0)}2^{2}_{6}8^{-1}_{3}3^{-3} $ & true \\
\rowcolor{lightgray!40!white}98b & $D_{12}$ & $ \II_{(3, 2)}2^{-4}_{2}4^{-1}_{3}3^{4} $ & true & 211 & $S_5$ & $ \II_{(3, 0)}2^{3}_{3}3^{-2} $ & true \\
101a & $C_3^4\rtimes A_6$ & $ \II_{(3, 1)}3^{-1}9^{-1} $ & false & 212 & $C_2\times T_{48}$ & $ \II_{(3, 0)}2^{2}_{0}16^{1}_{7}3^{2} $ & true \\
\rowcolor{lightgray!40!white}101b & $C_3^4\rtimes A_6$ & $ \II_{(3, 1)}3^{3}9^{-1} $ & true & 214a & $C_2\times S_3^2$ & $ \II_{(3, 0)}2^{-2}_{6}4^{-1}_{5}3^{2} $ & false \\
102 & $L_3(4)$ & $ \II_{(3, 1)}2^{-2}3^{2}7^{1} $ & true & 214b & $C_2\times S_3^2$ & $ \II_{(3, 0)}2^{-2}_{6}4^{-1}_{5}3^{2} $ & false \\
\rowcolor{lightgray!40!white}106a & $C_2^4\rtimes A_6$ & $ \II_{(3, 1)}4^{1}_{1}8^{-1}_{5} $ & false & 220a & $F_7$ & $ \II_{(3, 0)}2^{-3}_{1}3^{1}7^{-2} $ & true \\
106b & $C_2^4\rtimes A_6$ & $ \II_{(3, 1)}4^{1}_{1}8^{-1}_{5}3^{-2} $ & true & 220b & $F_7$ & $ \II_{(3, 0)}2^{-3}_{1}3^{1}7^{-2} $ & true \\
\rowcolor{lightgray!40!white}108a & $A_7$ & $ \II_{(3, 1)}5^{1}7^{-1} $ & false &  &  &  &  \\
\hline
E18a&$C_3^2$&$\II_{(3,2)}3^{-6}$&false&E20c&[108, 40]&$\II_{(3,1)}3^39^{-1}$&false\\
\rowcolor{lightgray!40!white}E18b&$C_3^2$&$\II_{(3,2)}3^{-6}$&true&E21a&$C_3\times S_3$&$\II_{(3,0)}2^{-3}_13^{-3}$&true\\
E20a&[108, 40]&$\II_{(3,1)}3^39^{-1}$&false&E21b&$C_3\times S_3$&$\II_{(3,0)}2^{-3}_13^{-3}$&false\\
\rowcolor{lightgray!40!white}E20b&[108, 40]&$\II_{(3,1)}3^39^{-1}$&true& &  &  &  \\
\hline

\label{tab: lovely table 45}
\end{longtable}}

\section{Extension approach --- algorithms}\label{sec: algorithms}In this section, we explain how to simplify the computations of representatives for the double cosets in \Cref{th:classth}. 
Let us observe the following:
\begin{proposition}\label{prop: type hearts}
    Let $C$ be a heart, and let $(F, a)$ be a head of $C$. One of the following two holds:
    \begin{itemize}
        \item either $D_C$ embeds into $D_F$, as abelian groups; 
        \item or $D_F$ embeds into $D_C$, as abelian groups.
    \end{itemize}
\end{proposition}
\begin{proof}
    Let us see $C\subseteq \Lambda$ as the image of a primitive embedding $i\colon C\hookrightarrow \Lambda$. According to \cite[Proposition 1.15.1]{nikulin}, the primitive embedding $i$ determines an isomorphism between a subgroup $I_C \leq D_C$ and a subgroup $I_\Lambda \leq D_\Lambda$. Now, since $D_\Lambda\cong \mathbb{Z}/3\mathbb{Z}$ as abelian group, then either $I_C$ is the trivial group, or $I_C\cong D_\Lambda$. In the former case, \cite[Proposition 1.15.1]{nikulin} tells us that $D_C$ is the glue domain of $i$, and thus $D_C$ is identified with a subgroup of $D_F$. Similar arguments apply in the other case by exchanging the role of $F$ and $C$.
\end{proof}

Given a heart $C$, and given a head $(F, a)$ of $C$, we can easily decide in which case of \Cref{prop: type hearts} the pair $(C, F)$ fits, by comparing the determinant of $C$ and $F$. In particular, we can already conclude the following.

\begin{corollary}\label{small lemma}
    Let $C$ be a heart and let $(F, a)$ be a head of $C$. If $\det(C)$ divides $\det(F)$, then the image of $\{\id_F\}\times O^\#(C)$ along the primitive extension $F\oplus C\subseteq  \Lambda$ is saturated in $O^+(\Lambda)$.
\end{corollary}
\begin{proof}
    According to \Cref{prop: type hearts}, we know that the glue map associated to $F\oplus C\subseteq \Lambda$ identify $D_C$ with a proper subgroup of $D_F$. Thus the result follows directly from \Cref{lem: stably sat,lemma saturation}.
\end{proof}

\begin{remark}
    According to \Cref{small lemma}, if $C$ is a heart and $(F, a)$ is a head of $C$ with $\det(C)\mid \det(F)$, then $a$ must have order 2. In the case where $\det(F)$ divides $\det(C)$, we cannot conclude similarly (see \Cref{possible saturations})
\end{remark}

In what follows, we make \Cref{th:classth} more explicit by separating these two cases from \Cref{prop: type hearts}. 
Indeed, we prove the following lemma.

\begin{lemma}\label{lem:action for spine}
    Let $C$ be a heart and let $(F, a)$ be a head of $C$. Let $b\in O(C)$ and let $\gamma$ be an $(a, b)$-equivariant glue map. Then, $a\oplus b\in O(\Lambda_\gamma)$ is non-stable if and only if
    \begin{itemize}
        \item $\det(C)\mid \det(F)$ and $D_a$ restricts to negative identity on the orthogonal complement of the glue domain of $F\hookrightarrow \Lambda_{\gamma}$;
        \item $\det(F)\mid \det(C)$ and $D_b$ restricts to negative identity on the orthogonal complement of the glue domain of $C\hookrightarrow \Lambda_\gamma$.
    \end{itemize}
\end{lemma}
\begin{proof}
  Let $\gamma$ be an $(a,b)$-equivariant glue map and let $h := a\oplus b\in O(\Lambda_\gamma)$. According to \cite[Proposition 1.15.1]{nikulin}, $D_{\Lambda_\gamma}$ is isometric to $\Gamma^\perp/\Gamma$ where $\Gamma$ is the graph of $\gamma$ in $D_F\oplus D_C$, and the action of $h$ on $D_{\Lambda_\gamma}$ coincides with the one of $a\oplus b$ on $\Gamma^\perp/\Gamma$. Now
  \begin{itemize}
      \item If $\det(C)\mid \det(F)$, we write $D_F = S\oplus T$ where $S\cong D_{\Lambda}$ and $T := S^{\perp}\cong D_C(-1)$ is the glue domain of $F\hookrightarrow \Lambda_\gamma$. In that case, the action of $a\oplus b$ on $\Gamma^\perp/\Gamma$ is given by $(D_a)_{\mid S}$;
      \item If $\det(F)\mid \det(C)$, we write $D_C = S\oplus T$ where $S\cong D_{\Lambda}$ and $T := S^{\perp}\cong D_F(-1)$ is the glue domain of $C\hookrightarrow \Lambda_\gamma$. In that case, the action of $a\oplus b$ on $\Gamma^{\perp}/\Gamma$ is given by $(D_b)_{\mid S}$.\qedhere
  \end{itemize}
\end{proof}

From a computational point of view, \Cref{lem:action for spine} together with \Cref{th:BHP} allows us to decide at the level of hearts, heads and their companions which equivariant gluings will not give rise to spines. This is featured in \Cref{alg:one,alg:two} to compute only the relevant equivariant primitive extensions (for our purpose). 

\begin{algorithm}
\setstretch{0.95}
\caption{Simplied extensions I}\label{alg:one}
\DontPrintSemicolon
\Input{A heart $C$ and a head $(F,a)$ of $C$ such that $\det(C)\mid \det(F)$}
\Output{Representatives of conjugacy classes of pairs $(\Lambda', H)$ where $\Lambda'\cong \Lambda$ and $H\leq O^+(\Lambda')$ is a symplectic finite subgroup such that $H^\#\leq O^{+, \#}(\Lambda')$ is saturated, $\#\overline{H}=2$, the heart of $H$ is isometric to $C$ and its head is isomorphic to $(F, a)$.}
Initialise the empty list $E= []$.\\
    Let $\mathcal{H}_F$ be a set of representatives of isometry classes in $\left\{H_F\leq D_F\mid H_F\cong D_\Lambda\right\}/\overline{O(F, a)}$\label{line:14}\\
    \For{$[S]\in\mathcal{H}_F$\label{line:15}}
      {
        $T \leftarrow S^{\perp}$.\label{line:16}\\
        \If{$D_aT\neq T$ \label{line:17}}
          {
            Discard $[S]$ and continue the for loop with the next representative.\label{line:18}
          }
        \If{$(D_a)_{\mid S}\neq -\textnormal{id}_S$\label{line:19}}
          {
            Discard $[S]$ and continue the for loop with the next representative.\label{line:110}
          }
        Let $\gamma\colon T\to D_C$ be a glue map.\label{line:111}\\
        $S^F_T \leftarrow \Stab_{\overline{O(F, a)}}(T)$.\\
        $S^T_T \leftarrow \im(S^F_T\to O(T))$.\\
        $S_T^{\gamma} \leftarrow \gamma\, S_T^T\,\gamma^{-1}$.\label{line:112}\\
        \For{$[g]\in \overline{O(C)}\diagdown O(D_C)\diagup S_T^{\gamma}$\label{line:113}}
          {
            $\gamma_g \leftarrow g\circ\gamma$.\label{line:114}\\
            $\bar{b} \leftarrow \gamma_gD_a\gamma_g^{-1}$.\label{line:115}\\
            \If{$\bar{b} \notin \overline{O(C)}$\label{line:116}}
              {
                Discard $[g]$ and continue the for loop with the next double coset.\label{line:117}
              }
            Let $b\in O(C)$ such that $D_b = \bar{b}$.\label{line:118}\\
            Let $\Lambda'$ be the overlattice of the glue map $\gamma_g$.\label{line:119}\\
            $h \leftarrow a\oplus b\in O(\Lambda')$.\label{line:120}\\
            $H \leftarrow \langle O^\#(C), \,h \rangle$.\label{line:124}\\
            \If{$\Lambda'_H\cap\markman(\Lambda')\neq\varnothing$ \label{line:125}}
              {
                Discard $[g]$ and continue the for loop with the next double coset.\label{line:126}
              }
            Append $(\Lambda', H)$ to $E$.\label{line:127}
          }
      }
  Return $E$.
\end{algorithm}

\begin{proposition}\label{propo:algo1}
    For any heart $C$ and any head $(F,a)$ of $C$ such that $\det(C)\mid \det(F)$, \Cref{alg:one} returns the correct output.
\end{proposition}

\begin{proof}
  Since $\det(C)\mid \det(F)$, \Cref{prop: type hearts} tells us that for any primitive extensions $F\oplus C\subseteq \Lambda'$ with $\Lambda'\cong \Lambda$, then the glue kernel of $C\hookrightarrow \Lambda'$ is the discriminant group $D_C$ of $C$.
  Let $(\Lambda', H)\in E$ be in the output of the algorithm. Since $C$ and $F_a$ are negative definite, the condition in \Cref{line:125} ensures that $H$ is symplectic (\Cref{lem: crit for bir symp eff}). Since it is generated by $O^\#(C)$ and $h$ where $h$ lies in $O^+(\Lambda')\setminus O^\#(\Lambda')$, we have that $H^\# = O^\#(C)$ is saturated in $O^{+, \#}(\Lambda')$. Note that here, we view $O^\#(C)$ as a saturated subgroup of $O^{+, \#}(\Lambda')$ after extending with the identity on $F$. Moreover \Cref{line:19}, together with \Cref{lem:action for spine}, ensures that $D_h$ acts by negative identity on $D_{\Lambda'}\cong S$. Therefore, together with the conditions in \Cref{line:17,line:116}, we know that $\gamma$ is a spine between $C$ and $(F, a)$ and $b$ is companion to $a$. Moreover, \Cref{th:BHP} tells us that the definition of $H$ does not depend on the choice of $b$ in \Cref{line:118}, and the conjugacy class of $(\Lambda', H)$ is independent on the choice of $(F, a)$ in its isomorphism class. Finally, by definition of $H$, we know that the heart of $H$ is isometric to $C$, and $\overline{H}\leq D_{\Lambda'}$ is nontrivial by definition of $h$ in \Cref{line:120}.
  
  Now, suppose that $H\leq O^+(\Lambda)$ is a symplectic finite subgroup, with $H^\#$ saturated in $O^{+, \#}(\Lambda)$, with $\overline{H}\leq D_\Lambda$ nontrivial, with heart isometric to $C$ and with head isomorphic to $(F, a)$. Let $h\in H$ be such that $H/O^\#(C)$ is generated by $hO^\#(C)$, and let $b$ be the restriction of $h$ to $\Lambda_{H^\#}\cong C$. Then, by \Cref{th:BHP,th:classth}, up to the choice of a representative in the class of the glue domain for $F\hookrightarrow \Lambda$ in \Cref{line:15}, the choice of a representative in the double coset of the associated spines in \Cref{line:113} and the choice of a suitable companion $b'$ of $a$ in \Cref{line:118} with $D_{b'} = D_b$, we have that there exists $(\Lambda', H')\in E$ which is conjugate to $(\Lambda, H)$. Note that the double cosets in \Cref{line:113} and in \Cref{th:classth} are in bijection.
\end{proof}

\begin{algorithm}
\setstretch{0.95}
\caption{Simplied extensions II}\label{alg:two}
\DontPrintSemicolon
\Input{A heart $C$ and a head $(F,a)$ of $C$ such that $\det(F)\mid \det(C)$}
\Output{Representatives of conjugacy classes of pairs $(\Lambda', H)$ where $\Lambda'\cong \Lambda$ and $H\leq O^+(\Lambda')$ is a symplectic finite subgroup such that $H^\#\leq O^{+, \#}(\Lambda')$ is saturated, $\#\overline{H}=2$, the heart of $H$ is isometric to $C$ and its head is isomorphic to $(F, a)$.}
Initialise the empty list $E= []$.\\
  Let $\mathcal{H}_C$ be a set of representatives of isometry classes in $\left\{H_C\leq D_C\mid H_C\cong D_\Lambda\right\}/\overline{O(C)}$.\\
    \For{$[S]\in\mathcal{H}_C$}
      {
        $T \leftarrow S^{\perp}$.\\
        $S^C_T \leftarrow \textnormal{Stab}_{\overline{O(C)}}(T)$.\label{line:27}\\
        $S^T_T \leftarrow \textnormal{im}(S^C_T\to O(T))$.\label{line:28}\\
        Let $\gamma\colon D_F\to T$ be a glue map.\\
         $S_F^{\gamma} \leftarrow \gamma \overline{O(F,a)}\gamma^{-1}$.\label{line:210}\\
    \For{$[g]\in S^T_T\diagdown O(T)\diagup S_F^{\gamma}$\label{line:211}}
      {
        $\gamma_g \leftarrow g\circ\gamma$.\label{line:212}\\
        $\tilde{b} \leftarrow \gamma_gD_a\gamma_g^{-1}$.\label{line:213}\\
        \If{$\tilde{b} \notin S^T_T$\label{line:214}}
          {
            Discard $[g]$ and continue the for loop with the next double coset.\label{line:215}
          }
        $\widehat{b}\leftarrow \tilde{b}\oplus (-\textnormal{id}_S)\in O(D_T)$.\label{line:216}\\
        \If{$\widehat{b}\notin S^C_T$\label{line:if}}
          {
            Discard $[g]$ and continue the for loop with the next double coset.\label{line:dis}
          }
            Let $b\in O(C)$ such that $D_b = \widehat{b}$.\label{line:218}\\
        Let $\Lambda'$ be the overlattice of the glue map $\gamma_g$.\label{line:219}\\
        $h \leftarrow a\oplus b\in O(\Lambda')$.\label{line:220}\\
        $H \leftarrow \langle O^\#(C),\, h \rangle$.\label{line:224}\\
        \If{$a \neq \textnormal{id}_F$ \textnormal{\textbf{and}} $\Lambda'_H\cap\markman(\Lambda')\neq\varnothing$\label{line:225}}
          {
            Discard $[g]$ and continue the for loop with the next double coset.\label{line:226}
          }
        Append $(\Lambda', H)$ to $E$.\label{line:227}
      }
  }
  Return $E$.
\end{algorithm}

\begin{proposition}\label{propo:algo2}
    For any heart $C$ and any head $(F,a)$ of $C$ such that $\det(F)\mid \det(C)$, \Cref{alg:two} returns the correct output.
\end{proposition}

\begin{proof}
  The proof is similar to the proof of \Cref{propo:algo1}. Note that the main difference is that we do not start with a fixed isometry of $C$, so the translation of the double cosets from \Cref{th:classth} to this context has to be adapted accordingly. 
  
  Let us note that if $a = \textnormal{id}_F$, it follows that $H$ is the saturation of $O^\#(C)$ in $O^+(\Lambda')$ as described in \Cref{possible saturations}. By definition of the saturation, $\Lambda'_H = C$ holds and this implies that $\Lambda'_H\cap\markman(\Lambda')$ is necessarily empty.
\end{proof}

\begin{remark}
    The isometry $b$ at Lines \ref{line:118} and \ref{line:218} of \Cref{alg:one,alg:two} respectively is computable in our setting since for us the respective lattices $C$ are definite: one can therefore effectively compute the discriminant representation $O(C)\to O(D_C)$ and for each element in $\overline{O(C)}$, one can compute a preimage. In the case where $F$ is indefinite, then we do not compute $\overline{O(F, a)}$ through $O(F, a)$ which might be in general infinite. However, the enumeration process described by Brandhorst and Hofmann features the computation of $\overline{O(F, a)}$ by induction on gluing stabilisers along equivariant primitive extensions \cite[Algorithm 2]{bh22}.
\end{remark}

\begin{remark}
    Each entry of \Cref{tab: lovely table 45} determines a pair $(C, F)$ where $C$ is a heart and $F$ is its orthogonal complement in $\Lambda$. In the cases where $\det(C)\mid \det(F)$ and $C$ has rank 21, then $F$ is positive definite and it admits no nontrivial isometries with negative definite coinvariant sublattice. All the other cases are uniquely determined by $C$ and $F$, up to isometry, except for the pair of cases 47a and 47b. In those cases, $F\in \II_{(3,3)}2^{-2}3^29^1$ and we have that $\det(F)\mid \det(C)$. However, one can actually show that in this situation the set $\mathcal{H}_C$, as defined in \Cref{alg:two}, has actually cardinality 2 and which is why we obtain these two non-isomorphic primitive sublattices of $\Lambda$.
    The upshot is the following. In our particular setting, for each pair $(\Lambda', H)$ in output of \Cref{alg:one,alg:two}, it is effectively possible to determine to which entry of \Cref{tab: lovely table 45} the stable sublattice $\Lambda'_{H^\#}\subseteq \Lambda'\cong\Lambda$ is isomorphic.
\end{remark}

\end{document}